\documentclass{article}
\usepackage[cm]{fullpage}
\usepackage{graphicx}
\usepackage{amsmath, amssymb, amsfonts, amsthm}
\usepackage{mathdots}
\usepackage{subcaption}
\usepackage{xcolor}
\usepackage{enumitem}
\usepackage[numbers]{natbib}
\usepackage{algorithm}
\usepackage{algpseudocode}

\theoremstyle{definition}
\newtheorem{definition}{Definition}
\theoremstyle{plain}
\newtheorem{theorem}{Theorem}
\newtheorem{lemma}{Lemma}
\newtheorem{proposition}{Proposition}
\newtheorem{corollary}{Corollary}

\usepackage[hidelinks]{hyperref}
\usepackage[capitalise]{cleveref}

\makeatletter
\renewcommand*\env@matrix[1][*\c@MaxMatrixCols c]{
  \hskip -\arraycolsep
  \let\@ifnextchar\new@ifnextchar
  \array{#1}}
\makeatother

\makeatletter
\if@titlepage\else
  \renewenvironment{abstract}{%
    \small
    \paragraph{\abstractname}
  }{\par\bigskip}
\fi
\makeatother

\allowdisplaybreaks

\title{H-invariance theory:\\ A complete characterization of minimax optimal fixed-point algorithms}
\author{TaeHo Yoon\textsuperscript{$\ast$} \\ \texttt{tyoon7@jhu.edu} \and Ernest K.\ Ryu\textsuperscript{$\dagger$} \\ \texttt{eryu@math.ucla.edu} \and  Benjamin Grimmer\textsuperscript{$\ast$}\\ \texttt{grimmer@jhu.edu}}
\date{\textsuperscript{$\ast$}Department of Applied Mathematics and Statistics, Johns Hopkins University \smallskip \\ \textsuperscript{$\dagger$}Department of Mathematics, University of California, Los Angeles}

\newcommand{\JA}{\opJ_{\opA}}

\DeclareMathOperator{\spann}{span}

\newcommand{\cut}[1]{{}}

\newcommand{\vdelta}{{\boldsymbol{\delta}}}

\newcommand{\ve}{{\mathbf{e}}}

\newcommand{\vg}{{\mathbf{g}}}

\newcommand{\vx}{{\mathbf{x}}}
\newcommand{\vy}{{\mathbf{y}}}

\newcommand{\vA}{{\mathbf{A}}}
\newcommand{\vB}{{\mathbf{B}}}
\newcommand{\vC}{{\mathbf{C}}}
\newcommand{\vD}{{\mathbf{D}}}
\newcommand{\vE}{{\mathbf{E}}}

\newcommand{\vG}{{\mathbf{G}}}

\newcommand{\vL}{{\mathbf{L}}}
\newcommand{\vM}{{\mathbf{M}}}

\newcommand{\vX}{{\mathbf{X}}}
\newcommand{\vY}{{\mathbf{Y}}}

\newcommand{\cA}{{\mathcal{A}}}

\newcommand{\cH}{{\mathcal{H}}}

\newcommand{\cM}{{\mathcal{M}}}

\newcommand{\cO}{{\mathcal{O}}}
\newcommand{\cP}{{\mathcal{P}}}

\usepackage[bb=boondox]{mathalfa}

\usepackage{xparse}
\DeclareFontFamily{U}{ntxmia}{}
\DeclareFontShape{U}{ntxmia}{m}{it}{<-> ntxmia }{}
\DeclareFontShape{U}{ntxmia}{b}{it}{<-> ntxbmia }{}
\DeclareSymbolFont{lettersA}{U}{ntxmia}{m}{it}
\SetSymbolFont{lettersA}{bold}{U}{ntxmia}{b}{it}

\ExplSyntaxOn
\NewDocumentCommand{\varmathbb}{m}
 {
  \tl_map_inline:nn { #1 }
   {
    \use:c { varbb##1 }
   }
 }
\tl_map_inline:nn { ABCDEFGHIJKLMNOPQRSTUVWXYZ }
 {
  \exp_args:Nc \DeclareMathSymbol{varbb#1}{\mathord}{lettersA}{\int_eval:n { `#1+67 }}
 }
\exp_args:Nc \DeclareMathSymbol{varbbk}{\mathord}{lettersA}{169}
\ExplSyntaxOff

\newcommand{\Tr}{{\mathrm{Tr}}} 

\newcommand*{\fix}{\mathrm{Fix}\,}

\newcommand{\proj}{\Pi}

\newcommand{\reals}{\mathbb{R}}

\newcommand{\opA}{{\varmathbb{A}}}

\newcommand{\opI}{{\varmathbb{I}}}
\newcommand{\opJ}{{\varmathbb{J}}}

\newcommand{\opT}{{\varmathbb{T}}}

\newcommand{\inprod}[2]{\left\langle #1,#2 \right\rangle}
\newcommand{\sqnorm}[1]{\left\| #1 \right\|^2}

\newcommand{\norm}[1]{\left\|#1\right\|}

\newcommand{\at}{\textsf{A}}
\newcommand{\T}{\textsf{T}}

\newcommand{\adj}{\operatorname{adj}}

\newcommand{\fA}{\mathfrak{A}}
\newcommand{\fP}{\mathfrak{P}}

\begin{document}

\maketitle

\begin{abstract}
For nonexpansive fixed-point problems, Halpern's method with optimal parameters, its so-called H-dual algorithm, and in fact, an infinite family of algorithms containing them, all exhibit the exact minimax optimal convergence rate.
In this work, we provide a characterization of the complete, exhaustive family of distinct algorithms using predetermined step-sizes, represented as lower triangular H-matrices, which attain the same optimal convergence rate. 
The characterization is based on polynomials in the entries of the H-matrix that we call \textit{H-invariants}, whose values stay constant over all optimal H-matrices, together with \textit{H-certificates}, the nonnegativity of which precisely specifies the region of optimality within the common level set of H-invariants.
The H-invariance theory we present offers a novel view of optimal acceleration in first-order optimization as a mathematical study of carefully selected invariants, certificates, and structures induced by them.

\vspace{.1cm}
\noindent
\textbf{Keywords\phantom{.} } Fixed-point problems $\cdot$ Monotone operators $\cdot$ Acceleration $\cdot$ Rates of convergence $\cdot$ Minimax optimality

\vspace{.1cm}
\noindent
\textbf{Mathematics Subject Classification\phantom{.} } 47H05 $\cdot$ 47H09 $\cdot$ 47H10 $\cdot$ 68Q25 $\cdot$ 90C47 $\cdot$ 90C60
\end{abstract}

\section{Introduction}

In optimization, finding a minimax optimal, or worst-case optimal, algorithm for a given problem class is a standard goal of optimizers.
More formally, this goal can be stated as the meta-problem 
\begin{align}
\label{eqn:minimax-optimality}
    \underset{\cA \in \mathfrak{A}}{\text{minimize}} \,\, \underset{\cP \in \mathfrak{P}}{\text{maximize}} \,\, \cM \left( \cA, \cP ; N \right)
\end{align}
where $\mathfrak{A}$ denotes a class of algorithms (e.g.\ first, second, or zeroth-order deterministic algorithms, stochastic algorithms using mini-batching, randomized algorithms, etc.), $\mathfrak{P}$ denotes a class of problems (e.g.\ smooth (strongly) convex minimization, nonsmooth nonconvex minimization, smooth convex-concave minimax optimization, nonexpansive fixed-point problems, etc.), $N$ is the number of accesses to the computation oracle associated with the algorithm class, and $\cM$ is a measure of the performance---error, where smaller is better---of an algorithm $\cA$ for solving $\cP$ using $N$ oracle accesses.

\citet{NemirovskiYudin1983_problem} formalized and popularized the idea of studying asymptotically optimal algorithms in the context of minimax optimality\footnote{They originally formulated the problem as finding $\min \left\{ N \,\middle|\, \exists \cA \in \mathfrak{A} \text{ such that } \cM(\cA,\cP; N) \le \epsilon \right\}$, which is, in the end, equivalent to \eqref{eqn:minimax-optimality}.}. 
For example, for smooth convex minimization problems, the celebrated accelerated gradient method (AGM) of \citet{Nesterov1983_method} achieves $f(x_N) - \min_x f(x) = \cO\left(N^{-2}\right)$, which is optimal up to a constant for~\eqref{eqn:minimax-optimality} among all deterministic first-order algorithms using the gradient oracle $\nabla f(\cdot)$.
This type of result could be viewed as searching for an \textit{approximate} solution $\cA \in \mathfrak{A}$ which solves~\eqref{eqn:minimax-optimality}.
Over the decades since then, numerous results of the same spirit have appeared in the literature, identifying up-to-constant optimal approximate solutions to \eqref{eqn:minimax-optimality}.

A recently active movement in the field, initiated by \citet{DroriTeboulle2014_performance, KimFessler2016_optimized, TaylorHendrickxGlineur2017_smooth}, is to reformulate \eqref{eqn:minimax-optimality} into a finite-dimensional optimization problem, compute its numerical solution, and mathematically characterize an optimized solution $\cA_\star$ based on numerical insights.
Following this process, \citet{KimFessler2016_optimized} identified the optimized gradient method (OGM), which outperforms AGM by a factor $\approx 2$, and OGM was later proved to be an \textit{exact} optimal solution to \eqref{eqn:minimax-optimality} for smooth convex minimization \citep{Drori2017_exact}.
Taking inspiration from these results, subsequent works have discovered new exact optimal solutions to smooth strongly convex minimization \citep{TaylorDrori2023_optimal, DroriTaylor2022_oracle}, nonsmooth convex minimization \citep{DroriTaylor2020_efficient} and composite convex minimization \citep{JangGuptaRyu2025_computerassisted}.
For nonexpansive fixed-point problems (equivalently, proximal setting for monotone inclusion problems), \citet{SabachShtern2017_first} first achieved accelerated $\cO(N^{-2})$ convergence in squared fixed-point residual, which was later refined to the exact minimax optimal rate \citep{Lieder2021_convergence, Kim2021_accelerated, ParkRyu2022_exact}.

While these works represent a clear step toward resolving the meta-problem~\eqref{eqn:minimax-optimality}, they fall short of providing a mathematically \textit{complete} answer.
For nonsmooth convex minimization, there are distinct exact optimal $\cA$'s known \citep{Nesterov2004_introductory, DroriTaylor2020_efficient}, and for nonexpansive fixed-point problems, there are infinitely many exact optimal $\cA$ \citep{YoonKimSuhRyu2024_optimal}. 
The discussion of whether known optimal algorithms for \eqref{eqn:minimax-optimality} for distinct setups comprise a complete set of solutions is either missing or has been answered in the negative \citep{YoonKimSuhRyu2024_optimal}.

To summarize the chronological development of solving problem~\eqref{eqn:minimax-optimality} for the nonexpansive fixed-point setting, the literature has progressed through the following three stages:
 (i)~obtaining approximate (up-to-constant optimal) algorithms \citep{SabachShtern2017_first};
 (ii)~identifying \emph{an} exact optimal algorithm \citep{Lieder2021_convergence, Kim2021_accelerated, ParkRyu2022_exact};
 (iii)~recognizing that the exact optimal algorithm is not unique \citep{YoonKimSuhRyu2024_optimal}.
The natural next step in this mathematical progression is therefore (iv)~determining \emph{all} exact optimal solutions---which is what we accomplish in this work. We present the complete set of \emph{all} exact optimal algorithms for nonexpansive fixed-point problems using the notions of H-invariants and H-certificates, thereby demonstrating that a full resolution of \eqref{eqn:minimax-optimality} is indeed achievable.

\section{Key concepts and main results}

We consider the fixed-point problems
\begin{align}
\label{eqn:fixed-point}
\tag{FPP}
\begin{array}{cc}
    \underset{y \in \reals^d}{\textrm{find}} & y = \opT (y) 
\end{array}
\end{align}
where $\opT\colon \reals^d \to \reals^d$ is nonexpansive (1-Lipschitz):
\[
    \norm{\opT x - \opT y} \le \norm{x - y}, \quad \forall x,y \in \reals^d .
\]
We measure an algorithm's convergence via squared norm of the fixed-point residual, $\sqnorm{y - \opT y}$.

\paragraph{Equivalence to monotone inclusion.}
\eqref{eqn:fixed-point} with nonexpansive $\opT$ can be recast into a monotone inclusion problem: define $\opA = \left(\frac{\opI + \opT}{2}\right)^{-1} - \opI$ (equivalently, $\opT = 2\JA - \opI$ where $\JA = (\opI + \opA)^{-1}$ is the resolvent).
Here the inverse of $\frac{\opI + \opT}{2}$ denotes the set-valued operator defined by $y \in \left(\frac{\opI + \opT}{2}\right)^{-1}(x) \iff (\opI + \opT)(y) = y + \opT y = 2x$.
Then $\opA\colon \reals^d \rightrightarrows \reals^d$ is a maximally monotone operator, satisfying $\inprod{u-v}{x-y} \ge 0$ for any $u\in \opA x$, $v\in \opA y$ and has a maximal graph with this property \citep[Proposition~23.8]{bauschke2011convex}. 
With this correspondence, for $y \in \reals^d$, letting $x = \opJ_\opA y$ and $\Tilde{\opA}x := y-x \in \opA x$, we have $y - \opT y = 2\Tilde{\opA}x$ and in particular $y = \opT y \iff \Tilde{\opA} x = 0$.
Therefore, \eqref{eqn:fixed-point} is equivalent to the monotone inclusion problem
\begin{align*}
\begin{array}{cc}
    \underset{x \in \reals^d}{\textrm{find}} & 0 = \opA (x) 
\end{array}
\end{align*}
where the corresponding convergence measure is the squared operator norm $\sqnorm{\Tilde{\opA}x} = \frac{1}{4}\sqnorm{y - \opT y}$.

\paragraph{Complexity lower bound.}
There exists a lower bound, \cref{theorem:lower-bound}, for deterministic first-order algorithms for solving \eqref{eqn:fixed-point}.
This result first considers the class of algorithms $\fA_{\text{span}}$ satisfying the span condition
\begin{align*}
    y_{k+1} \in y_k + \mathrm{span}\left\{ y_0 - \opT y_0 , \dots, y_k - \opT y_k \right\} 
\end{align*}
and then extends it to the class $\fA_{\text{det}}$ of all deterministic first-order algorithms using the resisting oracle technique of \citet{NemirovskiYudin1983_problem}.

\begin{theorem}[\citep{ParkRyu2022_exact}]
\label{theorem:lower-bound}
Let $d\ge 2(N-1)$, $y_0 \in \reals^d$ and $R > 0$.
For any deterministic first-order algorithm using $N-1$ operator evaluations, there exists a nonexpansive operator $\opT\colon \reals^d \to \reals^d$ with a fixed point $y_\star$ such that $\|y_0 - y_\star\| = R$, and 
\begin{align*}
    \sqnorm{y_{N-1} - \opT y_{N-1}} \ge \frac{4R^2}{N^2} = \frac{4\sqnorm{y_0 - y_\star}}{N^2} .
\end{align*}
\end{theorem}

Following the formulation and notation of~\eqref{eqn:minimax-optimality}, this states that when $\fP$ is the class of all nonexpansive fixed-point problems (represented by the tuple $(\opT, y_0)$), $\fA = \fA_\text{det}$ and $\cM(\cA,\opT,y_0; N) = \frac{\sqnorm{y_{N-1} - \opT y_{N-1}}}{\sqnorm{y_0 - y_\star}}$, then for any $\cA \in \fA$, we have $\max_{(\opT,y_0) \in \fP} \cM(\cA,\opT,y_0; N) \ge \frac{4}{N^2}$.
Here we use the convention of denoting the squared fixed-point residual norm for $y_{N-1}$ as $\cM(\cA,\opT,y_0; N)$ because computing the fixed-point residual itself requires an additional access to the operator evaluation oracle.

\paragraph{Optimal Halpern method exactly matches the lower bound.}
The lower bound of \cref{theorem:lower-bound} is matched by the classical Halpern iteration \citep{Halpern1967_fixed} with optimized interpolation parameters, which we call the \emph{Optimal Halpern Method (OHM)}: for $k=0,1,2,\dots$,
\begin{align}
\label{eqn:OHM}
    y_{k+1} = \frac{1}{k+2} y_0 + \frac{k+1}{k+2} \opT y_k .
    \tag{OHM}
\end{align}

\begin{theorem}[\citep{Kim2021_accelerated, Lieder2021_convergence}]
\label{theorem:OHM-rate}
Suppose $\opT\colon \reals^d \to \reals^d$ is nonexpansive and $y_\star \in \fix\opT$ (i.e., $y_\star = \opT y_\star$). Then for any $N=1,2,\dots$, \ref{eqn:OHM} exhibits the (exact optimal) rate
\begin{align}
\label{eqn:optimal-rate-N-iterate}
    \|y_{N-1} - \opT y_{N-1}\|^2 \le \frac{4\|y_0 - y_\star\|^2}{N^2} .
\end{align}
\end{theorem}

Again using the notation of~\eqref{eqn:minimax-optimality}, \cref{theorem:OHM-rate} states that $\max_{(\opT,y_0) \in \fP} \cM(\cA,\opT,y_0; N) \le \frac{4}{N^2}$.
Combining this with \cref{theorem:lower-bound}, we establish the equality, representing minimax optimality:
\begin{align*}
    \min_{\cA \in \fA_{\text{det}}} \, \max_{(\opT,y_0) \in \fP} \, \cM(\cA,\opT,y_0; N) = \frac{4}{N^2} .
\end{align*}

\paragraph{Class of H-matrix representable algorithms.}
\ref{eqn:OHM} is a fixed step-size algorithm that admits a representation of the form
\begin{align}
\label{eqn:H-matrix-representation}
    y_{k+1} = y_k - \sum_{j=0}^k h_{k+1,j+1} \left( y_j - \opT y_j \right)
\end{align}
with some predetermined $h_{k+1,j+1} \in \reals$. 
When an algorithm with $N-1$ total iterations (operator evaluations of $\opT$) is expressible in the format~\eqref{eqn:H-matrix-representation}, we say it is \emph{H-matrix representable}, and define its \emph{H-matrix} as the lower triangular matrix $H \in \reals^{(N-1) \times (N-1)}$ with
\begin{align*}
    (H)_{k,j} = \begin{cases}
        h_{k,j} & \text{if $j\le k$} \\
        0 & \text{otherwise.} 
    \end{cases}
\end{align*}
For \ref{eqn:OHM}, in particular, we have 
\begin{align*}
\begin{split}
    & \left(H_{\text{OHM}}\right)_{k,j} =
    \begin{cases}
        -\frac{j}{k(k+1)} & \text{if } j < k \\
        \frac{k}{k+1}     & \text{if } j = k .
    \end{cases}
\end{split} 
\end{align*}
We denote the class of H-matrix representable algorithms by $\fA_\text{H}$, so that we have $\fA_\text{H} \subset \fA_\text{span} \subset \fA_\text{det}$.

\paragraph{Non-uniqueness of minimax optimal algorithms via H-dual.}
Interestingly, \citet{YoonKimSuhRyu2024_optimal} showed that given a fixed $N$, a minimax optimal H-matrix representable algorithm satisfying \eqref{eqn:optimal-rate-N-iterate} is not unique.
In particular, they introduced the \emph{Dual-OHM} algorithm which exhibits the same rate~\eqref{eqn:optimal-rate-N-iterate} as OHM, defined as
\begin{align} \label{eqn:Dual-OHM} \tag{Dual-OHM}
    y_{k+1} = y_k + \frac{N-k-1}{N-k} \left( \opT y_k - \opT y_{k-1} \right), \quad k=0,\dots,N-2
\end{align}
where $\opT y_{-1} = y_0$.
\ref{eqn:Dual-OHM} has the H-matrix
\begin{align*}
\begin{split}
     \left(H_{\text{Dual-OHM}}\right)_{k,j} =
    \begin{cases}
        -\frac{N-k}{(N-j)(N-j+1)} & \text{if } j < k \\
        \frac{N-k}{N-k+1}         & \text{if } j = k
    \end{cases}
\end{split} 
\end{align*}
which is the anti-diagonal transpose of $H_{\text{OHM}}$, i.e., they are the reflections of each other along the anti-diagonal: 
\begin{align*}
    \left(H_{\text{Dual-OHM}}\right)_{k,j} = \left(H_{\text{OHM}}\right)_{N-j,N-k} .    
\end{align*}
In general, the \textit{H-dual} of an algorithm in $\mathfrak{A}_\mathrm{H}$ is an algorithm represented by the anti-diagonal transpose of its H-matrix---the terminology originated from \cite{KimOzdaglarParkRyu2023_timereversed}.

\citet{YoonKimSuhRyu2024_optimal} in fact proved the existence of a family of H-matrices, all satisfying~\eqref{eqn:optimal-rate-N-iterate}, admitting an $(N-2)$-dimensional open and convex parametrization.
At the same time, however, they showed that their family is incomplete, i.e., fails to capture all H-matrices with the rate~\eqref{eqn:optimal-rate-N-iterate}.
A complete characterization of minimax optimal H-matrices, i.e., finding all solutions to~\eqref{eqn:minimax-optimality} within $\fA_\text{H}$, thus has been left as an open problem. 

\paragraph{Main contributions.}
In this paper, we close the problem of identifying the set of all exact minimax optimal H-matrices for nonexpansive fixed-point problems.
The set is characterized as (a subset of) the intersection of level sets of $N-1$ homogeneous polynomials $P(N-1,m; H)$ of degree $m$ (where $m$ runs over $1,\dots,N-1$) in the entries of H-matrices.
We define them in \cref{section:H-invariants-certificates}, and name them \emph{H-invariants}, as their values stay invariant over the set of optimal H-matrices.
We additionally define what we call \emph{H-certificates} $\lambda_{k,j}^\star(H)$ (where $k=2,\dots,N$ and $j=1,\dots,k-1$), which are multipliers coupled with inequalities used in its convergence proof, which should be nonnegative. 
They are also polynomials in the entries of $H$, whose precise formulas are explicitly provided in \cref{theorem:lambda-characterization}, and the inequality constraints $\lambda_{k,j}^\star(H) \ge 0$ characterize a region of optimality within the common level set of H-invariants.
Using these algebraic quantities, we state our main result as follows.

\begin{theorem}
\label{theorem:main-result}
A fixed-point algorithm \eqref{eqn:H-matrix-representation} achieves the exact optimal worst-case rate of $\sqnorm{y_{N-1} - \opT y_{N-1}} \le \frac{4\sqnorm{y_0 - y_\star}}{N^2}$ over the class of nonexpansive operators $\opT\colon \reals^d \to \reals^d$ if and only if
\begin{align}
\label{eqn:optimal-H-conditions}
\tag{$\ast$}
    \begin{cases}
    P(N-1,m; H) = \frac{1}{N} \binom{N}{m+1}, & m=1,\dots,N-1 \\
    \lambda_{k,j}^\star(H) \geq 0, & k=2,\dots,N, \,\, j=1,\dots,k-1 .
    \end{cases}
\end{align}
\end{theorem}

As byproducts of this main characterization result, we obtain some interesting additional results.
\begin{itemize}
    \item We show that \ref{eqn:OHM} is the only \textit{anytime algorithm}, converging at the optimal rate for all iterations $k$ (\cref{corollary:anytime-optimal-algorithm-is-unique}).

    \item We provide a procedure for identifying new optimal algorithms, with explicit H-matrices, by requiring a majority of H-certificates to be $0$ (with only one nonzero $\lambda_{k,j}^\star (H)$ for each $j=1,\dots,N-1$). These algorithms are equipped with simple convergence proofs involving only the corresponding sparse set of inequalities.
    From this, we discover the first optimal algorithms that are self-H-dual or lack an optimal H-dual.
\end{itemize}

\paragraph{Equivalent result for monotone inclusion.}
Given the equivalent maximally monotone inclusion formulation with $\opA = \left(\frac{\opI + \opT}{2}\right)^{-1} - \opI$, for each $y_k$ ($k=0,\dots,N-1$) generated by an H-matrix representable algorithm \eqref{eqn:H-matrix-representation}, define $x_{k+1} = \JA y_k$ and $g_{k+1} = y_k - x_{k+1} \in \opA x_{k+1}$.
Then $\opT y_k = x_{k+1} - g_{k+1} = y_k - 2g_{k+1}$, so we can write
\begin{align}
\label{eqn:H-matrix-monotone-version}
    y_{k+1} = y_k - \sum_{j=0}^k 2h_{k+1,j+1} g_{j+1} .
\end{align}
Then \cref{theorem:main-result} can be translated into the following statement, and vice versa.

\begin{corollary}
A proximal-point algorithm~\eqref{eqn:H-matrix-monotone-version} achieves the exact optimal worst-case rate of $\sqnorm{g_N} \le \frac{\sqnorm{y_0 - y_\star}}{N^2}$, where $g_N = y_{N-1} - x_N \in \opA x_N$, over the class of maximally monotone operators $\opA\colon \reals^d \rightrightarrows \reals^d$ if and only if \eqref{eqn:optimal-H-conditions} holds.
\end{corollary}

\paragraph{Organization.}
The remainder of the paper is organized as follows. 
\Cref{section:H-invariants-certificates} defines H-invariants and H-certificates, and discusses their basic properties.
Sections~\ref{section:H-invariance-necessity} and \ref{section:completing-chracterization} together present the proof of \cref{theorem:main-result}.
\Cref{section:H-invariance-necessity} shows the necessity of H-invariance conditions $P(N-1,m; H) = \frac{1}{N} \binom{N}{m+1}$ for optimality and provides \cref{corollary:anytime-optimal-algorithm-is-unique} as its direct implication.
\Cref{subsection:optimality-and-lambda} shows that H-invariance, together with $\lambda_{k,j}^\star(H) \ge 0$, implies optimality.
\Cref{subsection:necessity-of-nonnegative-H-certificates} completes the characterization by conversely showing that the inequalities $\lambda_{k,j}^\star(H)\geq 0$ are also necessary for optimality.
Finally, \cref{section:new-boundary-algorithms} illustrates the characterization of novel optimal algorithms as an application of \cref{theorem:main-result}.

\section{H-invariants and H-certificates: Motivation and definitions}
\label{section:H-invariants-certificates}

\subsection{H-invariants}
\label{subsection:H-invariants}
H-invariants first appeared in \cite{YoonKimSuhRyu2024_optimal} (without the name) as a tool to show that, when the operator is linear, an algorithm in $\mathfrak{A}_\mathrm{H}$ and its H-dual produce the same terminal iterates from the same initial point (while the intermediate trajectories may significantly differ).
We start with this discussion, and then naturally define H-invariants.

Suppose that $\opT$ is linear, so that there exists a matrix $G \in \reals^{d\times d}$ satisfying $\opT y = (I - 2G)y = y - 2Gy$ for all $y \in \reals^d$, where $I$ is the identity matrix. 
Then we can write
\begin{align*}
    y_{k+1} = y_k - \sum_{j=0}^k 2h_{k+1,j+1} Gy_j .
\end{align*}
Using induction, we can write for $k=0,\dots,N-1$,
\begin{align}
\label{eqn:H-matrix-linear-case-polynomial-expression}
    y_k = \sum_{m=0}^{k} (-1)^m P(k,m; 2H) G^m y_0 
\end{align}
where $P(k,0; H) = 1$ and for $m=1,\dots,k$,
\begin{align*}
    P(k,m; H) = \sum_{\substack{1 \le j(1) \le i(1) < j(2) \le i(2) < \cdots \le i(m-1) < j(m) \le i(m) \le k}} \prod_{r=1}^m h_{i(r),j(r)} .
\end{align*}
The first few concrete expressions are as follows:
\begin{align*}
    P(1,1; H) & = h_{1,1} \\
    P(2,1; H) & = h_{1,1} + h_{2,1} +  h_{2,2} \\
    P(2,2; H) & = h_{1,1} h_{2,2} \\
    P(3,1; H) & = h_{1,1} + h_{2,1} + h_{2,2} + h_{3,1} + h_{3,2} + h_{3,3} \\
    P(3,2; H) & = h_{1,1}h_{2,2} + h_{1,1}h_{3,2}  + h_{1,1}h_{3,3} + h_{2,1}h_{3,3} + h_{2,2}h_{3,3} \\
    P(3,3; H) & = h_{1,1}h_{2,2}h_{3,3} .
\end{align*}
By definition, algorithms with the same values of $P(N-1,m; H)$ for $m=1,\dots,N-1$ produce the same terminal iterate $y_{N-1}$ if $\opT$ is linear.
This is the case for algorithms in the H-dual relationship \citep[Proposition~I.1]{YoonKimSuhRyu2024_optimal}. 
Note that each $P(k,m; H)$ is a homogeneous polynomial of degree $m$ in entries of $H$, so the effect of scaling $H$ is $P(k,m; cH) = c^m P(k,m; H)$ for any $c\in \reals$ and $m=1,\dots,k$.

\begin{definition}
\label{def:H-invariants}
For an algorithm represented by a lower triangular $H \in \reals^{(N-1)\times(N-1)}$, its \emph{H-invariants} are the quantities $P(N-1,m; H)$ for $m=1,\dots,N-1$.
\end{definition}

\subsection{H-certificates}
\label{subsection:H-certificates}

Consider the equivalent formulation~\eqref{eqn:H-matrix-monotone-version} of H-matrix representable algorithms based on the maximally monotone operator $\opA = \left(\frac{\opI+\opT}{2}\right)^{-1}$. With this notation, nonexpansivity inequalities for $\opT$ can be recast into monotonicity inequalities for $\opA$ as
\begin{align*}
    & \sqnorm{y_{k-1} - y_{j-1}} - \sqnorm{\opT y_{k-1} - \opT y_{j-1}} \ge 0 \\
    & \iff \sqnorm{x_k + g_k - x_j - g_j} - \sqnorm{x_k - g_k - x_j + g_j} \ge 0 \\
    & \iff \inprod{x_k - x_j}{g_k - g_j} \ge 0 
\end{align*}
and similarly,
\begin{align*}
    \sqnorm{y_{k-1} - y_\star} - \sqnorm{\opT y_{k-1} - y_\star} \ge 0 \iff \inprod{g_k}{x_k - y_\star} \ge 0
\end{align*}
for $k,j=1,\dots,N$ and $y_\star \in \fix\opT$.
In \cref{theorem:lambda-characterization}, we prove that if $H$ satisfies the \textit{H-invariance conditions} $P(N-1,m; H) = \frac{1}{N} \binom{N}{m+1}$ for $m=1,\dots,N-1$, then there exist unique $\lambda_{k,j}$, given as explicit polynomials in the entries of $H$ (whose precise forms are also given in \cref{theorem:lambda-characterization}) for $1 \le j < k \le N$, satisfying
\begin{align}
\label{eqn:H-symmetry-sufficiency-core-identity}
    0 = N \sqnorm{g_N} + \inprod{g_N}{x_N - y_0} + \sum_{k=1}^N \sum_{j=1}^{k-1} \lambda_{k,j} \inprod{x_k - x_j}{g_k - g_j} .
\end{align}

\begin{definition}
For an H-matrix representable algorithm satisfying $P(N-1,m;H) = \frac{1}{N}\binom{N}{m+1}$ for $m=1,\dots,N-1$, its \emph{H-certificates} are the unique values $\lambda_{k,j}^\star(H)$ of $\lambda_{k,j}$ for each pair $(k,j)$ of indices $1 \le j < k \le N$ satisfying \eqref{eqn:H-symmetry-sufficiency-core-identity}.
\end{definition}

If $\lambda_{k,j}^\star (H) \ge 0$ for all $1 \le j < k \le N$, then from \eqref{eqn:H-symmetry-sufficiency-core-identity} we have
\begin{align*}
    0 & \ge N \sqnorm{g_N} + \inprod{g_N}{x_N - y_0} \ge N \sqnorm{g_N} + \inprod{g_N}{y_\star - y_0} \ge \frac{N}{2} \sqnorm{g_N} - \frac{1}{2N} \sqnorm{y_\star - y_0}
\end{align*}
where the second inequality uses monotonicity of $\opA$ and the last inequality is Young's inequality.
This, in turn, implies
\begin{align*}
    \sqnorm{y_{N-1} - \opT y_{N-1}} = 4\sqnorm{g_N} \le \frac{4\sqnorm{y_0 - y_\star}}{N^2} .
\end{align*}
The above discussion illustrates why H-invariance and nonnegative H-certificates together imply exact optimality of the algorithm defined by $H$, and indicates that the proof of its optimality is given by the form~\eqref{eqn:H-symmetry-sufficiency-core-identity}.

\section{Necessity of H-invariance for optimality}
\label{section:H-invariance-necessity}

This section presents the following \cref{theorem:H-invariance-necessity}, showing that H-invariance conditions are necessary for optimality.
Furthermore, as a direct implication of this result, we prove that OHM is the unique anytime optimal algorithm (\cref{corollary:anytime-optimal-algorithm-is-unique}).

\begin{theorem}
\label{theorem:H-invariance-necessity}
A fixed-point algorithm \eqref{eqn:H-matrix-representation} achieves the exact optimal worst-case rate of $\sqnorm{y_{N-1} - \opT y_{N-1}} \le \frac{4\sqnorm{y_0 - y_\star}}{N^2}$ over the class of nonexpansive operators only if $P(N-1,m; H) = \frac{1}{N} \binom{N}{m+1}$ for $m=1,\dots,N-1$.
\end{theorem}

\begin{proof}
Consider the worst-case nonexpansive operator from \citep{ParkRyu2022_exact}, $\opT = \opI - 2G \colon \reals^{N} \to \reals^{N}$ where $G$ is the linear operator represented by the matrix
\begin{align*}
    G = \frac{1}{2} \begin{bmatrix} 
    1 & 0 & \cdots & 0 & 1 \\ 
    -1 & 1 & \cdots & 0 & 0 \\
    \vdots & \vdots & \ddots & \vdots & \vdots \\
    0 & 0 & \cdots & 1 & 0 \\
    0 & 0 & \cdots & -1 & 1
    \end{bmatrix} \in \reals^{N\times N} .
\end{align*}
The matrix $2G$ has determinant 2, showing that $G$ has a unique zero, i.e., $\opT$ has a unique fixed point $y_\star = 0$.

Let us denote the elementary basis vectors of $\reals^N$ by $e_1,\dots,e_N$.
Observe that $G(e_1+\dots+e_N) = e_1$, so
\begin{align*}
    G^{-1} e_1 = e_1 + \dots + e_N = \begin{bmatrix} 1 \\ 1 \\ \vdots \\ 1 \end{bmatrix} \in \reals^N .
\end{align*}
Fix $R>0$, and suppose that we run the algorithm \eqref{eqn:H-matrix-representation} with $y_0 = -\frac{R}{\sqrt{N}} G^{-1}e_1$.
(Note that with this choice, we have $\norm{y_0 - y_\star} = R$.)
By the definition of our update rules, we have
\begin{align}
\label{eqn:necessary-span-definition}
    y_{k+1} \in y_k + \spann \{Gy_0, Gy_1,\dots, Gy_k\}
\end{align}
for $k=0,\dots,N-2$.
Using this and the induction arguments, we can show that
\begin{align}
\label{eqn:necessary-span-elementary-basis}
    y_j \in y_0 + \spann \{e_1,\dots,e_j\} 
\end{align}
for $j=1,\dots,N-1$.
The case $j=1$ is immediate from \eqref{eqn:necessary-span-definition} with $k=0$ and $Gy_0 \in \spann\{e_1\}$.
Next, suppose that \eqref{eqn:necessary-span-elementary-basis} holds for $j=1,\dots,\ell$, where $\ell \le N-2$.
Then by \eqref{eqn:necessary-span-definition}, we have
\begin{align*}
    y_{\ell+1} & \in y_\ell + \spann \{Gy_0, Gy_1,\dots, Gy_\ell\} \\
    & \in y_0 + \spann \{e_1,\dots,e_\ell\} + \spann \{Gy_0, Gy_1,\dots, Gy_\ell\} .
\end{align*}
By induction hypothesis, we have $Gy_j \in Gy_0 + \spann\{Ge_1,\dots,Ge_j\}$ for $j=1,\dots,\ell$, and because $Ge_j = \frac{1}{2} e_j - \frac{1}{2} e_{j+1}$,
\begin{align*}
    \spann \{Gy_0, Gy_1,\dots, Gy_\ell\} \subseteq \spann\{ e_1, \dots, e_{\ell+1} \} .
\end{align*}
Therefore $y_{\ell+1} \in y_0 + \spann \{e_1,\dots,e_{\ell+1}\}$, which completes the induction and thus proves \eqref{eqn:necessary-span-elementary-basis}.

Now by definition of $\opT$ and \eqref{eqn:necessary-span-elementary-basis}, we have
\begin{align*}
    y_{N-1} - \opT y_{N-1} = 2Gy_{N-1} & \in 2Gy_0 + \spann \{Ge_1, \dots, Ge_{N-1}\} \\
    & = -\frac{2R}{\sqrt{N}} e_1 + \spann \{Ge_1, \dots, Ge_{N-1}\} .
\end{align*}
Observe that $V = \spann \{Ge_1, \dots, Ge_{N-1}\}$ is an $(N-1)$-dimensional subspace of $\reals^N$, and its orthogonal complement is $V^\perp = \spann \{ e_1 + \dots + e_N \}$.
Therefore,
\begin{align*}
    \sqnorm{y_{N-1} - \opT y_{N-1}} & = \sqnorm{2Gy_{N-1}} \\
    & = \sqnorm{\proj_{V} \left( 2Gy_{N-1} \right)} + \sqnorm{\proj_{V^\perp} \left( 2Gy_{N-1} \right)} \\
    & \ge \sqnorm{\proj_{V^\perp} \left( 2Gy_{N-1} \right)}  \\
    & = \sqnorm{\proj_{V^\perp} \left( -\frac{2R}{\sqrt{N}} e_1 \right)} \\
    & = \sqnorm{-\frac{2R}{N\sqrt{N}} (e_1+\dots+e_N)} \\
    & = \frac{4R^2}{N^2} 
\end{align*}
where $\proj_V$ and $\proj_{V^\perp}$ denote the projection maps onto the subspaces $V$ and $V^\perp$, respectively.
If $y_{N-1}$ is generated by an algorithm achieving the exact optimal worst case, then it must necessarily satisfy $\sqnorm{y_{N-1} - \opT y_{N-1}} \le \frac{4\sqnorm{y_0 - y_\star}}{N^2} = \frac{4R^2}{N^2}$ so the equality must hold.
This holds if and only if $\proj_V (2Gy_{N-1}) = 0$, i.e., if and only if
\begin{align}
\label{eqn:necessary-condition-projection-identity}
    Gy_{N-1} = \proj_{V^\perp} (Gy_{N-1}) = -\frac{R}{N\sqrt{N}} (e_1+\dots+e_N) .
\end{align}
Now from \eqref{eqn:H-matrix-linear-case-polynomial-expression} we have
\begin{align}
\label{eqn:necessary-condition-GyNm1-expression}
    Gy_{N-1} = \sum_{m=0}^{N-1} (-1)^m P(N-1,m; 2H) G^{m+1} y_0 .
\end{align}
Next, we show that
\begin{align}
\label{eqn:necessary-condition-Gmp1y0-expression}
    G^{m+1} y_0 = -\frac{R}{\sqrt{N}} \frac{1}{2^m} \sum_{j=1}^{m+1} (-1)^{j-1} \binom{m}{j-1} e_j
\end{align}
for $m=0,1,\dots,N-1$.
The case $m=0$ is immediate because $Gy_0 = -\frac{R}{\sqrt{N}} e_1 = -\frac{R}{\sqrt{N}} \frac{1}{2^0} \binom{0}{0} e_1$.
Now, to use induction, suppose that \eqref{eqn:necessary-condition-Gmp1y0-expression} holds for some $0 \le m \le N-2$.
Then
\begin{align*}
    G^{m+2} y_0 & = G \left( G^{m+1} y_0 \right) \\
    & = G \left( -\frac{R}{\sqrt{N}} \frac{1}{2^m} \sum_{j=1}^{m+1} (-1)^{j-1} \binom{m}{j-1} e_j \right) \\
    & = -\frac{R}{\sqrt{N}} \frac{1}{2^m} \sum_{j=1}^{m+1} (-1)^{j-1} \binom{m}{j-1} Ge_j \\
    & = -\frac{R}{\sqrt{N}} \frac{1}{2^m} \sum_{j=1}^{m+1} (-1)^{j-1} \binom{m}{j-1} \frac{e_j - e_{j+1}}{2} \\
    & = -\frac{R}{\sqrt{N}} \frac{1}{2^{m+1}} \sum_{j=1}^{m+1} (-1)^{j-1} \binom{m}{j-1} (e_j - e_{j+1}) \\
    & = -\frac{R}{\sqrt{N}} \frac{1}{2^{m+1}} \left( e_1 + \sum_{j=2}^{m+1} (-1)^{j-1} \left( \binom{m}{j-1} + \binom{m}{j-2} \right) e_j + (-1)^{m+1} e_{m+2} \right) \\
    & = -\frac{R}{\sqrt{N}} \frac{1}{2^{m+1}} \left( e_1 + \sum_{j=2}^{m+1} (-1)^{j-1} \binom{m+1}{j-1} e_j + (-1)^{m+1} e_{m+2} \right) \\
    & = -\frac{R}{\sqrt{N}} \frac{1}{2^{m+1}} \sum_{j=1}^{m+2} (-1)^{j-1} \binom{m+1}{j-1} e_j
\end{align*}
where for the second-to-last equality we use the combinatorial identity $\binom{m}{j-1} + \binom{m}{j-2} = \binom{m+1}{j-1}$.
This completes the induction, and proves \eqref{eqn:necessary-condition-Gmp1y0-expression}.
Now we plug \eqref{eqn:necessary-condition-Gmp1y0-expression} into \eqref{eqn:necessary-condition-GyNm1-expression} and use $P(N-1,m;2H) = 2^m P(N-1,m;H)$ to cancel out the $2^m$ factor to obtain
\begin{align*}
    Gy_{N-1} & = -\frac{R}{\sqrt{N}} \sum_{m=0}^{N-1} (-1)^m P(N-1,m; H) \sum_{j=1}^{m+1} (-1)^{j-1} \binom{m}{j-1} e_j \\
    & = -\frac{R}{\sqrt{N}} \sum_{j=1}^N \left[ \sum_{m=j-1}^{N-1} (-1)^{m+j-1} \binom{m}{j-1} P(N-1,m; H) \right] e_j .
\end{align*}
Therefore, \eqref{eqn:necessary-condition-projection-identity} holds if and only if
\begin{align}
\label{eqn:necessary-P-invariant-equations}
    \sum_{m=j-1}^{N-1} (-1)^{m+j-1} \binom{m}{j-1} P(N-1,m; H) = \frac{1}{N}
\end{align}
for $j=1,\dots,N$.
Starting with the case $j=N$, which yields
\[
    \binom{N-1}{N-1} P(N-1,N-1; H) = \frac{1}{N} \iff P(N-1,N-1; H) = \frac{1}{N} = \frac{1}{N} \binom{N}{N} ,
\]
we can solve \eqref{eqn:necessary-P-invariant-equations} in the order of $j=N,N-1,\dots,1$ to determine the values of each $P(N-1,j; H)$---more formally, the system of linear equations \eqref{eqn:necessary-P-invariant-equations} for $j=N,\dots,1$ having $P(N-1,m;H)$ ($m=N-1,N-2,\dots,0$) as variables, is lower triangular with unit diagonal---so the values of $P(N-1,m;H)$ satisfying \eqref{eqn:necessary-P-invariant-equations} uniquely exist. 
Now consider the combinatorial identity
\begin{align*}
    \sum_{m=j-1}^{N-1} (-1)^{m+j-1} \binom{m}{j-1} \binom{N}{m+1} = 1 
\end{align*}
for any fixed $j=1,\dots,N$, which follows by setting $a=-j$, $b=N$ and $c=N-j$ in \Cref{lemma:chu-vandermonde}---the standard Chu-Vandermonde identity derived from the binomial series expansion, provided in Appendix~\ref{section:combinatorial-lemmas}---and then using the identities $\binom{-j}{i} = (-1)^i \binom{i+j-1}{i} = (-1)^i \binom{i+j-1}{j-1}$, applying the change of index $m=i+j-1$ and using $\binom{N}{N-m-1} = \binom{N}{m+1}$.
This shows that \eqref{eqn:necessary-P-invariant-equations} holds if and only if $P(N-1,m;H) = \frac{1}{N} \binom{N}{m+1}$ for $m=0,\dots,N-1$, which proves \cref{theorem:H-invariance-necessity}.
Note that in the theorem statement, we do not include the case $m=0$ which gives $P(N-1,0;H) = 1$, as it holds trivially by definition.

\end{proof}

Although \cref{theorem:H-invariance-necessity} is only a partial characterization result, establishing the necessity of H-invariance, it already has the following interesting implication.

\begin{definition}
\label{def:anytime-optimal}
We say that an H-matrix representable algorithm of the form \eqref{eqn:H-matrix-representation} with predetermined $h_{k+1,j+1} \in \reals$ (defined for all $k=0,1,2,\dots$ and $j=0,1,\dots,k$) is \emph{anytime optimal} if
\begin{align}
\label{eqn:anytime-optimal}
    \sqnorm{y_k - \opT y_k} \le \frac{4\sqnorm{y_0 - y_\star}}{(k+1)^2}
\end{align}
holds for all $k=0,1,\dots$, for any nonexpansive operator $\opT\colon \reals^d \to \reals^d$ and an initial point $y_0 \in \reals^d$.
If $N\in \mathbb{N}$ and \eqref{eqn:anytime-optimal} holds for all $k=N-1,N,\dots$, i.e., for all iterates including and after $y_{N-1}$, then we say that the algorithm is \emph{anytime-beyond-$N$-steps optimal}.    
\end{definition}

\begin{corollary}
\label{corollary:anytime-optimal-algorithm-is-unique}
OHM is the only anytime optimal algorithm of the form~\eqref{eqn:H-matrix-representation}.
More generally, for any $N\ge 1$, if an algorithm of the form~\eqref{eqn:H-matrix-representation} is anytime-beyond-$N$-steps optimal, then it takes the form
\begin{align*}
    y_{k+1} = \frac{1}{k+2} y_0 + \frac{k+1}{k+2} \opT y_k
\end{align*}
for $k=N-1,N,\dots$.
\end{corollary}

\begin{proof}
Suppose that the algorithm is anytime-beyond-$N$-steps optimal.
By \Cref{theorem:H-invariance-necessity}, for any $k=N-1,N,\dots$, we must have $P(k,m) = \frac{1}{k+1}\binom{k+1}{m+1}$ for all $m=1,\dots,k$ (we drop the $H$-dependence throughout this proof).
This implies 
\begin{align*}
    h_{k+1,k+1} = \frac{h_{1,1}h_{2,2}\cdots h_{k+1,k+1}}{h_{1,1}h_{2,2}\cdots h_{k,k}} = \frac{P(k+1,k+1)}{P(k,k)} = \frac{\frac{1}{k+2}}{\frac{1}{k+1}} = \frac{k+1}{k+2}. 
\end{align*}
Next, we show that $h_{k+1,m} = -\frac{1}{k+2} \left( h_{m,m} + h_{m+1,m} + \dots + h_{k,m} \right)$
for each $m=1,\dots,k$ by backward induction on $m$ (starting from $m=k$). 
We state a useful lemma for this purpose, 
whose proof is deferred to Appendix~\ref{section:proof-simple-P-recursion-lemma} because although it is simple, it relies on a notation established within the appendix.

\begin{lemma}
\label{lemma:P-recursion}
For any $k=1,2,\dots$ and $m=2,\dots,k+1$, 
\begin{align*}
    P(k+1,m) = \sum_{j=m-1}^{k} \left( \sum_{i=j+1}^{k+1} h_{i,j+1} \right) P(j,m-1) .
\end{align*}
\end{lemma}

Now to show the base case $m=k$, we use \cref{lemma:P-recursion} to obtain
\begin{align*}
    P(k+1,k) & = \sum_{j=k-1}^{k} \left( \sum_{i=j+1}^{k+1} h_{i,j+1} \right) P(j,k-1) \\
    & = h_{k+1,k+1} P(k,k-1) + \left( h_{k,k} + h_{k+1,k} \right) P(k-1,k-1) ,
\end{align*}
which implies
\begin{align*}
    h_{k+1,k} & = \frac{1}{P(k-1,k-1)} \left( P(k+1,k) - h_{k+1,k+1} P(k,k-1) \right) - h_{k,k} \\
    & = \frac{1}{h_{1,1}h_{2,2}\dots h_{k-1,k-1}} \left( \frac{1}{k+2} \binom{k+2}{k+1} - \frac{k+1}{k+2} \frac{1}{k+1} \binom{k+1}{k} \right) - h_{k,k} \\
    & = \frac{h_{k,k}}{P(k,k)} \frac{1}{k+2} - h_{k,k} \\
    & = -\frac{1}{k+2} h_{k,k}
\end{align*}
where the last equality holds because $P(k,k) = \frac{1}{k+1}$.

Let $2\le m\le k-1$ and assume $h_{k+1,n} = -\frac{1}{k+2} \left( h_{n,n} + h_{n+1,n} + \dots + h_{k,n} \right)$ holds for all $n=m+1,\dots,k$.
Applying \cref{lemma:P-recursion} again, we obtain 
\begin{align}
    P(k+1,m) 
    & = h_{k+1,k+1} P(k,m-1) + \left(h_{k,k} + h_{k+1,k}\right) P(k-1,m-1) + \dots + \left(h_{m+1,m+1} + \dots + h_{k+1,m+1}\right) P(m,m-1) \nonumber\\
    & \quad + \left(h_{m,m} + \dots + h_{k+1,m}\right) P(m-1,m-1) \nonumber\\
    & = \frac{k+1}{k+2} \frac{1}{k+1} \binom{k+1}{m} + \frac{k+1}{k+2} h_{k,k} P(k-1,m-1) + \dots + \frac{k+1}{k+2} \left( h_{m+1,m+1} + \dots + h_{k,m+1} \right) P(m,m-1) \nonumber\\
    & \quad + \left(h_{m,m} + \dots + h_{k+1,m}\right) P(m-1,m-1) \nonumber\\
    & = \frac{1}{k+2} \binom{k+1}{m} + \frac{k+1}{k+2} \sum_{j=m}^{k-1} \left( \sum_{i=j+1}^k h_{i,j+1} \right) P(j,m-1) + \left(h_{m,m} + \dots + h_{k+1,m}\right) P(m-1,m-1) \label{eqn:OHM-uniqueness-theorem-induction}
\end{align}
where we have used the induction hypothesis for the second equality. Now applying \cref{lemma:P-recursion} once again, with $k-1$ in place of $k$, we have 
\begin{align*}
    & \sum_{j=m-1}^{k-1} \left( \sum_{i=j+1}^k h_{i,j+1} \right) P(j,m-1) = P(k,m) = \frac{1}{k+1} \binom{k+1}{m+1} \\
    & \implies \sum_{j=m}^{k-1} \left( \sum_{i=j+1}^k h_{i,j+1} \right) P(j,m-1) = \frac{1}{k+1} \binom{k+1}{m+1} - \left( \sum_{i=m}^k h_{i,m} \right) P(m-1,m-1) .
\end{align*}
Plugging this into \eqref{eqn:OHM-uniqueness-theorem-induction} and using the H-invariance condition for $P(k+1,m)$, we obtain
\begin{align*}
    \frac{1}{k+2} \binom{k+2}{m+1} & = P(k+1,m) \\
    & = \frac{1}{k+2} \binom{k+1}{m} + \frac{k+1}{k+2} \left[ \frac{1}{k+1} \binom{k+1}{m+1} - \left( \sum_{i=m}^k h_{i,m} \right) P(m-1,m-1) \right] \\
    & \quad + \left(h_{m,m} + \dots + h_{k+1,m}\right) P(m-1,m-1) \\
    & = \frac{1}{k+2} \left[ \binom{k+1}{m} + \binom{k+1}{m+1} \right] + P(m-1,m-1) \left[ \frac{1}{k+2} \left( h_{m,m} + \dots + h_{k,m} \right) + h_{k+1,m} \right] .
\end{align*}
Because $\binom{k+2}{m+1} = \binom{k+1}{m} + \binom{k+1}{m+1}$ and $P(m-1,m-1) \ne 0$, we conclude that $h_{k+1,m} = -\frac{1}{k+2} \left( h_{m,m} + \dots + h_{k,m} \right)$.

Finally, assuming the induction hypothesis for all $m=2,\dots,k$, we have 
\begin{align*}
    h_{1,1} + h_{2,1} + \dots + h_{k+1,1} & = P(k+1,1) - \sum_{j=2}^{k} \sum_{m=j}^{k+1} h_{m,j} - h_{k+1,k+1} \\
    & = \frac{1}{k+2} \binom{k+2}{2} - \sum_{j=2}^k \frac{k+1}{k+2} \sum_{m=j}^k h_{m,j} - \frac{k+1}{k+2} \\
    & = \frac{k+1}{2} - \frac{k+1}{k+2} \left( P(k,1) - h_{1,1} - h_{2,1} - \dots - h_{k,1} \right) - \frac{k+1}{k+2} \\
    & = \frac{k+1}{2} - \frac{k+1}{k+2} \frac{1}{k+1} \binom{k+1}{2} + \frac{k+1}{k+2} (h_{1,1} + \dots + h_{k,1}) - \frac{k+1}{k+2} \\
    & = \frac{k+1}{k+2} (h_{1,1} + \dots + h_{k,1})
\end{align*}
and thus $h_{k+1,1} = -\frac{1}{k+2} (h_{1,1} + \dots + h_{k,1})$, completing the induction.

Now note that by the update rule \eqref{eqn:H-matrix-representation},
\begin{align*}
    y_0 - y_k & = \sum_{i=1}^{k} \sum_{j=1}^{i} h_{i,j} (y_{j-1} - \opT y_{j-1}) \\
    & = \sum_{j=1}^k \sum_{i=j}^k h_{i,j} (y_{j-1} - \opT y_{j-1}) \\
    & = \sum_{j=1}^k (h_{j,j} + h_{j+1,j} + \dots + h_{k,j}) (y_{j-1} - \opT y_{j-1}) .
\end{align*}
This implies
\begin{align*}
    y_{k+1} & = y_k - \sum_{j=1}^{k+1} h_{k+1,j} (y_{j-1} - \opT y_{j-1}) \\
    & = y_k + \frac{1}{k+2} \sum_{j=1}^k (h_{j,j} + h_{j+1,j} + \dots + h_{k,j}) (y_{j-1} - \opT y_{j-1}) - \frac{k+1}{k+2} (y_k - \opT y_k) \\
    & = y_k + \frac{1}{k+2} (y_0 - y_k) - \frac{k+1}{k+2} (y_k - \opT y_k) \\
    & = \frac{1}{k+2} y_0 + \frac{k+1}{k+2} \opT y_k .
\end{align*}
\end{proof}

\noindent
\textbf{Remark.} As a converse to \cref{corollary:anytime-optimal-algorithm-is-unique}, running any H-matrix representable algorithm for $N-1$ iterations satisfying the conditions of \cref{theorem:main-result}, and using the OHM recursion $y_{k+1} = \frac{1}{k+2} y_0 + \frac{k+1}{k+2} \opT y_k$ thereafter (for $k=N-1,N,\dots$) indeed yields an anytime-beyond-$N$-steps optimal algorithm.
This can be seen from the proof template identity~\eqref{eqn:H-symmetry-sufficiency-core-identity}, which implies
\begin{align*}
    N^2 \sqnorm{g_N} + N \inprod{g_N}{x_N - y_0} = N \left( N \sqnorm{g_N} + \inprod{g_N}{x_N - y_0} \right) \le 0 ,
\end{align*}
together with the fact that the \textit{Lyapunov function}
\begin{align*}
    V_k = k^2 \sqnorm{g_k} + k\inprod{g_k}{x_k - y_0}
\end{align*}
is nonincreasing with the OHM updates (for $k=N,N+1,\dots$), because
\begin{align}
\label{eqn:OHM-Lyapunov}
    V_k - V_{k+1} = k(k+1) \inprod{x_{k+1} - x_k}{g_{k+1} - g_k} \ge 0 .
\end{align}
This has been proved in \cite{RyuYin2022_largescale} for the pure OHM, but even in our case---where we have the OHM recursion only after $N-1$ updates---the same derivation applies, because that Lyapunov analysis does not depend on the history (previous iterates) of the algorithm besides $y_0$.

\paragraph{Potential extensions of anytime optimality notion.}
\cref{def:anytime-optimal} is a natural way of defining anytime optimal algorithms, but broader definitions can be considered.
For example, while we currently require that each $y_k$ given by \eqref{eqn:H-matrix-representation} should have the optimal guarantee, 
one may allow generating an \textit{auxiliary sequence} $z_k$ from $y_k$ by operations such as extrapolation or averaging (but without additional evaluation of $\opT$), and define the algorithm as anytime optimal if it admits a choice of auxiliary sequence $\sqnorm{z_k - \opT z_k} \le \frac{4\sqnorm{y_0 - y_\star}}{(k+1)^2}$ for all $k=0,1,\dots$.
Whether OHM is still unique even with this generalized notion of anytime optimality is an interesting question which we leave for future work.
As another possible extension, one may study algorithms that attain optimal guarantees only along an increasing sequence of iteration numbers $k_0 < k_1 < \cdots$, and characterize which algorithms beyond OHM possess such \emph{quasi-anytime} optimality.

\section{Completing the characterization of optimality}
\label{section:completing-chracterization}

In the previous section, we have proved that the H-invariance conditions $P(N-1,m; H) = \frac{1}{N} \binom{N}{m+1}$ ($m=1,\dots,N-1$) are necessary for exact optimality.
In this section, we complete the proof of this paper's main result, \cref{theorem:main-result}.
As a first step, in \cref{subsection:optimality-and-lambda} we prove a partial converse (\cref{theorem:lambda-characterization}) of the previous necessity result---an algorithm satisfying H-invariance achieves exact optimality \emph{under the additional nonnegativity conditions $\lambda^\star_{k,j}(H) \ge 0$} on H-certificates.
Then in the second step, we prove \cref{theorem:negative-lambda-suboptimality} in \cref{subsection:necessity-of-nonnegative-H-certificates}: provided that H-invariance holds, the conditions $\lambda^\star_{k,j}(H) \ge 0$ are \emph{necessary} for exact optimality.
Combining these results altogether completes the proof of \cref{theorem:main-result}.

\subsection{Optimality proof under H-invariance and nonnegative H-certificates}
\label{subsection:optimality-and-lambda}

\subsubsection{$\lambda^\star(H)$ as a solution to a linear system}

We have seen in \cref{subsection:H-certificates} that if the key identity \eqref{eqn:H-symmetry-sufficiency-core-identity} is satisfied with $\lambda_{k,j} \ge 0$, then it implies the optimal rate.
To clarify, this identity indicates the equality between formal quadratic-form-like expressions involving the vector variables $g_1,\dots,g_N$.
That is, by expanding~\eqref{eqn:H-symmetry-sufficiency-core-identity} and regrouping the terms, we can rewrite \eqref{eqn:H-symmetry-sufficiency-core-identity} in the form
\begin{align}
\label{eqn:H-symmetry-sufficiency-core-identity-vector-quadratic-form}
    0 = \sum_{k=1}^N \sum_{j=1}^{k} s_{k,j}(H,\lambda) \inprod{g_k}{g_j} ,
\end{align}
and the equality is equivalent to the system of equations $s(H,\lambda) = 0$, i.e., $s_{k,j} (H,\lambda) = 0$ for all $k=1,\dots,N$ and $j=1,\dots,k$.
Below, we provide the expressions for $s_{k,j}(H,\lambda)$, which reveal that $s(H,\lambda) = 0$ is a linear system in $\lambda$.

\begin{lemma}
\label{lemma:sNj-characterization}
The $s_{k,j}(H,\lambda)$'s in \eqref{eqn:H-symmetry-sufficiency-core-identity-vector-quadratic-form} have the explicit expressions
\begin{align}
    s_{N,N}(H,\lambda) & = N - 1 - \sum_{j=1}^{N-1} \lambda_{N,j} \label{eqn:sNN-expression} \\
    s_{N,j}(H,\lambda) & = 2 \left( \lambda_{N,j} - \sum_{i=j}^{N-1} \sum_{k=1}^{i} h_{i,j} \lambda_{N,k} - \sum_{i=j}^{N-1} h_{i,j} \right), \quad j=1,\dots,N-1 \label{eqn:sNj-expression} \\
    s_{k,k}(H,\lambda) & = \sum_{i=k+1}^N \left( \sum_{j=k}^{i-1} 2 h_{j,k} - 1 \right) \lambda_{i,k} - \sum_{j=1}^{k-1} \lambda_{k,j}, \quad k = 1,\dots,N-1 \label{eqn:skk-expression} \\
    s_{k,j}(H,\lambda) & = 2 \left( \lambda_{k,j} - \sum_{n=j}^{k-1} h_{n,j} \sum_{i=1}^n \lambda_{k,i} + c_{k,j} \right), \quad j=1,\dots,k-1 \label{eqn:skj-expression} 
\end{align}    
where $c_{k,j} = \sum_{m=k+1}^N \sum_{n=k}^{m-1} \left(h_{n,j} \lambda_{m,k} + h_{n,k} \lambda_{m,j}\right)$.
\end{lemma}

\paragraph{Remark.} Note that $c_{k,j}$ depends only on $\lambda_{m,\cdot}$'s with $m>k$.

\begin{proof}
For each $x_k$ ($k=1,\dots,N$), we have
\begin{align*}
    x_k = y_{k-1} - g_k = y_0 + \sum_{i=1}^{k-1} (y_i - y_{i-1}) - g_k = y_0 - g_k - \sum_{i=1}^{k-1} \sum_{j=1}^i 2h_{i,j} g_{j} .
\end{align*}
Plugging this expression into \eqref{eqn:H-symmetry-sufficiency-core-identity} in places of $x_N, x_k, x_j$ (with appropriate choice of summation indices to avoid overlapping), expanding all inner products and gathering the coefficients of each $\inprod{g_k}{g_j}$, we obtain the desired expressions for $s_{k,j}$.
\end{proof}

The remaining step is to solve this symbolic linear system $s(H,\lambda) = 0$ for $\lambda$.
That is, we express the (unique) solution $\lambda = \lambda^\star (H)$ to $s(H,\lambda) = 0$ as an explicit function of $H$.
This is the part that requires a significant technical effort;
we provide the explicit form of $\lambda^\star(H)$ in the following, while deferring the lengthy proof details to Appendix~\ref{section:lambda-characterization-proof}.

\subsubsection{Explicit characterization of $\lambda^\star(H)$}
\label{subsubsection:lambda-characterization}

We define some additional expressions used for characterizing $\lambda^\star(H)$.
First, define
\begin{align*}
    D(N;H) = \sum_{m=0}^{N-1} (-1)^m P(N-1,m; H) .
\end{align*}
Next, we define the \textit{Q-functions}, which can be viewed as partial H-invariants, as
\begin{align*}
    Q(k,m,j; H) = \sum_{\substack{ j = j(1) \le i(1) < j(2) \le i(2) < \cdots \le i(m-1) < j(m) \le i(m) \le k}} \prod_{r=1}^m h_{i(r),j(r)} 
\end{align*}
for $k=1,\dots,N-1$ and $m,j=1,\dots,k$.
Each $Q(k,m,j; H)$ is the partial sum of terms that constitute $P(k,m; H)$ such that its smallest column index $j(1)$ equals $j$.
Note that when $j > k-m+1$, the summation defining $Q(k,m,j; H)$ becomes vacuous, and in this case we define $Q(k,m,j;H) = 0$.
For later convenience, we also set $Q(k,m,j;H) = 0$ for $m\ge k+1$.

In the following, for simpler notation, we will fix $H$ and drop the $H$-dependency within $P, Q, D, s$ and $\lambda$.

\begin{theorem}
\label{theorem:lambda-characterization}
Suppose $P(N-1, m) = \frac{1}{N} \binom{N}{m+1}$ for $m=1,\dots,N-1$.
Then the unique solution to the linear system $s(\lambda) = 0$ is given by: 
\begin{align*}
    \lambda_{N,j}^\star = \frac{1}{D(N)} \sum_{m=1}^{N-1} (-1)^{m-1} Q(N-1,m,j) = N\sum_{m=1}^{N-j} (-1)^{m-1} Q(N-1,m,j) , \quad j=1,\dots,N-1
\end{align*}
and for $k=1,\dots,N-1$,
\begin{align*}
    \lambda_{k,j}^\star & = \frac{1}{D(N)} \sum_{\ell=1}^{N-1} \sum_{m=1}^{N-1} (-1)^{\ell+m-1} \binom{\ell+m}{m} Q(N-1,\ell,j) Q(N-1,m,k) \\
    & = N\sum_{\ell=1}^{N-j} \sum_{m=1}^{N-k} (-1)^{\ell+m-1} \binom{\ell+m}{m} Q(N-1,\ell,j) Q(N-1,m,k) , \quad j=1,\dots,k-1 .
\end{align*}
\end{theorem}

In the secondary expressions for $\lambda_{k,j}^\star$ in \cref{theorem:lambda-characterization}, the scope of summation is reduced from $1,\dots,N-1$ to $1,\dots,N-j$ or $1,\dots,N-k$. This holds because $Q(N-1,m,j) = 0$ when $m+j>N$ (in that case, the corresponding summation becomes vacuous).

In the proof of \cref{theorem:lambda-characterization}, presented in Appendix~\ref{section:lambda-characterization-proof}, we first algebraically verify that the formulas for $\lambda^\star(H)$ presented above indeed solves $s(\lambda) = 0$.
This, even without the uniqueness part, already establishes that $\lambda^\star$ satisfies \eqref{eqn:H-symmetry-sufficiency-core-identity};
hence, the H-invariance identities, combined with nonnegativity of H-certificates $\lambda^\star$, imply optimality.

Then, it remains to prove uniqueness of $\lambda^\star(H)$ as a solution.
The key idea for showing this is to subdivide the system $s(\lambda) = 0$ into $s_{k,j}(\lambda) = 0$, for each fixed $k$ and $j=1,\dots,k$, and run $k$ backwards from $N$ to $1$.
We sequentially solve the partial systems $s_{k,j}(\lambda) = 0$, viewing only $\lambda_{k,1}, \dots, \lambda_{k,k-1}$ as variables and regarding $\lambda_{m,\cdot}$ with $m>k$ as constants, as their values are already determined from previous rounds.
Then we analyze the determinants of those partial systems to show that their solutions must be unique, if they are consistent.
In fact, we also need (a slightly generalized version of) this uniqueness result in \cref{subsection:necessity-of-nonnegative-H-certificates}---see \cref{lemma:lambda-are-unique-solutions} and the remark following it---so we defer the detailed, self-contained proof of uniqueness to that point.

\subsection{Necessity of nonnegative H-certificates for optimality}
\label{subsection:necessity-of-nonnegative-H-certificates}

We now present \cref{theorem:negative-lambda-suboptimality}, the final ingredient to complete characterization of optimal H-matrices.
Here we provide a proof outline capturing all essential ideas and lemmas that need to be proved, and provide all missing details in Appendix~\ref{section:H-certificate-necessity-proof}.

\begin{theorem}
\label{theorem:negative-lambda-suboptimality}
Suppose $P(N-1,m; H) = \frac{1}{N} \binom{N}{m+1}$ for $m=1,\dots,N-1$. 
If $\lambda_{i_0, j_0}^\star(H) < 0$ for some $1\le j_0 < i_0 \le N$, then there exists a nonexpansive operator $\opT \colon \reals^{N+1} \to \reals^{N+1}$ such that $\norm{y_0 - y_\star} = 1$ and $\sqnorm{y_{N-1} - \opT y_{N-1}} > \frac{4}{N^2}$.
\end{theorem}

\paragraph{Part 1: Reformulating the worst-case construction via Gram matrix.}
Suppose that we are given an H-matrix $H$ satisfying the H-invariance conditions but at least one of the H-certificates is negative.
Rather than specifying an explicit operator $\opT$, we will show that there exist the values of $y_0, y_\star \in \reals^{N+1}$ and $g_1, \dots, g_N \in \reals^{N+1}$ such that $\norm{y_0 - y_\star} = 1$, $\norm{g_N} > \frac{1}{N}$, and if $y_1, \dots, y_{N-1}$ are defined as
\begin{align}
\label{eqn:update-rule-PEP-form}
    y_i = y_{i-1} - \sum_{j=1}^i 2h_{i,j} g_j
\end{align}
for $i=1,\dots,N-1$ and $x_i = y_{i-1} - g_i$ for $i=1,\dots,N$, then 
\begin{align}
\label{eqn:monotonicity-i-j-PEP-form}
    \inprod{x_i - x_j}{g_i - g_j} \ge 0 \iff \sqnorm{y_{i-1} - y_{j-1}} - \sqnorm{(y_{i-1} - 2g_i) - (y_{j-1} - 2g_j)} \ge 0
\end{align}
and
\begin{align}
\label{eqn:monotonicity-i-star-PEP-form}
    \inprod{x_i - y_\star}{g_i} \ge 0 \iff \sqnorm{y_{i-1} - y_\star} - \sqnorm{(y_{i-1} - 2g_i) - y_\star} \ge 0
\end{align}
holds for all $i,j = 1,\dots,N$.
When this holds, one can reconstruct a nonexpansive (1-Lipschitz) operator $\opT \colon \reals^{N+1} \to \reals^{N+1}$ such that $\opT y_\star = y_\star$ and $\opT y_{i-1} = y_{i-1} - 2g_i$ for each $i=1,\dots,N$ by the classical Kirszbraun--Valentine theorem \citep{Kirszbraun1934_uber, Valentine1943_extension, Valentine1945_lipschitz}. 
Now for such an operator $\opT$, the algorithm defined by~\eqref{eqn:H-matrix-representation}, if started at $y_0$, will precisely generate the sequence $y_1,\dots,y_{N-1}$ as its outputs.
Therefore,
\begin{align*}
    \sqnorm{y_{N-1} - \opT y_{N-1}} = 4\sqnorm{g_N} > \frac{4}{N^2} = \frac{4\sqnorm{y_0 - y_\star}}{N^2}
\end{align*}
and the algorithm fails to have the exact optimal worst-case rate.

Instead of providing concrete values of $y_0, y_\star, g_1, \dots, g_N \in \reals^d$, it suffices to specify a positive semidefinite \emph{Gram matrix}
\begin{align*}
    \vG = \begin{bmatrix}
        \sqnorm{g_1} & \inprod{g_1}{g_2} & \cdots & \inprod{g_1}{g_N} & \inprod{g_1}{y_0 - y_\star} \\
        \inprod{g_2}{g_1} & \sqnorm{g_2} & \cdots & \inprod{g_2}{g_N} & \inprod{g_2}{y_0 - y_\star} \\
        \vdots & \vdots & \ddots & \vdots & \vdots \\
        \inprod{g_N}{g_1} & \inprod{g_N}{g_2} & \cdots & \sqnorm{g_N} & \inprod{g_N}{y_0 - y_\star} \\
        \inprod{y_0 - y_\star}{g_1} & \inprod{y_0 - y_\star}{g_2} & \cdots & \inprod{y_0 - y_\star}{g_N} & \sqnorm{y_0 - y_\star}
    \end{bmatrix} \in \mathbb{S}^{N+1}_+ 
\end{align*}
because $y_0 - y_\star, g_1, \dots, g_N \in \reals^{N+1}$ can be found using the Cholesky decomposition of $\vG$ (then $y_\star$ can be chosen arbitrarily, e.g., $y_\star = 0$).
The equivalent conditions for $y_0 - y_\star, g_1, \dots, g_N \in \reals^{N+1}$ constructed in this way satisfying~\eqref{eqn:monotonicity-i-j-PEP-form} and \eqref{eqn:monotonicity-i-star-PEP-form} in terms of $\vG$ respectively are: 
\begin{align}
\label{eqn:trace-nonnegative-conditions}
    \Tr (\vG\vA_{i,j}) \ge 0 \quad \text{and} \quad \Tr (\vG \vB_i) \ge 0
\end{align}
for $i=1,\dots,N$ and $j=1,\dots,i-1$, where
\begin{align*}
    \vA_{i,j} & = \frac{1}{2} \left( (\vx_i - \vx_j) (\vg_i - \vg_j)^\intercal + (\vg_i - \vg_j) (\vx_i - \vx_j)^\intercal \right) \in \mathbb{S}^{N+1} \\
    \vB_{i} & = \frac{1}{2} \left( (\vx_i - \vy_\star) (\vg_i - \vg_\star)^\intercal + (\vg_i - \vg_\star) (\vx_i - \vy_\star)^\intercal \right) \in \mathbb{S}^{N+1}
\end{align*}
with $\vy_\star = 0$, $\vg_\star = 0$, $\vy_0 = \ve_{N+1}$, $\vg_i = \ve_i$ for $i=1,\dots,N$, and
\begin{align*}
    \vy_{i+1} & = \vy_i - \sum_{j=0}^i 2h_{i+1,j+1} \vg_{j+1} \quad \text{for $i=0,\dots,N-2$} \\
    \vx_{i+1} & = \vy_i - \vg_{i+1} \quad \text{for $i=0,\dots,N-1$} .
\end{align*}
Here we note that the idea of reducing the construction of worst-case operators into a search of finite-dimensional PSD Gram matrix satisfying LMIs based on interpolation conditions, and our choice of notations for this reformulation, have been motivated by prior work \citep{DroriTeboulle2014_performance, TaylorHendrickxGlineur2017_smooth, RyuTaylorBergelingGiselsson2020_operator} that proposed computer-assisted performance estimation and proof search for first-order algorithms. 
However, we only utilize their conceptual framework for the ease of articulating the arguments, and do not perform any numerical search with this formulation.

\paragraph{Part 2: Perturbation analysis in the space of Gram matrices.}

We start with analyzing a specific Gram matrix $\vG_0$, obtained from the worst-case operator used in \cref{theorem:H-invariance-necessity}.

\begin{lemma}
\label{lemma:worst-case-gram-matrix-explicit-form}
Let $\opT = I - 2G \colon \reals^N \to \reals^N$ be the worst-case nonexpansive linear operator from the proof of \cref{theorem:H-invariance-necessity}, and let $y_0 = -\frac{1}{\sqrt{N}}G^{-1} e_1 = -\frac{1}{\sqrt{N}} (e_1 + \dots + e_N)$.
Then the Gram matrix $\vG_0$ obtained by running the algorithm with H-matrix representation~\eqref{eqn:H-matrix-representation} for the operator $\opT$ is given by
\begin{enumerate}[label=\normalfont(\alph*)]
    \item $\left( \vG_0 \right)_{i,j} = \frac{1}{N} \sum_{m=0}^{i-1} \sum_{n=0}^{j-1} (-1)^{m+n} \binom{m+n}{m} P(i-1,m) P(j-1,n)$ for $i,j = 1, \dots, N$.
    \item $\left( \vG_0 \right)_{i,N+1} = \left( \vG_0 \right)_{N+1,i} = \frac{1}{N}$ for $i = 1, \dots, N$.
    \item $\left( \vG_0 \right)_{N+1,N+1} = 1$.
\end{enumerate}
\end{lemma}

When $H$ satisfies H-invariance, then the Gram matrix $\vG_0$ from \cref{lemma:worst-case-gram-matrix-explicit-form} satisfies additional desirable properties.

\begin{lemma}
\label{lemma:worst-case-gram-matrix-properties}
Let $\vG_0$ be as in \cref{lemma:worst-case-gram-matrix-explicit-form} and suppose $P(N-1,m) = \frac{1}{N} \binom{N}{m+1}$ for $m=1,\dots,N-1$.
Then
\begin{enumerate}[label=\normalfont(\alph*)]
    \item $\Tr(\vG_0 \vA_{i,j}) = 0$ and $\Tr(\vG_0 \vB_i) = 0$ for all $i=1,\dots,N$ and $j=1,\dots,i-1$.
    \item $\left( \vG_0 \right)_{i,N} = \left( \vG_0 \right)_{N,i} = \frac{1}{N^2}$ for $i=1,\dots,N$. 
    In particular, $\det \vG_0 = 0$.
    \item For $k=2,\dots,N$, $\det \left( \vG_0 \right)_{1:k, 1:k} = \frac{1}{N^k} P(1,1)^2 \cdots P(k-1,k-1)^2 = \frac{1}{N^k} h_{1,1}^{2k-2} h_{2,2}^{2k-4} \cdots h_{k-1,k-1}^2 > 0$.
    In particular, $(\vG_0)_{1:N,1:N} \succ 0$.
\end{enumerate}
\end{lemma}

Note that in \cref{lemma:worst-case-gram-matrix-properties}(c), we implicitly use the fact that $h_{1,1} h_{2,2} \dots h_{N-1,N-1} = P(N-1,N-1) = \frac{1}{N} \binom{N}{N} = \frac{1}{N}$ so all diagonal entries $h_{j,j}$'s are nonzero.

The main idea of the proof is to show that there exists a small perturbation $\vG = \vG_0 + \epsilon \vdelta$ of $\vG_0$, where $\epsilon > 0$ and $\vdelta \in \mathbb{S}^{N+1}$ is a symmetric matrix specifying the direction of perturbation, such that  $\vG \succ 0$, $(\vG)_{N,N} > \frac{1}{N^2}$, $(\vG)_{N+1,N+1} = 1$, $\Tr(\vG \vB_i) = 0$ for $i=1,\dots,N$, $\Tr(\vG \vA_{i,j}) = 0$ for all $1 \le j < i \le N$ with $(i, j) \ne (i_0, j_0)$ and $\Tr(\vG \vA_{i_0,j_0}) > 0$.
Once this is done, the Cholesky factorization of $\vG$ will yield the values of $y_0 - y_\star$ and $g_1, \dots, g_N$ such that $\sqnorm{y_0 - y_\star} = 1$, $\sqnorm{g_N} > \frac{1}{N^2}$, and the iterates $y_i$ defined as in \eqref{eqn:update-rule-PEP-form} can be interpolated with a nonexpansive $\opT\colon \reals^{N+1} \to \reals^{N+1}$ with operator values $\opT y_i = y_i - 2g_{i+1}$ for $i=0,\dots,N-1$, which completes the proof.

Here all conditions needed for $\vG$ besides $\vG \succ 0$ are linear in $\vG$, while $\vG \succ 0$ is a nonlinear condition which is difficult to handle because $\det \vG_0 = 0$.
Therefore, we locally linearize it by taking the derivative of the determinant function at $\vG_0$.
Owing to the desirable properties of $\vG_0$ from \cref{lemma:worst-case-gram-matrix-properties}, the result is surprisingly simple.

\begin{lemma}
\label{lemma:determinant}
Under the assumptions and notations of \cref{lemma:worst-case-gram-matrix-properties}, if $\vdelta \in \mathbb{S}^{N+1}$ and $(\vdelta)_{N+1,N+1} = 0$, then
\begin{align*}
    \det (\vG_0 + t\vdelta) & = t \cdot \frac{P(1,1)^2 \cdots P(N-1,N-1)^2}{N^{N-2}} \left( (\vdelta)_{N,N} - \frac{2}{N} (\vdelta)_{N,N+1} \right) + O(t^2) \\
    & = t \cdot \frac{h_{1,1}^{2N-2} \cdots h_{N-1,N-1}^2}{N^{N-2}} \left( (\vdelta)_{N,N} - \frac{2}{N} (\vdelta)_{N,N+1} \right) + O(t^2)
\end{align*}
\end{lemma}

\Cref{lemma:determinant} shows that as long as we have $(\vdelta)_{N,N} - \frac{2}{N} (\vdelta)_{N,N+1} > 0$, by taking $\epsilon > 0$ small enough, we will have $\det(\vG_0 + \epsilon\vdelta) > 0$.
Furthermore, because $\det (\vG_0)_{1:k,1:k} > 0$ for all $k=1,\dots,N$, with small enough $\epsilon$, one will have $\det(\vG_0 + \epsilon\vdelta)_{1:k,1:k} > 0$ for all $k=1,\dots,N+1$, i.e., $\vG = \vG_0 + \epsilon\vdelta \succ 0$.
Now define the symmetric $(N+1) \times (N+1)$ matrices $\vC_N, \vD_N, \vE_N$ by
\begin{align*}
    \vC_N = \begin{bmatrix}
        \begin{array}{c|c}
             0_{N \times N} & 0_{N \times 1} \\
            \hline 
            0_{1\times N} & 1
        \end{array}
    \end{bmatrix}, 
    \quad
    \vD_N = \begin{bmatrix}
        \begin{array}{c|c}
             0_{(N-1) \times (N-1)} & 0_{(N-1) \times 2} \\
            \hline 
            0_{2\times (N-1)} &
            \begin{array}{cc}
                1 & -\frac{1}{N} \\
                -\frac{1}{N} & 0
            \end{array}
        \end{array}
    \end{bmatrix} ,
    \quad
    \vE_N = \begin{bmatrix}
        \begin{array}{c|c}
             0_{(N-1) \times (N-1)} & 0_{(N-1) \times 2} \\
            \hline 
            0_{2\times (N-1)} &
            \begin{array}{cc}
                1 & 0 \\
                0 & 0
            \end{array}
        \end{array}
    \end{bmatrix} , 
\end{align*}
Then, the proof is complete if there exists $\vdelta \in \mathbb{S}^{N+1}$ satisfying
\begin{enumerate}
    \item $\Tr(\vdelta \vA_{i,j}) = 0$ for $1\le j < i \le N$ such that $(i,j) \ne (i_0, j_0)$
    \item $\Tr(\vdelta \vA_{i_0,j_0}) > 0$
    \item $\Tr(\vdelta \vB_i) = 0$ for $i=1,\dots,N$
    \item $\Tr(\vdelta \vC_N) = 0$
    \item $\Tr(\vdelta \vD_N), \Tr(\vdelta \vE_N) > 0$.
\end{enumerate}
It remains to show how to construct such $\vdelta$.

\paragraph{Part 3: Linear algebraic tricks for constructing $\vdelta$.}

We endow the vector space $V = \mathbb{S}^{N+1}$ of $(N+1) \times (N+1)$ symmetric matrices with the inner product $\inprod{\vX}{\vY}_V = \Tr(\vX\vY)$.
We will make use of the following general result:
\begin{lemma}
\label{lemma:linear-algebra-trick}
Let $(W, \inprod{\cdot}{\cdot}_W)$ be a real inner product space, and let $u_1, \dots, u_p \in W$ and $v_1, v_2 \in W$.
Consider the linear map $L\colon \reals^{p+2} \to W$ defined by
\begin{align*}
    L(\alpha_1, \dots, \alpha_p, \beta_1, \beta_2) = \alpha_1 u_1 + \dots + \alpha_p u_p + \beta_1 v_1 + \beta_2 v_2 .
\end{align*}
Suppose that $\ker L$ is one-dimensional and is spanned by some $(a_1, \dots, a_p, b_1, b_2) \in \reals^{p+2}$ such that $b_1, b_2 > 0$.
Then there exists $w \in W$ satisfying both
\begin{enumerate}[label=\normalfont(\alph*)]
    \item $\inprod{w}{v_1} > 0$, $\inprod{w}{v_2} > 0$
    \item $\inprod{w}{u_i} \ge 0$ for all $i=1,\dots,p$
\end{enumerate}
if and only if $a_j < 0$ for some $j=1,\dots,p$.
Furthermore, when this is the case, one can choose $w$ to satisfy $\inprod{w}{u_i} = 0$ for $i = 1, \dots, j-1, j+1, \dots, p$ and $\inprod{w}{u_j} > 0$.
\end{lemma}

Finally, the proof is completed by \cref{lemma:lambda-are-unique-solutions}, which shows that the assumptions of \cref{lemma:linear-algebra-trick} are satisfied with $W = V$,
\begin{align*}
    \{u_1,\dots,u_p\} & = \left\{ \vA_{i,j} : 1\le j < i \le N \right\} \cup \left\{ \vB_i : i=1,\dots,N \right\} \cup \left\{ \vC_N \right\}
\end{align*}
and $v_1 = \vD_N$, $v_2 = \vE_N$.

\begin{lemma}
\label{lemma:lambda-are-unique-solutions}
Suppose $P(N-1,m; H) = \frac{1}{N} \binom{N}{m+1}$ for $m=1,\dots,N-1$.
The identity
\begin{align*}
    \sum_{i=2}^N \sum_{j=1}^{i-1} a_{i,j} \vA_{i,j} + \sum_{i=1}^N b_i \vB_i + c_N \vC_N + d_N \vD_N + e_N \vE_N = 0
\end{align*}
is satisfied if and only if
\begin{itemize}
    \item $d_N = e_N$
    \item $c_N = 0$
    \item $b_i = 0$ for $i=1,\dots,N-1$
    \item $b_N = \frac{2d_N}{N}$
    \item $a_{i,j} = b_N \lambda_{i,j}^\star(H)$ for $1\le j < i \le N$.
\end{itemize}
\end{lemma}

\Cref{lemma:linear-algebra-trick}, together with \cref{lemma:lambda-are-unique-solutions}, shows a clear distinction between the case where $\lambda_{i,j}^\star(H) \ge 0$ for all $i > j$ (when the algorithm is exact optimal) and the case where $\lambda_{i_0,j_0}^\star(H) < 0$ for some $i_0 > j_0$; the strategy of perturbing $\vG_0$ to violate the optimal rate only works in the latter case.

\vspace{.2cm}
\noindent
\textbf{Remark.} Once we show that $d_N = e_N$, $c_N = 0$, $b_N = \frac{2d_N}{N}$ and $b_i = 0$ for $i=1,\dots,N-1$, taking the inner product of the matrix summation in \cref{lemma:lambda-are-unique-solutions} with a generic Gram matrix reveals that the identity is equivalent to
\begin{align}
\label{eqn:matrix-identity-in-plain-equation}
    Nb_N \sqnorm{g_N} + b_N \inprod{g_N}{x_N - y_0} + \sum_{k=2}^N \sum_{j=1}^{k-1} a_{k,j} \inprod{x_k - x_j}{g_k - g_j} = 0
\end{align}
in the sense that the coefficient of every $\inprod{g_k}{g_j}$ term vanishes.
Therefore, \cref{lemma:lambda-are-unique-solutions} implies that $\lambda_{k,j}^\star(H)$ are the unique values satisfying \eqref{eqn:H-symmetry-sufficiency-core-identity}, so by proving this result we also complete the remaining uniqueness argument from \cref{theorem:lambda-characterization}.

\section{New optimal algorithms with sparse H-certificates}
\label{section:new-boundary-algorithms}

In the case $N=4$, a generic H-matrix $H \in \reals^{3\times 3}$ is parametrized by its six lower-triangular entries, and the optimal ones must satisfy the three H-invariance identities.
Using those identities to eliminate the last-row entries, one can parametrize the optimal H-matrix set $\cH_\star \subset \reals^{3\times 3}$ using $(h_{1,1}, h_{2,1}, h_{2,2}) \in \reals^3$.
The conditions $\lambda^\star_{k,j}(H) \ge 0$ characterize the boundary surfaces of $\cH_\star$ under this parametrization, which enclose a 3-dimensional solid altogether (Figure~\ref{fig:3D-optimal-family}).
We observe that this 3D solid has six \emph{vertices}, which are intersections of maximal number of boundary surfaces of the form $\lambda_{k,j}^\star(H) = 0$ (in this case, three of them).
In particular, the vertex $A$ corresponds to OHM, satisfying $\lambda^\star_{3,1} = \lambda^\star_{4,1} = \lambda^\star_{4,2} = 0$, while the vertex $B$ represents Dual-OHM, satisfying $\lambda^\star_{2,1} = \lambda^\star_{3,1} = \lambda^\star_{3,2} = 0$.
Such specification of $\lambda_{k,j}^\star(H)$'s that are zeros is the \emph{sparsity pattern} of H-certificates associated with that vertex algorithm.

\begin{figure}[t]
    \centering
    \includegraphics[width=0.4\linewidth]{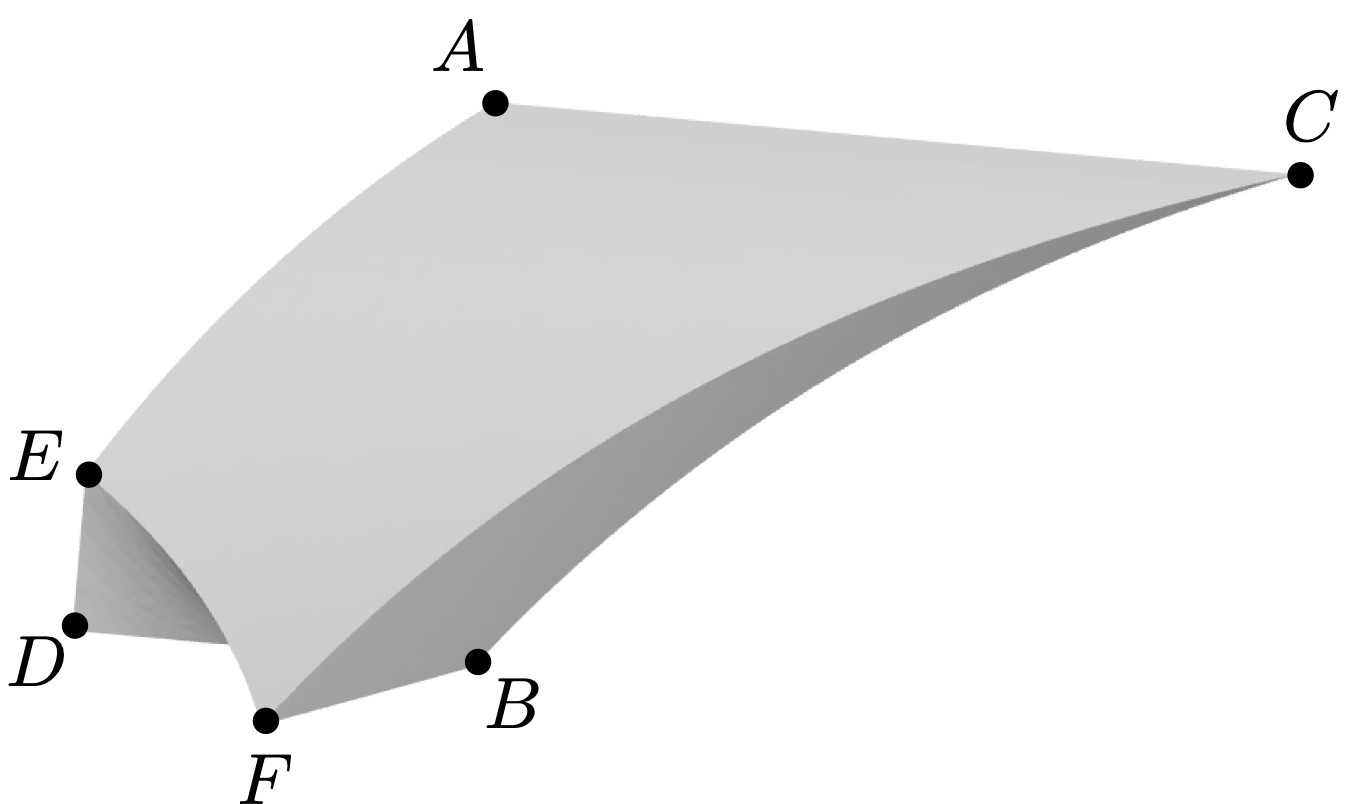}
    \caption{The optimal H-matrix set $\cH_\star$ for $N=4$, parametrized by $(h_{1,1}, h_{2,1}, h_{2,2})$, has six vertices with sparse $\lambda^\star$.}
    \label{fig:3D-optimal-family}
\end{figure}

In this section, we discover several new classes of optimal algorithms from the sparsity patterns of $\lambda^\star(H)$.
These are, by construction, algorithms whose convergence proofs use a minimal number of inequalities. 
To illustrate how a sparse $\lambda^\star(H)$, together with H-invariance, characterizes an H-matrix, we first revisit the cases of OHM and Dual-OHM.
Leveraging the insight from these cases, we then analyze some new sparsity patterns, which lead to novel algorithms.
This naturally lets us understand what the remaining vertices in Figure~\ref{fig:3D-optimal-family} other than $A$ (OHM) and $B$ (Dual-OHM) are.

\vspace{.2cm}
\noindent
Throughout this section, whenever clear from the context, we drop the dependence on $H$ from all notations.
Furthermore, we harmlessly (and implicitly) assume the condition $h_{1,1} h_{2,2} \cdots h_{N-1,N-1} \ne 0$ for all results and derivations below, as this is a necessary condition for optimality.
For cleaner exposition, we defer most proofs involving tedious computations using combinatorial identities to Appendix~\ref{section:new-algorithm-proofs}.

\subsection{Re-deriving OHM and Dual-OHM from sparsity patterns of $\lambda^\star$}

\subsubsection{Case of OHM}
\label{subsubsection:re-deriving-OHM}
The well-known convergence proof of OHM uses~\eqref{eqn:OHM-Lyapunov}---equivalently, the proof can be written in the form~\eqref{eqn:H-symmetry-sufficiency-core-identity}, where only $\lambda_{j+1,j}^\star > 0$ and all the other H-certificates are $0$.
In this section, we show that OHM can in fact be derived solely from this sparsity pattern of H-certificates, i.e., from the system of equations $\lambda_{k,j}^\star = 0$ for $j=1,\dots,N-2$ and $k=j+2,\dots,N$.

\begin{proposition}
\label{proposition:OHM-linear-equation}
For each fixed $j=1,\dots,N-2$, viewing the Q-function values $Q(N-1,1,j),\dots,Q(N-1,N-j,j)$ as variables, the quadratic system of $N-j-1$ equations
\begin{align}
\label{eqn:sparsity-pattern-OHM}
\tag{top-sparsity}
    \lambda^\star_{j+2,j} = \cdots = \lambda^\star_{N-1,j} = \lambda^\star_{N,j} = 0
\end{align}
is equivalent to a full-rank linear system~\eqref{eqn:OHM-Q-linear-system} of $N-j-1$ equations. Its (one-dimensional) solution is characterized by 
\begin{align}
\label{eqn:Q-relationship-formula-OHM}
    Q(N-1,k,j) = \binom{N-j-1}{k-1} Q(N-1,N-j,j) 
\end{align}
for $k=1,\dots,N-j-1$.
\end{proposition}

Combining \cref{proposition:OHM-linear-equation} with the H-invariance conditions
\begin{align}
\label{eqn:H-invariance-Q-sum}
    Q(N-1,m,1) + \dots + Q(N-1,m,N-m) = P(N-1,m) = \frac{1}{N} \binom{N}{m+1}
\end{align}
for $m=1,\dots,N-1$, we can uniquely determine the values of $Q(N-1,\cdot,\cdot)$ via the following procedure:
\begin{enumerate}
    \item For each $1\le j \le N-1$, assuming that the value of $Q(N-1,N-j,j)$ is known, we can also determine the values of $Q(N-1,k,j)$ for $k=1,\dots,N-j-1$ using \eqref{eqn:Q-relationship-formula-OHM}.
    Note that for the base case $j=1$, we indeed know that $Q(N-1,N-1,1) = h_{1,1}h_{2,2}\cdots h_{N-1,N-1} = P(N-1,N-1) = \frac{1}{N}$, so we can determine $Q(N-1,k,1)$ for all $k=1,\dots,N-1$.

    \item Assuming that all values of $Q(N-1,k,i)$ are found for $i=1,\dots,j$ and $k=1,\dots,N-i$, we can determine the value of $Q(N-1,N-j-1,j+1)$ using the identity \eqref{eqn:H-invariance-Q-sum} with $m=N-j-1$:
    \begin{align*}
        Q(N-1,N-j-1,j+1) = \frac{1}{N}\binom{N}{N-j} - \sum_{i=1}^{j} Q(N-1,N-j-1,i) .
    \end{align*}
    Return to Step 1, and repeat until all Q-function values are found.
\end{enumerate}

\begin{proposition}
\label{proposition:OHM-Q-function-values}
The conditions \eqref{eqn:Q-relationship-formula-OHM} from \cref{proposition:OHM-linear-equation}, together with H-invariance conditions~\eqref{eqn:H-invariance-Q-sum}, uniquely determine $Q(N-1,k,j) = \frac{j}{N}\binom{N-j-1}{k-1}$ for $j=1,\dots,N-1$ and $k=1,\dots,N-j$.
\end{proposition}

\begin{proof}
As the claimed formula trivially satisfies \eqref{eqn:Q-relationship-formula-OHM} for each $j=1,\dots,N-2$, it suffices to verify \eqref{eqn:H-invariance-Q-sum} (then these values must be the unique Q-function values that the above procedure produces).
This is an immediate application of \cref{lemma:simple-combinatorial-summations}(b) with $p=N-1, q=k-1$.
\end{proof}

Having all values of $Q(N-1,\dots,\dots)$ determined, we can uniquely recover $H$ through the following procedure: for each $k=1,\dots,N-1$, we determine the values of $h_{k,k}, h_{k+1,k}, \dots, h_{N-1,k}$ in sequential order as follows.
\begin{enumerate}
    \item First, if $k=N-1$, then $h_{N-1,N-1} = Q(N-1,1,N-1)$; otherwise $k<N-1$, so
    \begin{align*}
        h_{k,k} = \frac{h_{k,k}h_{k+1,k+1}\dots h_{N-1,N-1}}{h_{k+1,k+1}h_{k+2,k+2}\dots h_{N-1,N-1}} = \frac{Q(N-1,N-k,k)}{Q(N-1,N-k-1,k+1)} .
    \end{align*}

    \item For each $i=k+1,k+2,\dots,N-2$, we have
    \begin{align*}
        h_{i,k} = \frac{1}{Q(N-1,N-i-1,i+1)} \left[ Q(N-1,N-i,k) - \sum_{\ell=k+1}^i \left(\sum_{m=k}^{\ell-1} h_{m,k}\right) Q(N-1,N-i-1,\ell) \right] - \sum_{m=k}^{i-1} h_{m,k} ,
    \end{align*}
    which is a restatement of the second recursion in \cref{lemma:P_Q_identities}.
    Note that this formula only relies on the values of $Q(N-1,\cdot,\cdot)$ and $h_{m,k}$ with $m=k,\dots,i-1$, which are assumed to be already found.
    \item Finally, $h_{N-1,k} = Q(N-1,1,k) - \sum_{m=k}^{N-2} h_{m,k}$.
\end{enumerate}

Note that under H-invariance we have $Q(N-1,N-j,j) = h_{j,j} h_{j+1,j+1} \dots h_{N-1,N-1} \ne 0$ for any $j=1,\dots,N-1$, so there is no issue with division by these terms.
In fact, the above procedure more broadly establishes the one-to-one correspondence between matrices $H$ with nonzero diagonal entries and $\big\{Q(N-1,k,j)\big\}_{\substack{j=1,\dots,N-1\\k=1,\dots,N-j}}$ satisfying $Q(N-1,N-j,j) \ne 0$ for all $j$, even without H-invariance or exact optimality.

\begin{proposition}
\label{proposition:OHM-is-that-algorithm}
The unique optimal $H\in \reals^{(N-1) \times (N-1)}$ satisfying \eqref{eqn:sparsity-pattern-OHM} for $j=1,\dots,N-2$ is that of OHM.
\end{proposition}

Although we can certainly verify this via direct computations, for simplicity, we instead leave a brief indirect ``proof'': OHM is \textit{an} exact optimal H-matrix representable algorithm that has the sparsity pattern of \cref{proposition:OHM-linear-equation}, and our discussion above shows that it is a unique such one.

\subsubsection{Case of Dual-OHM}

We can similarly derive Dual-OHM from its $\lambda^\star$-sparsity pattern: $\lambda_{k,j}^\star = 0$ for $j=1,\dots,N-2$ and $k=j+1,\dots,N-1$, as shown in \cite{YoonKimSuhRyu2024_optimal}.

\begin{proposition}
\label{proposition:Dual-OHM-linear-equation}
For each fixed $j=1,\dots,N-2$, the system of $N-j-1$ equations
\begin{align}
\label{eqn:sparsity-pattern-Dual-OHM}
\tag{bottom-sparsity}
    \lambda^\star_{j+1,j} = \lambda^\star_{j+2,j} = \cdots = \lambda^\star_{N-1,j} = 0
\end{align}
is equivalent to a full-rank linear system~\eqref{eqn:Dual-OHM-Q-linear-system} of $N-j-1$ equations. Its (one-dimensional) solution is characterized by 
\begin{align}
\label{eqn:Q-relationship-formula-dual-OHM}
    Q(N-1,k,j) = \frac{N-j+1}{k+1}\binom{N-j-1}{k-1} Q(N-1,N-j,j) 
\end{align} 
for $k=1,\dots,N-j-1$.
\end{proposition}

\begin{proposition}
\label{proposition:Dual-OHM-is-that-algorithm-with-Qs}
The conditions \eqref{eqn:Q-relationship-formula-dual-OHM} from \cref{proposition:Dual-OHM-linear-equation} and H-invariance~\eqref{eqn:H-invariance-Q-sum} uniquely determine $Q(N-1,k,j) = \frac{1}{k+1} \binom{N-j-1}{k-1}$ for $j=1,\dots,N-1$ and $k=1,\dots,N-j$.
The unique $H$ satisfying these conditions is that of Dual-OHM.
\end{proposition}

\subsection{New algorithms from mixed sparsity patterns}

In the previous section, we discovered that the sparsity patterns \eqref{eqn:sparsity-pattern-OHM} and \eqref{eqn:sparsity-pattern-Dual-OHM} respectively reduce to explicit linear relations~\eqref{eqn:Q-relationship-formula-OHM} and \eqref{eqn:Q-relationship-formula-dual-OHM} for Q-functions.
In both of these cases, we can write
\begin{align}
\label{eqn:new-algorithm-search-ckj}
    Q(N-1,k,j) = c_{k,j} Q(N-1,N-j,j)
\end{align}
for some constants $c_{k,j}$, where $c_{k,j}$ is either $\binom{N-j-1}{k-1}$ or $\frac{N-j+1}{k+1}\binom{N-j-1}{k-1}$, depending on which sparsity pattern is chosen for each $j=1,\dots,N-2$.

In this section, we will consider the mixed sparsity patterns where we choose \eqref{eqn:sparsity-pattern-OHM} for a subset of $j$'s and \eqref{eqn:sparsity-pattern-Dual-OHM} for the rest.
Combining the resulting linear relations \eqref{eqn:new-algorithm-search-ckj} with the H-invariance condition~\eqref{eqn:H-invariance-Q-sum}, we can uniquely characterize $Q(N-1,\cdot,\cdot)$ and $H$.
We show that the H-matrices produced by this process indeed define new optimal algorithms.

\subsubsection{Selecting \eqref{eqn:sparsity-pattern-OHM} for $j=1,\dots,N'-1$ and \eqref{eqn:sparsity-pattern-Dual-OHM} for $j=N',\dots,N-1$}
\label{section:self-dual-optimal-algorithm}

In this sparsity pattern, with a fixed $2\le N'\le N-2$, we have $\lambda^\star_{j+1,j}\ne 0$ for $j=1,\dots,N'-1$, $\lambda^\star_{N,j}\ne 0$ for $j=N',\dots,N-1$, and all the other H-certificates are $0$.

\begin{proposition}
\label{proposition:self-dual-family-Q-functions}
The unique set of values of the Q-functions $Q(N-1,\cdot,\cdot)$ satisfying \eqref{eqn:sparsity-pattern-OHM} for $j=1,\dots,N'-1$, \eqref{eqn:sparsity-pattern-Dual-OHM} for $j=N',\dots,N-1$ and the H-invariance condition~\eqref{eqn:H-invariance-Q-sum} is given by
\begin{align}
\label{eqn:Q-formula-self-dual-family}
    Q(N-1,k,j) = \begin{cases}
        \frac{j}{N}\binom{N-j-1}{k-1} & \text{for } j=1,\dots,N'-1 \\
        \frac{N'(N-N'+1)}{N(k+1)}\binom{N-N'-1}{k-1} & \text{for } j=N' \\
        \frac{1}{k+1}\binom{N-j-1}{k-1} & \text{for } j=N'+1,\dots,N-1 .
    \end{cases}
\end{align}
where $k=1,\dots,N-j$. With these Q-functions, we have  $\lambda^\star_{j+1,j} > 0$ for $j=1,\dots,N'-1$ and $\lambda^\star_{N,j} > 0$ for $j=N',\dots,N-1$, so the corresponding $H$ represents a minimax optimal algorithm.
\end{proposition}

Note that in \cref{proposition:self-dual-family-Q-functions}, we have verified the optimality of $H$ using the information of its Q-functions only, without even computing $H$ explicitly.
However, the precise form of $H$ is of interest on its own, which we discuss next.
For convenience, denote the H-matrix of OHM up to $k$ steps by $H_{\text{OHM}}(k) \in \reals^{k\times k}$ and its anti-diagonal transpose (H-dual) by $H_{\text{OHM}}^\at (k) \in \reals^{k\times k}$.

\begin{proposition}
\label{proposition:self-dual-family-H-matrix}
The H-matrix corresponding to the Q-function profile \eqref{eqn:Q-formula-self-dual-family} is given by
\begin{align}
\label{eqn:H-matrix-self-dual-family}
    H = \begin{bmatrix}
        \begin{array}{c|cc}
             H_{\text{OHM}}(N'-1) & \\ \begin{array}{ccc} h_{N',1} & \cdots & h_{N',N'-1} \end{array} & h_{N',N'} \\
            \hline 
            0_{(N-N'-1) \times (N'-1)} & \begin{array}{c} h_{N'+1,N'} \\ \vdots \\ h_{N-1,N'} \end{array} & H_{\text{OHM}}^\at(N-N'-1)
        \end{array}
    \end{bmatrix} 
\end{align}
where $h_{N',N'} = \frac{N'(N-N')}{N}$ and
\begin{align*}
\begin{cases}
    h_{N',j} = j\left(\frac{1}{N} - \frac{1}{N'}\right) & \text{for } j=1,\dots,N'-1 \\
    h_{k,N'} = (N-k)\left(\frac{1}{N} - \frac{1}{N-N'}\right) & \text{for } k=N'+1,\dots,N-1 .
\end{cases}
\end{align*}
\end{proposition}

The characterization \eqref{eqn:H-matrix-self-dual-family} shows that in particular, when $N$ is even and $N' = \frac{N}{2}$, we have $H=H^\at$, i.e., the corresponding algorithm is \emph{self-dual} with respect to the H-dual operation.

\subsubsection{Selecting \eqref{eqn:sparsity-pattern-Dual-OHM} for $j=1,\dots,N'-1$ and \eqref{eqn:sparsity-pattern-OHM} for $j=N',\dots,N-1$}
\label{section:second-mixture-optimal-algorithm}

In this sparsity pattern, with a fixed $2\le N'\le N-2$, we have $\lambda^\star_{N,j}\ne 0$ for $j=1,\dots,N'-1$, $\lambda^\star_{j+1,j}\ne 0$ for $j=N',\dots,N-1$, and all the other H-certificates are $0$.

\begin{proposition}
\label{proposition:second-mixture-family-Q-functions}
The unique set of values of the Q-functions $Q(N-1,\cdot,\cdot)$ satisfying \eqref{eqn:sparsity-pattern-Dual-OHM} for $j=1,\dots,N'-1$, \eqref{eqn:sparsity-pattern-OHM} for $j=N',\dots,N-1$ and the H-invariance condition~\eqref{eqn:H-invariance-Q-sum} is given by
\begin{align}
\label{eqn:Q-formula-second-mixture-family}
     Q(N-1,k,j) = \begin{cases}
        \frac{1}{k+1} \binom{N-j-1}{k-1} & \text{for } j=1,\dots,N'-1 \\
        \frac{j-N'+1}{N-N'+1} \binom{N-j-1}{k-1} & \text{for } j=N',\dots,N-1 .
    \end{cases}
\end{align}
where $k=1,\dots,N-j$. With these Q-functions, we have  $\lambda^\star_{N,j} > 0$ for $j=1,\dots,N'-1$ and $\lambda^\star_{j+1,j} > 0$ for $j=N',\dots,N-1$, so the corresponding $H$ represents a minimax optimal algorithm.
\end{proposition}

As before, we provide the explicit expression for the associated $H$ below.

\begin{proposition}
\label{proposition:second-mixture-family-H-matrix}
The H-matrix corresponding to the Q-function profile \eqref{eqn:Q-formula-second-mixture-family} is given by
\begin{align}
\label{eqn:H-matrix-second-mixture-family}
    H = \begin{bmatrix}
        \begin{array}{c|cc}
             \left[ H_{\text{Dual-OHM}}(N-1) \right]_{1:N'-1,1:N'-1} \\
            \hline 
            \begin{array}{ccc} 
            -c_1 & \cdots & -c_{N'-1} \\
            \vdots & & \vdots \\
            -c_1 & \cdots & -c_{N'-1}
            \end{array} &  H_{\text{OHM}}(N-N')
        \end{array}
    \end{bmatrix} 
\end{align}
where $\left[ H_{\text{Dual-OHM}}(N-1) \right]_{1:N'-1,1:N'-1}$ denotes the upper left $(N'-1)\times (N'-1)$ submatrix of Dual-OHM's H-matrix (with $N-1$ iterations), and $h_{N',j} = \cdots = h_{N-1,j} = -c_j = -\frac{N-N'+1}{2(N-j)(N-j+1)}$ for $j=1,\dots,N'-1$.
In analytical form, 
\begin{align*}
    h_{k,j} = \begin{cases}
        \frac{N-k}{N-k+1} & \text{for } j = k \le N'-1 \\
        -\frac{N-k}{(N-j)(N-j+1)} & \text{for } j < k \le N'-1 \\
        \frac{k-N'+1}{k-N'+2} & \text{for } j = k \ge N' \\
        -\frac{j-N'+1}{(k-N'+1)(k-N'+2)} & \text{for } N' \le j < k \\
        -\frac{N-N'+1}{2(N-j)(N-j+1)} & \text{for }k \ge N', j \le N'-1 .
    \end{cases}
\end{align*}
\end{proposition}

\vspace{.1in}
\noindent
\textbf{Remark.} The H-dual of \eqref{eqn:H-matrix-second-mixture-family} is:
\begin{align}
\label{eqn:H-matrix-Dual-OHM-then-OHM}
    H^\at = \begin{bmatrix}
        \begin{array}{c|cc}
             H_{\text{Dual-OHM}}(N-N') \\
            \hline 
            \begin{array}{ccc} 
            -c_{N'-1} & \cdots & -c_{N'-1} \\
            \vdots & & \vdots \\
            -c_1 & \cdots & -c_1
            \end{array} &  \left[H_{\text{OHM}}(N-1)\right]_{N-N'+1:N-1,N-N'+1:N-1}
        \end{array}
    \end{bmatrix} 
\end{align}
where $\left[H_{\text{OHM}}(N-1)\right]_{N-N'+1:N-1,N-N'+1:N-1}$ denotes the lower right $(N'-1)\times (N'-1)$ submatrix of OHM's H-matrix with $N-1$ iterations. 
Interestingly, this is the H-matrix of the anytime-beyond-$(N-N')$-steps optimal algorithm that runs Dual-OHM for $N-N'$ iterations and then proceeds with the OHM recursion from there on (we do not formally prove this fact, but it is not difficult to verify through direct calculations).
With this observation, one may be inclined to conjecture that the H-dual of any optimal H-matrix representable algorithm is also optimal---e.g., the H-invariance conditions are preserved under the H-dual operation because H-invariants remain unchanged under the anti-diagonal transpose \citep{YoonKimSuhRyu2024_optimal}.
However, we show in the following Section~\ref{subsection:non-dual-optimal-optimal-algorithm} that this is \textit{not} the case because nonnegativity of H-certificates is not necessarily preserved under the H-dual.

\subsection{An optimal algorithm whose H-dual is not optimal}
\label{subsection:non-dual-optimal-optimal-algorithm}

Based on the results developed earlier in this section, we can enumerate the algorithms and their associated sparsity patterns of $\lambda^\star$ corresponding to vertices of the 3D volume $\cH_\star$ shown in Figure~\ref{fig:3D-optimal-family}:

\begin{itemize}
    \item Vertex $A$: OHM with $N=4$, $\lambda^\star_{3,1} = \lambda^\star_{4,1} = \lambda^\star_{4,2} = 0$.
    \item Vertex $B$: Dual-OHM with $N=4$, $\lambda^\star_{2,1} = \lambda^\star_{3,1} = \lambda^\star_{3,2} = 0$.
    \item Vertex $C$: Algorithm of \cref{section:self-dual-optimal-algorithm} with $N=4, N'=2$ (which is self-H-dual), $\lambda^\star_{3,1} = \lambda^\star_{3,2} = \lambda^\star_{4,1} = 0$.
    \item Vertex $D$: Algorithm of \cref{section:second-mixture-optimal-algorithm} with $N=4, N'=2$, $\lambda^\star_{2,1} = \lambda^\star_{3,1} = \lambda^\star_{4,2} = 0$.
    \item Vertex $E$: Algorithm~\eqref{eqn:H-matrix-Dual-OHM-then-OHM} with $N=4, N'=2$ (Dual-OHM augmented by OHM), $\lambda^\star_{2,1} = \lambda^\star_{4,1} = \lambda^\star_{4,2} = 0$.
\end{itemize}
However, $\cH_\star$ still has one last vertex $F$ that is not a member of this list, which corresponds to the sparsity pattern 
\begin{align}
\label{eqn:non-dual-optimal-sparsity-pattern}
    \lambda^\star_{2,1} = \lambda^\star_{3,2} = \lambda^\star_{4,1} = 0 .
\end{align}
We explicitly compute the associated values of $Q(3,\cdot,\cdot)$ by solving \eqref{eqn:non-dual-optimal-sparsity-pattern} together with the H-invariance conditions, as
\begin{align*}
    Q(3,1,1) = \frac{5}{12}, \quad Q(3,1,2) = \frac{1}{2}, \quad Q(3,1,3) = \frac{7}{12}, \quad Q(3,2,1) = \frac{2}{3}, \quad Q(3,2,2) =\frac{1}{3}, \quad Q(3,3,1) = \frac{1}{4} 
\end{align*}
and identify the corresponding H-matrix as
\begin{align*}
    H_\text{strange} = \begin{bmatrix}
        \frac{3}{4} \\
        -\frac{1}{4} & \frac{4}{7} \\
        -\frac{1}{12} & -\frac{1}{14} & \frac{7}{12}
    \end{bmatrix} .
\end{align*}
We call this algorithm the \textit{strange 3-iterate algorithm}.
$H_\text{strange}$ is indeed optimal, as we can compute its H-certificates as
\begin{align*}
    & \lambda^\star_{4,3} = 4Q(3,1,3) = \frac{7}{3} \\
    & \lambda^\star_{4,2} = 4(Q(3, 1, 2) - Q(3, 2, 2)) = \frac{2}{3} \\
    & \lambda^\star_{3,1} = 4Q(3,1,3)(-2Q(3,1,1) + 3Q(3,2,1) - 4Q(3,3,1)) = \frac{7}{18} 
\end{align*}
and the remaining ones are $0$ by design \eqref{eqn:non-dual-optimal-sparsity-pattern}, which are all nonnegative.

Now consider its H-dual
\begin{align*}
    H_\text{strange}^\at = \begin{bmatrix}
        \frac{7}{12} \\
        -\frac{1}{14} & \frac{4}{7} \\
        -\frac{1}{12} & -\frac{1}{4} & \frac{3}{4}
    \end{bmatrix} .
\end{align*}
This matrix satisfies the H-invariance conditions, but its Q-functions $\Tilde{Q}(3,\cdot,\cdot) = Q\left(3,\cdot,\cdot; H_\text{strange}^\at \right)$ are as follows:
\begin{align*}
    \Tilde{Q}(3,1,1) = \frac{3}{7}, \quad \Tilde{Q}(3,1,2) = \frac{9}{28}, \quad \Tilde{Q}(3,1,3) = \frac{3}{4}, \quad \Tilde{Q}(3,2,1) = \frac{4}{7}, \quad \Tilde{Q}(3,2,2) =\frac{3}{7}, \quad \Tilde{Q}(3,3,1) = \frac{1}{4} .
\end{align*}
From this we can deduce that $H_\text{strange}^\at$ is not optimal---e.g., we have $\lambda^\star_{4,2}\left( H_\text{strange}^\at \right) = 4(\Tilde{Q}(3,1,2) - \Tilde{Q}(3,2,2)) = -\frac{3}{7} < 0$.

\section{Conclusion}
We provide a complete characterization of the minimax optimal fixed-point algorithms admitting H-matrix representations via H-invariants and H-certificates.
In our view, this translates  
the study of optimal optimization algorithms into a study of mathematical structures, providing a holistic understanding of them beyond analyzing each algorithm’s convergence in isolation.
For fixed-point algorithms, our study raises interesting questions including characterizing optimal algorithms that remain optimal under the H-dual, classifying all extremal elements---vertices---of the optimal H-matrix set, or identifying operations beyond the H-dual that induce interesting symmetry among optimal algorithms.

We believe similar lines of thought can be extended beyond the problem setup considered in this work.
For instance, one may broaden the algorithm class to include, e.g., an adaptive (yet still minimax optimal) variant of OHM proposed in \cite[Theorem~7.2]{SuhParkRyu2023_continuoustime}.
Nonsmooth convex minimization with bounded subgradients is another problem setup where exact minimax optimal algorithms are known to be non-unique \citep{Nesterov2004_introductory, DroriTeboulle2016_optimal, DroriTaylor2020_efficient}, and we expect an analogous characterization of optimal algorithms to be achievable there.
For smooth convex minimization, exact optimality is achieved by OGM \citep{KimFessler2016_optimized, Drori2017_exact}, and similar results are achieved under strong convexity by ITEM \citep{TaylorDrori2023_optimal} or for composite minimization with an additional proximable term by OptISTA \citep{JangGuptaRyu2025_computerassisted}.
To the best of our knowledge, it is unknown whether these algorithms are uniquely optimal for their respective problem classes---interestingly, for the case of OptISTA, it is reported in \cite{JangGuptaRyu2025_computerassisted} that numerical evidence suggests that the optimal algorithm is not unique. 
Careful consideration of worst-case instances, as in our work, may pave the way toward answering these questions.
For smooth minimax optimization problems, up-to-constant optimal acceleration of gradient-norm reduction has been achieved \citep{YoonRyu2021_accelerated, LeeKim2021_fast, Tran-DinhLuo2021_halperntype}, and although exact optimality is not yet established, there seems to be a strong connection between (proximal) fixed-point and minimax optimization algorithms \citep{yoonAcceleratedMinimaxAlgorithms2025a, MokhtariOzdaglarPattathil2020_unified}; hence, we anticipate that our theory may extend to that setting as well.
Finally, we raise the question of whether our theory can be combined with the recently proposed concept of dynamic optimality based on the notion of subgame perfect equilibrium, which is stronger than minimax optimality \citep{grimmerMinimaxOptimalitySubgame2026}.

\section*{Statements and declarations}

\paragraph{Funding. }
Ernest K.\ Ryu’s contribution to this work was supported by the Air Force Office of Scientific Research under award number FA95502510183.
Benjamin Grimmer's work was supported as a fellow of the Alfred P.\ Sloan Foundation.

\paragraph{Competing interests. }
The authors have no relevant financial or non-financial interests to disclose.

\paragraph{Availability of data and materials.}
No datasets were generated or analyzed during the current study.

\bibliographystyle{unsrtnat}
\bibliography{ref}

\appendix

\section{Useful combinatorial lemmas}
\label{section:combinatorial-lemmas}

\begin{lemma}[Chu-Vandermonde]
\label{lemma:chu-vandermonde}
For $a,b \in \reals$ and a nonnegative integer $c$, 
\begin{align*}
    \sum_{i=0}^c \binom{a}{i} \binom{b}{c-i} = \binom{a+b}{c} .
\end{align*}
\end{lemma}

\begin{proof}
Both sides represent the coefficient of $x^c$ in the product of the binomial series expansions of $(1+x)^a$ and $(1+x)^b$.
\end{proof}

\begin{lemma}
\label{lemma:DN-under-H-invariance}
Suppose $P(N-1,m) = \frac{1}{N} \binom{N}{m+1}$ for $m=0,\dots,N-1$. Then $D(N) = \frac{1}{N}$.
\end{lemma}

\begin{proof}
Observe that $\sum_{m=0}^N (-1)^m \binom{N}{m} = (1-1)^N = 0$. Therefore,
\begin{align*}
    D(N) = \frac{1}{N} \sum_{m=0}^{N-1} (-1)^m \binom{N}{m+1} = \frac{1}{N} \left( 1 - \sum_{m=0}^N (-1)^m \binom{N}{m} \right) = \frac{1}{N} .
\end{align*}
\end{proof}

\begin{lemma}
\label{lemma:generalized-hockeystick}
Let $q,r$ be nonnegative integers and let $p$ be an integer satisfying $p \ge q+r$. Then 
\begin{align*}
    \sum_{j=r}^{p-q} \binom{p-j}{q} \binom{j}{r} = \binom{p+1}{q+r+1} .
\end{align*}
\end{lemma}

\begin{proof}
Note that $\binom{p-j}{q} = (-1)^{p-q-j}\binom{-q-1}{p-q-j}$ and $\binom{j}{r} = (-1)^{j-r} \binom{-r-1}{j-r}$.
Using \cref{lemma:chu-vandermonde} with $a=-q-1$, $b=-r-1$ and $c=p-q-r$ we have
\begin{align*}
    \sum_{i=0}^{p-q-r} \binom{-q-1}{i} \binom{-r-1}{p-q-r-i} = \binom{-q-r-2}{p-q-r} = (-1)^{p-q-r} \binom{p+1}{q+r+1} .
\end{align*}
The conclusion follows by making the index change $j=p-q-i$.
\end{proof}

\begin{lemma}
\label{lemma:simple-combinatorial-summations}
Let $q$ be a nonnegative integer. Then
\begin{enumerate}[label=(\alph*)]
    \item For any integers $p\ge s \ge q$, we have $\sum_{i=s}^p \binom{i}{q} = \binom{p+1}{q+1} - \binom{s}{q+1}$. In particular, $\sum_{i=q}^p \binom{i}{q} = \binom{p+1}{q+1}$.
    
    \item For any integers $p\ge q+1$ and $1\le s \le p-q$, 
    \begin{align*}
        \sum_{j=1}^s j\binom{p-j}{q} & = \binom{p+1}{q+2} - s \binom{p-s}{q+1} - \binom{p-s+1}{q+2} \\
        \sum_{j=s+1}^{p-q} j\binom{p-j}{q} & = s \binom{p-s}{q+1} + \binom{p-s+1}{q+2} .
    \end{align*}
    In particular, $\sum_{j=1}^{p-q} j\binom{p-j}{q} = \binom{p+1}{q+2}$.

    \item For any integer $p \ge q$, we have $\sum_{i=q}^p (i+1) \binom{i}{q} = (q+1) \binom{p+2}{q+2}$.
\end{enumerate}
\end{lemma}

\begin{proof}
(a) Setting $r=0$ in \cref{lemma:generalized-hockeystick}, we have $\sum_{j=0}^{p-q} \binom{p-j}{q} = \sum_{i=q}^p \binom{i}{q} = \binom{p+1}{q+1}$ where we make the index change $i=p-j$. Replacing $p$ with $s-1$ gives $\sum_{i=q}^{s-1} \binom{i}{q} = \binom{s}{q+1}$, with the understanding that when $s=q$, the summation is vacuous and therefore is $0$. Subtracting the two identities yields $\sum_{i=s}^p \binom{i}{q} = \binom{p+1}{q+1} - \binom{s}{q+1}$.

\vspace{.1in}
\noindent
(b) Setting $r=1$ in \cref{lemma:generalized-hockeystick}, we have $\sum_{j=1}^{p-q} j\binom{p-j}{q} = \binom{p+1}{q+2}$.
Replacing $p$ with $p-s$ gives $\sum_{j=1}^{p-q-s} j \binom{p-s-j}{q} = \sum_{i=s+1}^{p-q} (i-s) \binom{p-i}{q} = \binom{p-s+1}{q+2}$ (the case $s=p-q$ yields vacuous summation $0$), where we substitute $i=s+j$.
This implies $\sum_{j=s+1}^{p-q} j \binom{p-j}{q} = s \sum_{j=s+1}^{p-q} \binom{p-j}{q} + \sum_{j=s+1}^{p-q} (j-s) \binom{p-j}{q} = s \sum_{i=q}^{p-s-1} \binom{i}{q} + \sum_{i=s+1}^{p-q} (i-s) \binom{p-i}{q} = s \binom{p-s}{q+1} + \binom{p-s+1}{q+2}$.
The first identity follows by subtracting this from $\sum_{j=1}^{p-q} j\binom{p-j}{q} = \binom{p+1}{q+2}$.

\vspace{.1in}
\noindent
(c) Use $(i+1)\binom{i}{q} = (i+1) \frac{i!}{q!(i-q)!} =\frac{(i+1)!}{q!(i-q)!} = (q+1) \frac{(i+1)!}{(q+1)!(i-q)!} = (q+1) \binom{i+1}{q+1}$ and apply (a) with $p \leftarrow p+1$, $q \leftarrow q+1$.
\end{proof}

\section{Proof of \cref{theorem:lambda-characterization}}
\label{section:lambda-characterization-proof}

We first establish the following handy algebraic identities.

\begin{lemma}
\label{lemma:P_Q_identities}
For $k=1,\dots,N-1$ and $m,j=1,\dots,k$, the following identities hold:
\begin{gather*}
    P(k,m) = \sum_{j=1}^k Q(k,m,j) \\
    Q(k,m+1,j) = \sum_{\ell=j+1}^{k} \left( Q(k,m,\ell) \sum_{i=j}^{\ell-1} h_{i,j} \right) = \sum_{\ell=j+1}^{k-m+1} \left( Q(k,m,\ell) \sum_{i=j}^{\ell-1} h_{i,j} \right) .
\end{gather*}    
\end{lemma}

\begin{proof}

Define the sets of sequences of indices $I(k,m)$ and $I(k,m,j)$ by
\begin{align*}
    I(k,m) & = \left\{ (i(1), j(1), \dots, i(m), j(m)) \,|\, j(r) \le i(r) < j(r+1) \le i(r+1) \,\text{ for }\, r=1,\dots,m-1, \,\, j(m) \le i(m) \le k \right\} \\
    & \subset \{1,\dots,k\}^{2m}
\end{align*}
and
\begin{align*}
    I(k,m,j) & = \left\{ (i(1), j(1), \dots, i(m), j(m)) \in I(k,m) \,|\, j(1) = j \right\} \subset \{1,\dots,k\}^{2m} .
\end{align*}
Note that $I(k,m,j)$ is possibly empty (if $m+j>k+1$).
By definition, $I(k,m) = \coprod_{j=1}^{k} I(k,m,j)$ where $\coprod$ denotes the disjoint union.
Additionally, by definition we can write
\begin{align*}
    P(k,m) & = \sum_{(i(1), j(1), \dots, i(m), j(m)) \in I(k,m)} \prod_{r=1}^m h_{i(r),j(r)} \\
    Q(k,m,j) & = \sum_{(i(1), j(1), \dots, i(m), j(m)) \in I(k,m,j)} \prod_{r=1}^m h_{i(r),j(r)}
\end{align*}
where any vacuous summation occurring in the case $I(k,m,j) = \emptyset$ is treated as zero.
The first identity
\begin{align*}
    P(k,m) = \sum_{j=1}^k Q(k,m,j) 
\end{align*}
follows directly from the fact that $I(k,m)$ is the disjoint union of $I(k,m,j)$ for $j=1,\dots,k$.

Next, observe that
\begin{align*}
    Q(k,m+1,j) & = \sum_{(i(1), j(1), \dots, i(m+1), j(m+1)) \in I(k,m+1,j)} \prod_{r=1}^{m+1} h_{i(r),j(r)} \\
    & = \sum_{\ell=j+1}^{k} \sum_{\substack{(i(1), j(1), \dots, i(m+1), j(m+1)) \in I(k,m+1,j)\\j(2)=\ell}} \prod_{r=1}^{m+1} h_{i(r),j(r)} \\
    & = \sum_{\ell=j+1}^{k} \sum_{i=j}^{\ell-1} \sum_{\substack{(i(1), j(1), \dots, i(m+1), j(m+1)) \in I(k,m+1,j)\\i(1)=i, \,\, j(2)=\ell}} \prod_{r=1}^{m+1} h_{i(r),j(r)} \\
    & = \sum_{\ell=j+1}^{k} \sum_{i=j}^{\ell-1} \sum_{\substack{(i,j, i(2), j(2), \dots, i(m+1), j(m+1)) \in I(k,m+1,j)\\j(2)=\ell}} h_{i,j} \prod_{r=2}^{m+1} h_{i(r),j(r)} \\
    & = \sum_{\ell=j+1}^{k} \left( \sum_{i=j}^{\ell-1} h_{i,j} \right) \sum_{\substack{(i(2), j(2), \dots, i(m+1), j(m+1)) \in I(k,m,\ell)}} \prod_{r=2}^{m+1} h_{i(r),j(r)}
\end{align*}
where the last line follows from the fact that the inner summation is independent of $(i(1),j(1)) = (i,j)$, and
\[
    (i,j, i(2), j(2), \dots, i(m+1), j(m+1)) \in I(k,m+1,j) \,\text{ with }\, j(2) = \ell
\]
holds if and only if 
\begin{align*}
    & \ell = j(2) \le i(2) < j(3) \le \cdots < j(m+1) \le i(m+1) \le k \\
    & \!\! \iff (i(2), j(2), \dots, i(m+1), j(m+1)) \in I(k,m,\ell) .
\end{align*}
Finally, by definition of $Q(k,m,\ell)$ we conclude that
\begin{align*}
    Q(k,m+1,j) = \sum_{\ell=j+1}^{k} \left( \sum_{i=j}^{\ell-1} h_{i,j} \right) Q(k,m,\ell) , 
\end{align*}
which proves the second identity.
\end{proof}

\begin{proof}[Proof of \cref{theorem:lambda-characterization}]

We first show that the proposed expressions for $\lambda^\star$ satisfy $s(\lambda^\star) = 0$ by plugging $\lambda^\star$ into each line of \cref{lemma:sNj-characterization} and verifying that it simplifies to zero.

\vspace{.2cm}
\noindent
\textbf{Part 1.}
$s_{N,N}(\lambda^\star) = 0$.
For this, observe that
\begin{align*}
    \sum_{j=1}^{N-1} \lambda_{N,j}^\star & = \frac{1}{D(N)} \sum_{j=1}^{N-1} \sum_{m=1}^{N-1} (-1)^{m-1} Q(N-1,m,j) = \frac{1}{D(N)} \sum_{m=1}^{N-1} (-1)^{m-1} \sum_{j=1}^{N-1} Q(N-1,m,j) \\
    & = \frac{1}{D(N)} \sum_{m=1}^{N-1} (-1)^{m-1} P(N-1,m) = \frac{1 - D(N)}{D(N)} = N - 1 .
\end{align*}
where we use the first identity from \cref{lemma:P_Q_identities} and \cref{lemma:DN-under-H-invariance} ($D(N) = \frac{1}{N}$).
This shows $s_{N,N}(\lambda^\star) = N - 1 - \sum_{j=1}^{N-1} \lambda_{N,j}^\star = 0$. 
\vspace{.2cm}

\noindent
\textbf{Part 2.}
$s_{N,j}(\lambda^\star) = 0$ for $j=1,\dots,N-1$.
For any fixed $j \in \{1,\dots,N-1\}$, we have
\begin{align*}
    \lambda_{N,j}^\star - \sum_{i=j}^{N-1} \sum_{k=1}^{i} h_{i,j} \lambda_{N,k}^\star & = \lambda_{N,j}^\star - \sum_{k=1}^{N-1} \sum_{i=\max\{k,j\}}^{N-1} h_{i,j} \lambda_{N,k}^\star \\
    & = \lambda_{N,j}^\star - \sum_{k=1}^{j} \sum_{i=j}^{N-1} h_{i,j} \lambda_{N,k}^\star - \sum_{k=j+1}^{N-1} \sum_{i=k}^{N-1} h_{i,j} \lambda_{N,k}^\star \\
    & = \lambda_{N,j}^\star - \sum_{k=1}^{j} \sum_{i=j}^{N-1} h_{i,j} \lambda_{N,k}^\star - \sum_{k=j+1}^{N-1} \sum_{i=j}^{N-1} h_{i,j} \lambda_{N,k}^\star - \sum_{k=j+1}^{N-1} \left( \sum_{i=k}^{N-1} h_{i,j} - \sum_{i=j}^{N-1} h_{i,j} \right) \lambda_{N,k}^\star \\
    & = \lambda_{N,j}^\star - \sum_{k=1}^{N-1} \sum_{i=j}^{N-1} h_{i,j} \lambda_{N,k}^\star + \sum_{k=j+1}^{N-1} \sum_{i=j}^{k-1} h_{i,j} \lambda_{N,k}^\star \\
    & = \lambda_{N,j}^\star - \sum_{k=1}^{N-1} \sum_{i=j}^{N-1} h_{i,j} \lambda_{N,k}^\star + \sum_{\ell=j+1}^{N-1} \sum_{i=j}^{\ell-1} h_{i,j} \lambda_{N,\ell}^\star \\
    & = \frac{1}{D(N)} \left[ \sum_{m=1}^{N-1} (-1)^{m-1} Q(N-1,m,j) - \left(\sum_{i=j}^{N-1} h_{i,j} \right) \sum_{k=1}^{N-1} \sum_{m=1}^{N-1} (-1)^{m-1} Q(N-1,m,k) \right. \\
    & \qquad\qquad + \left. \sum_{\ell=j+1}^{N-1} \sum_{i=j}^{\ell-1} h_{i,j} \sum_{m=1}^{N-1} (-1)^{m-1} Q(N-1,m,\ell) \right] \\
    & \stackrel{(*)}{=} \frac{1}{D(N)} \left[ \sum_{m=1}^{N-1} (-1)^{m-1} Q(N-1,m,j) - \left(\sum_{i=j}^{N-1} h_{i,j} \right) \sum_{m=1}^{N-1} (-1)^{m-1} P(N-1,m) \right. \\
    & \qquad\qquad + \left. \sum_{m=1}^{N-1} (-1)^{m-1} Q(N-1,m+1,j) \right] \\
    & = \frac{1}{D(N)} \left[ Q(N-1,1,j) + \left(\sum_{i=j}^{N-1} h_{i,j} \right) (D(N) - 1) \right] \\
    & = \sum_{i=j}^{N-1} h_{i,j} 
\end{align*}
where $(*)$ uses \cref{lemma:P_Q_identities}, and the last equality uses $Q(N-1,1,j) = \sum_{i=j}^{N-1} h_{i,j}$, which holds by definition.
This proves
\begin{align*}
    s_{N,j} (\lambda^\star) = 2 \left( \lambda_{N,j}^\star - \sum_{i=j}^{N-1} \sum_{k=1}^{i} h_{i,j} \lambda_{N,k}^\star - \sum_{i=j}^{N-1} h_{i,j} \right) = 0 .
\end{align*}

\vspace{.2cm}
\noindent
\textbf{Part 3.}
$s_{k,k}(\lambda^\star) = 0$ for $k=1,\dots,N-1$.
To show this, first observe that 
\begin{align}
    \sum_{j=1}^{k-1} \lambda_{k,j}^\star + \sum_{i=k+1}^{N} \lambda_{i,k}^\star & = \sum_{i=1}^{k-1} \lambda_{k,i}^\star + \sum_{i=k+1}^{N-1} \lambda_{i,k}^\star + \lambda_{N,k}^\star \nonumber \\
    & \begin{aligned}
        & = \frac{1}{D(N)} \sum_{i=1}^{k-1} \sum_{\ell=1}^{N-1} \sum_{m=1}^{N-1} (-1)^{\ell+m-1} \binom{\ell+m}{m} Q(N-1,\ell,i) Q(N-1,m,k) \\
        & \quad + \frac{1}{D(N)} \sum_{i=k+1}^{N-1} \sum_{m=1}^{N-1} \sum_{\ell=1}^{N-1} (-1)^{\ell+m-1} \binom{\ell+m}{\ell} Q(N-1,m,k) Q(N-1,\ell,i) \\
        & \quad + \frac{1}{D(N)} \sum_{m=1}^{N-1} (-1)^{m-1} Q(N-1,m,k) .
    \end{aligned}
    \label{eqn:lambda-theorem-skk-expansion}
\end{align}
Within the first and second summations of \eqref{eqn:lambda-theorem-skk-expansion}, the binomial coefficients $\binom{\ell+m}{m}$ and $\binom{\ell+m}{\ell}$ are the same, and $Q(N-1,m,k)$ appears as a common factor in the summand. Therefore, we can merge these summations to obtain
\begin{align}
    & \sum_{j=1}^{k-1} \lambda_{k,j}^\star + \sum_{i=k+1}^{N} \lambda_{i,k}^\star \nonumber \\
    & = \frac{1}{D(N)} \sum_{\ell=1}^{N-1} \sum_{m=1}^{N-1} (-1)^{\ell+m-1} \binom{\ell+m}{m} Q(N-1,m,k) \left( \sum_{i=1}^{k-1} Q(N-1,\ell,i) + \sum_{i=k+1}^{N-1} Q(N-1,\ell,i) \right) \nonumber \\
    & \quad + \frac{1}{D(N)} \sum_{m=1}^{N-1} (-1)^{m-1} Q(N-1,m,k) \nonumber \\
    & \begin{aligned}
        & = \frac{1}{D(N)} \sum_{\ell=1}^{N-1} \sum_{m=1}^{N-1} (-1)^{\ell+m-1} \binom{\ell+m}{m} Q(N-1,m,k) \left( P(N-1,\ell) - Q(N-1,\ell,k) \right) \\
        & \quad + \frac{1}{D(N)} \sum_{m=1}^{N-1} (-1)^{m-1} Q(N-1,m,k)
    \end{aligned}
    \label{eqn:lambda-theorem-skk-expansion-second}
\end{align}
where we use $P(N-1,\ell) = \sum_{i=1}^{N-1} Q(N-1,\ell,i)$ (from \cref{lemma:P_Q_identities}) for the last equality.
Now, we plug in the identities $P(N-1,\ell) = \frac{1}{N} \binom{N}{\ell+1}$ for $\ell=1,\dots,N-1$ to \eqref{eqn:lambda-theorem-skk-expansion-second}, which gives
\begin{align}
    & \sum_{j=1}^{k-1} \lambda_{k,j}^\star + \sum_{i=k+1}^{N} \lambda_{i,k}^\star \nonumber \\
    & = \frac{1}{D(N)} \sum_{\ell=1}^{N-1} \sum_{m=1}^{N-1} (-1)^{\ell+m-1} \frac{1}{N} \binom{\ell+m}{m} \binom{N}{\ell+1} Q(N-1,m,k) + \frac{1}{D(N)} \sum_{m=1}^{N-1} (-1)^{m-1} Q(N-1,m,k) \nonumber \\
    & \quad - \frac{1}{D(N)} \sum_{\ell=1}^{N-1} \sum_{m=1}^{N-1} (-1)^{\ell+m-1} \binom{\ell+m}{m} Q(N-1,m,k) Q(N-1,\ell,k) \nonumber \\
    & \stackrel{(**)}{=} \frac{1}{D(N)} \sum_{m=1}^{N-1} \frac{Q(N-1,m,k)}{N} \underbrace{\sum_{\ell=0}^{N-1} (-1)^{\ell+m-1} \binom{\ell+m}{m} \binom{N}{\ell+1}}_{=0} \nonumber \\
    & \quad + \frac{1}{D(N)} \sum_{\ell=1}^{N-1} \sum_{m=1}^{N-1} (-1)^{\ell+m} \binom{\ell+m}{m} Q(N-1,m,k) Q(N-1,\ell,k) \label{eqn:lambda-theorem-skk-lambda-sum-before-expansion} \\
    & = \frac{1}{D(N)} \left[ \sum_{\ell=1}^{N-1} \binom{2\ell}{\ell} Q(N-1,\ell,k)^2 + 2 \sum_{1\le\ell<m\le N-1} (-1)^{\ell+m} \binom{\ell+m}{\ell} Q(N-1,m,k) Q(N-1,\ell,k) \right] . \label{eqn:lambda-theorem-skk-lambda-sum}
\end{align}
For $(**)$ we merge the summation $\frac{1}{D(N)} \sum_{m=1}^{N-1} (-1)^{m-1} Q(N-1,m,k)$ into the double summation as the case $\ell=0$, and then $\sum_{\ell=0}^{N-1} (-1)^{\ell+m-1} \binom{\ell+m}{m} \binom{N}{\ell+1} = 0$ follows from plugging $a=-m-1$, $b=N$, $c=N-1$ into \cref{lemma:chu-vandermonde} and using $\binom{-m-1}{\ell} = (-1)^\ell \binom{\ell+m}{\ell} = (-1)^\ell \binom{\ell+m}{m}$ and $\binom{N-m-1}{N-1} = 0$ for $m=1,2,\dots$.
Next, we compute the following:
\begin{align}
    & \sum_{i=k+1}^N \left( \sum_{j=k}^{i-1} h_{j,k} \right) \lambda_{i,k}^\star \nonumber \\
    & = \sum_{i=k+1}^{N-1} \left( \sum_{j=k}^{i-1} h_{j,k} \right) \lambda_{i,k}^\star + \left( \sum_{j=k}^{N-1} h_{j,k} \right) \lambda_{N,k}^\star \nonumber \\
    & = \frac{1}{D(N)} \sum_{\ell=1}^{N-1} \sum_{m=1}^{N-1} (-1)^{\ell+m-1} \binom{\ell+m}{m} Q(N-1,\ell,k) \left[ \sum_{i=k+1}^{N-1} \left(\sum_{j=k}^{i-1} h_{j,k}\right) Q(N-1,m,i) \right] \nonumber \\
    & \quad + \frac{1}{D(N)} \left( \sum_{j=k}^{N-1} h_{j,k} \right) \sum_{\ell=1}^{N-1} (-1)^{\ell-1} Q(N-1,\ell,k) \nonumber \\
    & \begin{aligned}
        & = \frac{1}{D(N)} \sum_{\ell=1}^{N-1} \sum_{m=1}^{N-1} (-1)^{\ell+m-1} \binom{\ell+m}{m} Q(N-1,\ell,k) Q(N-1,m+1,k) \\
        & \quad + \frac{1}{D(N)} Q(N-1,1,k) \sum_{\ell=1}^{N-1} (-1)^{\ell-1} Q(N-1,\ell,k) .
    \end{aligned}
    \label{eqn:lambda-theorem-skk-expansion-third}
\end{align}
Here, for the last equality, we use the second identity of \cref{lemma:P_Q_identities} and $Q(N-1,1,k) = \sum_{j=k}^{N-1} h_{j,k}$.
Now note that the second summation in \eqref{eqn:lambda-theorem-skk-expansion-third} can be merged with the first summation as the case $m=0$, which gives
\begin{align}
    & \sum_{i=k+1}^N \left( \sum_{j=k}^{i-1} h_{j,k} \right) \lambda_{i,k}^\star \nonumber \\
    & = \frac{1}{D(N)} \sum_{\ell=1}^{N-1} \sum_{m=0}^{N-1} (-1)^{\ell+m-1} \binom{\ell+m}{m} Q(N-1,\ell,k) Q(N-1,m+1,k) \nonumber \\
    & = \frac{1}{D(N)} \sum_{\ell=1}^{N-1} \sum_{m=1}^{N-1} (-1)^{\ell+m} \binom{\ell+m-1}{m-1} Q(N-1,\ell,k) Q(N-1,m,k) \nonumber \\
    & = \frac{1}{D(N)} \left[ \sum_{\ell=1}^{N-1} \binom{2\ell - 1}{\ell} Q(N-1,\ell,k)^2 + \sum_{1\le\ell<m\le N-1} (-1)^{\ell+m} \left( \binom{\ell+m-1}{m-1} + \binom{\ell+m-1}{\ell-1} \right) Q(N-1,m,k) Q(N-1,\ell,k) \right] \nonumber \\
    & = \frac{1}{D(N)} \left[ \sum_{\ell=1}^{N-1} \binom{2\ell - 1}{\ell} Q(N-1,\ell,k)^2 + \sum_{1\le\ell<m\le N-1} (-1)^{\ell+m} \binom{\ell+m}{\ell} Q(N-1,m,k) Q(N-1,\ell,k) \right] . \label{eqn:lambda-theorem-skk-lambda-h-product-sum}
\end{align}
The third equality follows by shifting the index $m$ (replace $m$ by $m-1$) with the understanding that $Q(N-1,m+1,k) = 0$ for $m=N-1$.
The last equality uses the identity $\binom{\ell+m-1}{m-1} + \binom{\ell+m-1}{\ell-1} = \binom{\ell+m-1}{\ell} + \binom{\ell+m-1}{\ell-1} = \binom{\ell+m}{\ell}$.
Now noting that $\binom{2\ell}{\ell} = \frac{(2\ell)!}{\ell!\ell!} = 2\frac{(2\ell-1)!}{\ell!(\ell-1)!} = 2\binom{2\ell-1}{\ell}$ and comparing \eqref{eqn:lambda-theorem-skk-lambda-sum} with \eqref{eqn:lambda-theorem-skk-lambda-h-product-sum}, we obtain
\begin{align*}
    & \sum_{j=1}^{k-1} \lambda_{k,j}^\star + \sum_{i=k+1}^{N} \lambda_{i,k}^\star = 2\sum_{i=k+1}^N \left( \sum_{j=k}^{i-1} h_{j,k} \right) \lambda_{i,k}^\star \\
    & \iff s_{k,k}(\lambda^\star) = \sum_{i=k+1}^N \left( \sum_{j=k}^{i-1} 2 h_{j,k} - 1 \right) \lambda_{i,k}^\star - \sum_{j=1}^{k-1} \lambda_{k,j}^\star = 0
\end{align*}
as desired.

\vspace{.2cm}
\noindent
\textbf{Part 4.}
$s_{k,j}(\lambda^\star) = 0$ for $k=1,\dots,N-1$ and $j=1,\dots,k-1$. In full detail, this is equivalent to
\begin{align}
\label{eqn:skj-is-0}
    \lambda_{k,j}^\star - \sum_{n=j}^{k-1} h_{n,j} \sum_{i=1}^n \lambda_{k,i}^\star + \sum_{m=k+1}^N \sum_{n=k}^{m-1} \left( h_{n,j} \lambda_{m,k}^\star + h_{n,k} \lambda_{m,j}^\star \right) = 0 .
\end{align}
We first simplify the last summation $\sum_{m=k+1}^N \sum_{n=k}^{m-1} h_{n,k} \lambda_{m,j}^\star$ as follows:
\begin{align}
    & \sum_{m=k+1}^N \sum_{n=k}^{m-1} h_{n,k} \lambda_{m,j}^\star \nonumber \\
    & = \sum_{m=k+1}^{N-1} \sum_{n=k}^{m-1} h_{n,k} \lambda_{m,j}^\star + \sum_{n=k}^{N-1} h_{n,k} \lambda_{N,j}^\star \nonumber \\
    & = \frac{1}{D(N)} \sum_{m=k+1}^{N-1} \sum_{n=k}^{m-1} h_{n,k} \sum_{\ell=1}^{N-1} \sum_{i=1}^{N-1} (-1)^{\ell+i-1} \binom{\ell+i}{i} Q(N-1,\ell,j) Q(N-1,i,m) \nonumber \\
    & \quad + \frac{1}{D(N)} \sum_{n=k}^{N-1} h_{n,k} \sum_{\ell=1}^{N-1} (-1)^{\ell-1} Q(N-1,\ell,j) \nonumber \\
    & = \frac{1}{D(N)} \sum_{\ell=1}^{N-1} \sum_{i=1}^{N-1} (-1)^{\ell+i-1} \binom{\ell+i}{i} Q(N-1,\ell,j) \sum_{m=k+1}^{N-1} Q(N-1,i,m) \sum_{n=k}^{m-1} h_{n,k} \nonumber \\
    & \quad + \frac{1}{D(N)} Q(N-1,1,k) \sum_{\ell=1}^{N-1} (-1)^{\ell-1} Q(N-1,\ell,j) \nonumber \\
    & = \frac{1}{D(N)} \sum_{\ell=1}^{N-1} \sum_{i=1}^{N-1} (-1)^{\ell+i-1} \binom{\ell+i}{i} Q(N-1,\ell,j) Q(N-1,i+1,k) \nonumber \\
    & \quad + \frac{1}{D(N)} \sum_{\ell=1}^{N-1} (-1)^{\ell-1} Q(N-1,\ell,j) Q(N-1,1,k) \nonumber \\
    & = \frac{1}{D(N)} \sum_{\ell=1}^{N-1} \sum_{i=0}^{N-1} (-1)^{\ell+i-1} \binom{\ell+i}{i} Q(N-1,\ell,j) Q(N-1,i+1,k) \nonumber \\
    & = \frac{1}{D(N)} \sum_{\ell=1}^{N-1} \sum_{i=1}^{N-1} (-1)^{\ell+i} \binom{\ell+i-1}{i-1} Q(N-1,\ell,j) Q(N-1,i,k)
    \label{eqn:lambda-theorem-skj-last-term-expansion}
\end{align}
where the fourth equality uses \cref{lemma:P_Q_identities} and the last line shifts index.
Next, we handle 
\begin{align}
    -\sum_{n=j}^{k-1} h_{n,j} \sum_{i=1}^n \lambda_{k,i}^\star & = -\sum_{n=j}^{k-1} h_{n,j} \sum_{i=1}^{k-1} \lambda_{k,i}^\star + \sum_{n=j}^{k-1} h_{n,j} \sum_{i=n+1}^{k-1} \lambda_{k,i}^\star = \sum_{n=j}^{k-1} h_{n,j} \left( -\sum_{i=1}^{k-1} \lambda_{k,i}^\star \right) + \sum_{i=j+1}^{k-1} \lambda_{k,i}^\star \sum_{n=j}^{i-1} h_{n,j} 
    \label{eqn:lambda-theorem-skj-second-term-expansion}
\end{align}
where we switch the order in the last summation.
Now, we can rewrite
\begin{align*}
    -\sum_{i=1}^{k-1} \lambda_{k,i}^\star & = \sum_{i=k+1}^N \lambda_{i,k}^\star - \left( \sum_{i=1}^{k-1} \lambda_{k,i}^\star + \sum_{i=k+1}^N \lambda_{i,k}^\star \right) \\
    & = \sum_{m=k+1}^N \lambda_{m,k}^\star - \frac{1}{D(N)} \sum_{\ell=1}^{N-1} \sum_{m=1}^{N-1} (-1)^{\ell+m} \binom{\ell+m}{m} Q(N-1,m,k) Q(N-1,\ell,k)
\end{align*}
where we change the first summation index from $i$ to $m$ and use \eqref{eqn:lambda-theorem-skk-lambda-sum-before-expansion}.
Plugging this back into \eqref{eqn:lambda-theorem-skj-second-term-expansion} we obtain
\begin{align}
    & -\sum_{n=j}^{k-1} h_{n,j} \sum_{i=1}^n \lambda_{k,i}^\star \nonumber \\
    & \begin{aligned}
        & = \underbrace{\left( \sum_{n=j}^{k-1} h_{n,j} \right) \left( \sum_{m=k+1}^N \lambda_{m,k}^\star \right)}_{\mathrm{(I)}} - \underbrace{\frac{1}{D(N)} \sum_{\ell=1}^{N-1} \sum_{m=1}^{N-1} (-1)^{\ell+m} \binom{\ell+m}{m} Q(N-1,m,k) Q(N-1,\ell,k) \sum_{n=j}^{k-1} h_{n,j}}_{\mathrm{(II)}} \\
        & \quad + \underbrace{\frac{1}{D(N)} \sum_{\ell=1}^{N-1} \sum_{m=1}^{N-1} (-1)^{\ell+m-1} \binom{\ell+m}{m} Q(N-1,m,k) \sum_{i=j+1}^{k-1} Q(N-1,\ell,i) \sum_{n=j}^{i-1} h_{n,j}}_{\mathrm{(III)}}
    \end{aligned}
    \label{eqn:lambda-theorem-skj-second-term-expansion-final}
\end{align}
where $\mathrm{(III)}$ arises from plugging in the expression for $\lambda_{k,i}^\star$ into $\sum_{i=j+1}^{k-1} \lambda_{k,i}^\star \sum_{n=j}^{i-1} h_{n,j}$.
Now, we combine \eqref{eqn:lambda-theorem-skj-second-term-expansion-final} with another summation $\sum_{m=k+1}^N \sum_{n=k}^{m-1} h_{n,j} \lambda_{m,k}^\star$ that appears in $s_{k,j}(\lambda^\star)$ to obtain
\begin{align*}
    & -\sum_{n=j}^{k-1} h_{n,j} \sum_{i=1}^n \lambda_{k,i}^\star + \sum_{m=k+1}^N \sum_{n=k}^{m-1} h_{n,j} \lambda_{m,k}^\star \\
    & = \mathrm{(I) - (II) + (III)} + \sum_{m=k+1}^N \sum_{n=k}^{m-1} h_{n,j} \lambda_{m,k}^\star \\
    & = \sum_{m=k+1}^N \sum_{n=j}^{k-1} h_{n,j} \lambda_{m,k}^\star - \mathrm{(II) + (III)} + \sum_{m=k+1}^N \sum_{n=k}^{m-1} h_{n,j} \lambda_{m,k}^\star \\
    & = \mathrm{-(II) + (III)} + \sum_{m=k+1}^N \sum_{n=j}^{m-1} h_{n,j} \lambda_{m,k}^\star \\
    & = \frac{1}{D(N)} \sum_{\ell=1}^{N-1} \sum_{m=1}^{N-1} (-1)^{\ell+m-1} \binom{\ell+m}{m} Q(N-1,m,k) Q(N-1,\ell,k) \sum_{n=j}^{k-1} h_{n,j} \\
    & \quad + \frac{1}{D(N)} \sum_{\ell=1}^{N-1} \sum_{m=1}^{N-1} (-1)^{\ell+m-1} \binom{\ell+m}{m} Q(N-1,m,k) \sum_{i=j+1}^{k-1} Q(N-1,\ell,i) \sum_{n=j}^{i-1} h_{n,j} \\
    & \quad + \frac{1}{D(N)} \sum_{\ell=1}^{N-1} \sum_{i=1}^{N-1} (-1)^{\ell+i-1} \binom{\ell+i}{i} Q(N-1,i,k) \sum_{m=k+1}^{N-1} Q(N-1,\ell,m) \sum_{n=j}^{m-1} h_{n,j} + \sum_{n=j}^{N-1} h_{n,j} \lambda_{N,k}^\star
\end{align*}
where the very last line is the result of plugging in the expression for $\lambda_{m,k}^\star$ for $m=k+1,\dots,N-1$.
Now switching the roles of indices $i$ and $m$ in the last line, we can merge all double summations to obtain
\begin{align*}
    & -\sum_{n=j}^{k-1} h_{n,j} \sum_{i=1}^n \lambda_{k,i}^\star + \sum_{m=k+1}^N \sum_{n=k}^{m-1} h_{n,j} \lambda_{m,k}^\star \\
    & = \frac{1}{D(N)} \sum_{\ell=1}^{N-1} \sum_{m=1}^{N-1} (-1)^{\ell+m-1} \binom{\ell+m}{m} Q(N-1,m,k) \\
    & \qquad \qquad \qquad \qquad \cdot \left[ Q(N-1,\ell,k) \sum_{n=j}^{k-1} h_{n,j} + \sum_{i=j+1}^{k-1} Q(N-1,\ell,i) \sum_{n=j}^{i-1} h_{n,j} + \sum_{i=k+1}^{N-1} Q(N-1,\ell,i) \sum_{n=j}^{i-1} h_{n,j} \right] \\
    & \quad + \sum_{n=j}^{N-1} h_{n,j} \lambda_{N,k}^\star \\
    & = \frac{1}{D(N)} \sum_{\ell=1}^{N-1} \sum_{m=1}^{N-1} (-1)^{\ell+m-1} \binom{\ell+m}{m} Q(N-1,m,k) \sum_{i=j+1}^{N-1} Q(N-1,\ell,i) \sum_{n=j}^{i-1} h_{n,j} \\
    & \quad + \left(\sum_{n=j}^{N-1} h_{n,j}\right) \lambda_{N,k}^\star \\
    & = \frac{1}{D(N)} \sum_{\ell=1}^{N-1} \sum_{m=1}^{N-1} (-1)^{\ell+m-1} \binom{\ell+m}{m} Q(N-1,m,k) Q(N-1,\ell+1,j) \\
    & \quad + Q(N-1,1,j) \frac{1}{D(N)} \sum_{m=1}^{N-1} (-1)^{m-1} Q(N-1,m,k) \\
    & = \frac{1}{D(N)} \sum_{\ell=0}^{N-1} \sum_{m=1}^{N-1} (-1)^{\ell+m-1} \binom{\ell+m}{m} Q(N-1,m,k) Q(N-1,\ell+1,j) \\
    & = \frac{1}{D(N)} \sum_{\ell=1}^{N-1} \sum_{m=1}^{N-1} (-1)^{\ell+m} \binom{\ell+m-1}{m} Q(N-1,m,k) Q(N-1,\ell,j) 
\end{align*}
where we shift the index $\ell$ by $1$ for the last equality.
Combining the above with \eqref{eqn:lambda-theorem-skj-last-term-expansion}, we obtain
\begin{align*}
    & -\sum_{n=j}^{k-1} h_{n,j} \sum_{i=1}^n \lambda_{k,i}^\star + \sum_{m=k+1}^N \sum_{n=k}^{m-1} \left( h_{n,j} \lambda_{m,k}^\star + h_{n,k} \lambda_{m,j}^\star \right) \\
    & = \frac{1}{D(N)} \sum_{\ell=1}^{N-1} \sum_{m=1}^{N-1} (-1)^{\ell+m} \binom{\ell+m-1}{m} Q(N-1,m,k) Q(N-1,\ell,j) \\
    & \quad + \frac{1}{D(N)} \sum_{\ell=1}^{N-1} \sum_{i=1}^{N-1} (-1)^{\ell+i} \binom{\ell+i-1}{i-1} Q(N-1,\ell,j) Q(N-1,i,k) \\
    & = \frac{1}{D(N)} \sum_{\ell=1}^{N-1} \sum_{m=1}^{N-1} (-1)^{\ell+m} \binom{\ell+m-1}{m} Q(N-1,m,k) Q(N-1,\ell,j) \\
    & \quad + \frac{1}{D(N)} \sum_{\ell=1}^{N-1} \sum_{m=1}^{N-1} (-1)^{\ell+m} \binom{\ell+m-1}{m-1} Q(N-1,\ell,j) Q(N-1,m,k) \\
    & = \frac{1}{D(N)} \sum_{\ell=1}^{N-1} \sum_{m=1}^{N-1} (-1)^{\ell+m} \binom{\ell+m}{m} Q(N-1,m,k) Q(N-1,\ell,j) \\
    & = -\lambda_{k,j}^\star
\end{align*}
where the third equality changes the index $i$ to $m$ and the last equality uses $\binom{\ell+m-1}{m} + \binom{\ell+m-1}{m-1} = \binom{\ell+m}{m}$.
This proves~\eqref{eqn:skj-is-0}.

As mentioned in \ref{subsubsection:lambda-characterization}, we provide the uniqueness proof when we show \cref{theorem:negative-lambda-suboptimality}.

\end{proof}

\section{Proofs for \cref{subsection:necessity-of-nonnegative-H-certificates}}
\label{section:H-certificate-necessity-proof}

\begin{proof}[Proof of \cref{lemma:worst-case-gram-matrix-explicit-form}]
First, note that from the proof of \cref{theorem:H-invariance-necessity}, we have \eqref{eqn:necessary-condition-Gmp1y0-expression} (with $R=1$), which implies
\begin{align}
    Gy_k & = \sum_{m=0}^k (-1)^m P(k,m; 2H) G^{m+1} y_0 \nonumber \\
    & = \sum_{m=0}^k (-1)^m P(k,m; H) 2^m G^{m+1} y_0 \nonumber \\
    & = -\frac{1}{\sqrt{N}} \sum_{m=0}^k (-1)^m P(k,m; H) \sum_{j=1}^{m+1} (-1)^{j-1} \binom{m}{j-1} e_j \nonumber \\
    & = -\frac{1}{\sqrt{N}} \sum_{\ell=1}^{k+1} \left( \sum_{m=\ell-1}^k (-1)^{m+\ell-1} \binom{m}{\ell-1} P(k,m; H) \right) e_\ell 
    \label{eqn:gram-matrix-gyk-formula}
\end{align}
where the last line changes the order of summations and replace the index $j$ with $\ell$.

Now we show (a). Using \eqref{eqn:gram-matrix-gyk-formula}, for $i,j=1,\dots,N$, we compute:
\begin{align*}
    \left(\vG_0\right)_{i,j} & = \inprod{g_i}{g_j} \\
    & = \inprod{Gy_{i-1}}{Gy_{j-1}} \\
    & = \frac{1}{N} \sum_{\ell=1}^{\min\{i,j\}} \left( \sum_{m=\ell-1}^{i-1} (-1)^{m+\ell-1} \binom{m}{\ell-1} P(i-1,m) \right) \left( \sum_{n=\ell-1}^{j-1} (-1)^{n+\ell-1} \binom{n}{\ell-1} P(j-1,n) \right) \\
    & = \frac{1}{N} \sum_{\ell=1}^{\min\{i,j\}} \sum_{m=\ell-1}^{i-1} \sum_{n=\ell-1}^{j-1} (-1)^{m+n} \binom{m}{\ell-1} \binom{n}{\ell-1} P(i-1,m)  P(j-1,n) \\
    & = \frac{1}{N} \sum_{m=0}^{i-1} \sum_{n=0}^{j-1} \sum_{\ell=1}^{\min\{m+1,n+1\}} (-1)^{m+n} \binom{m}{\ell-1} \binom{n}{\ell-1} P(i-1,m) P(j-1,n) \\
    & = \frac{1}{N} \sum_{m=0}^{i-1} \sum_{n=0}^{j-1} (-1)^{m+n} P(i-1,m) P(j-1,n) \sum_{\ell=1}^{\min\{m+1,n+1\}}  \binom{m}{\ell-1} \binom{n}{\ell-1} \\
    & = \frac{1}{N} \sum_{m=0}^{i-1} \sum_{n=0}^{j-1} (-1)^{m+n} P(i-1,m) P(j-1,n) \sum_{\ell=0}^{\min\{m,n\}}  \binom{m}{\ell} \binom{n}{\ell} 
\end{align*}
where in the last line we shift the index $\ell$ by $1$.
By \cref{lemma:chu-vandermonde}, we have $\sum_{\ell=0}^{\min\{m,n\}} \binom{m}{\ell} \binom{n}{\ell} = \binom{m+n}{m}$, and using this to simplify the last expression we obtain
\begin{align*}
    \left(\vG_0\right)_{i,j} = \frac{1}{N} \sum_{m=0}^{i-1} \sum_{n=0}^{j-1} (-1)^{m+n} \binom{m+n}{m} P(i-1,m) P(j-1,n) .
\end{align*}
Next, we show (b). We have $y_0 - y_\star = y_0 = -\frac{1}{\sqrt{N}}(e_1 + \dots + e_N)$, so using \eqref{eqn:gram-matrix-gyk-formula} we have for $i=1,\dots,N$:
\begin{align*}
    \left( \vG_0 \right)_{i,N+1} & = \inprod{Gy_{i-1}}{y_0 - y_\star} \\
    & = \frac{1}{N} \sum_{\ell=1}^i \sum_{m=\ell-1}^{i-1} (-1)^{m+\ell-1} \binom{m}{\ell-1} P(i-1,m; H) \\
    & = \frac{1}{N} \sum_{m=0}^{i-1} \sum_{\ell=1}^{m+1} (-1)^{m+\ell-1} \binom{m}{\ell-1} P(i-1,m; H) \\
    & = \frac{1}{N} \sum_{m=0}^{i-1} (-1)^m P(i-1,m; H) \sum_{\ell=1}^{m+1} (-1)^{\ell-1} \binom{m}{\ell-1}
\end{align*}
where in the third line, we change the order of summations.
Now, the inner summation $\sum_{\ell=1}^{m+1} (-1)^{\ell-1} \binom{m}{\ell-1}$ in the last line is $0$ when $m>0$ and $1$ when $m=0$, and therefore, we conclude that
\begin{align*}
    \left( \vG_0 \right)_{i,N+1} = \frac{1}{N} P(i-1,0; H) = \frac{1}{N} .
\end{align*}
Finally, (c) is straightforward: $\left( \vG_0 \right)_{N+1,N+1} = \sqnorm{y_0 - y_\star} = \sqnorm{-\frac{1}{\sqrt{N}}(e_1 + \dots + e_N)} = 1$.
\end{proof}

\begin{proof}[Proof of \cref{lemma:worst-case-gram-matrix-properties}]

We first show (a). Observe that 
\begin{align*}
    \Tr(\vG_0 \vA_{i,j}) & = \inprod{x_i - x_j}{g_i - g_j} \\
    & = \inprod{y_{i-1} - y_{j-1} - (g_i - g_j)}{g_i - g_j} \\
    & = \inprod{y_{i-1} - y_{j-1}}{g_i - g_j} - \sqnorm{g_i - g_j} \\
    & = \inprod{y_{i-1} - y_{j-1}}{G(y_{i-1} - y_{j-1})} - \sqnorm{G(y_{i-1} - y_{j-1})} \\
    & = (y_{i-1} - y_{j-1})^\T G (y_{i-1} - y_{j-1}) - (y_{i-1} - y_{j-1})^\T G^\T G (y_{i-1} - y_{j-1}) .
\end{align*}
and similarly, 
\begin{align*}
    \Tr(\vG_0 \vB_i) & = \inprod{x_i - y_\star}{g_i} \\
    & = \inprod{y_{i-1} - g_i}{g_i} \\
    & = \inprod{y_{i-1}}{g_i} - \sqnorm{g_i} \\
    & = \inprod{y_{i-1}}{Gy_{i-1}} - \sqnorm{Gy_{i-1}} \\
    & = y_{i-1}^\T G y_{i-1} - y_{i-1}^\T G^\T G y_{i-1} .
\end{align*}
Now, by direct computation, we observe that
\begin{align*}
    G^\T G = \frac{1}{4} \begin{bmatrix}
        2 & -1 & 0 & \dots & 0 & 1 \\
        -1 & 2 & -1 & \dots & 0 & 0 \\
        0 & -1 & 2 & \dots & 0 & 0 \\
        \vdots & \vdots & \vdots & \ddots & \vdots & \vdots \\
        0 & 0 & 0 & \dots & 2 & -1 \\
        1 & 0 & 0 & \dots & -1 & 2
    \end{bmatrix}
    = \frac{1}{2} (G + G^\T) .
\end{align*}
Therefore, for any $y \in \reals^N$, we have
\begin{align*}
    y^\T G y = y^\T \frac{1}{2} (G + G^\T) y = y^\T G^\T G y ,
\end{align*}
which, together with the above expressions, implies that $\Tr(\vG_0 \vA_{i,j}) = 0 = \Tr(\vG_0 \vB_i)$.

For (b), we see that when $P(N-1,m) = \frac{1}{N} \binom{N}{m+1}$, the proof of \cref{theorem:H-invariance-necessity} shows that \eqref{eqn:necessary-condition-projection-identity} (with $R=1$) holds, i.e.,
\begin{align*}
    g_N = Gy_{N-1} = -\frac{1}{N\sqrt{N}}(e_1 + \dots + e_N) = \frac{1}{N} (y_0 - y_\star) .
\end{align*}
Therefore,
\begin{align*}
    \left(\vG_0\right)_{i,N} = \left(\vG_0\right)_{N,i} = \inprod{g_N}{g_i} = \frac{1}{N}\inprod{y_0 - y_\star}{g_i} = \frac{1}{N} \left(\vG_0\right)_{i,N+1} = \frac{1}{N^2}
\end{align*}
where we use \cref{lemma:worst-case-gram-matrix-explicit-form}(b) for the last equality.
In particular, this shows $\det\vG_0 = 0$ because 
\begin{align*}
    \left(\vG_0 \right)_{N,:} = \begin{bmatrix} \frac{1}{N^2} & \dots & \frac{1}{N^2} & \frac{1}{N} \end{bmatrix} = \frac{1}{N} \begin{bmatrix} \frac{1}{N} & \dots & \frac{1}{N} & 1 \end{bmatrix} = \frac{1}{N} \left(\vG_0 \right)_{N+1,:} 
\end{align*}
where $\left(\vG_0 \right)_{N,:}$ and $\left(\vG_0 \right)_{N+1,:}$ respectively denote the $N$-th and $(N+1)$-th rows of $\vG_0$.

Finally, we show (c).
Note that $\left(\vG_0\right)_{1,1} = \frac{1}{N} > 0$. Next, for $k\ge 2$, because $\left(\vG_0\right)_{1:k,1:k}$ is the Gram matrix for the set of vectors $g_1,\dots,g_k$, we have
\begin{align}
    \det \left(\vG_0\right)_{1:k,1:k} = \sqnorm{g_1 \wedge g_2 \wedge \dots \wedge g_k}
    \label{eqn:gram-determinant-using-wedge-product}
\end{align}
where $\wedge$ is the exterior product.
Now we show by induction that
\begin{align}
\label{eqn:gk-wedge-product-formula}
    g_1 \wedge \dots \wedge g_k = (-1)^{k} \frac{1}{N^{k/2}} P(1,1) \cdots P(k-1,k-1) (e_1 \wedge \dots \wedge e_k)
\end{align}
for $k=2,\dots,N$.
For the base case $k=2$, the formula~\eqref{eqn:gram-matrix-gyk-formula} gives $g_1 = Gy_0 = -\frac{1}{\sqrt{N}}e_1$ and
\[
    g_2 = Gy_1 = -\frac{1}{\sqrt{N}}\left( (1-P(1,1))e_1 + P(1,1)e_2 \right) .
\]
Therefore, using $e_1 \wedge e_1 = 0$ we have
\begin{align*}
    g_1 \wedge g_2 = \frac{1}{N} P(1,1) (e_1 \wedge e_2)
\end{align*}
Now assume that \eqref{eqn:gk-wedge-product-formula} holds for some $2\le k\le N-1$ and consider $g_1 \wedge \dots \wedge g_{k+1}$.
By \eqref{eqn:gram-matrix-gyk-formula}, we have
\begin{align*}
    g_{k+1} = Gy_k \in \spann\{e_1,\dots,e_k\} - \frac{1}{\sqrt{N}} P(k,k) e_{k+1}
\end{align*}
and combining this with the induction hypothesis, we obtain 
\begin{align*}
    g_1 \wedge \dots \wedge g_{k+1} & = \left( (-1)^{k} \frac{1}{N^{k/2}} P(1,1) \cdots P(k-1,k-1) (e_1 \wedge \dots \wedge e_k) \right) \wedge \left( -\frac{1}{\sqrt{N}} P(k,k) e_{k+1} \right) \\
    & = (-1)^{k+1} \frac{1}{N^{(k+1)/2}} P(1,1) \cdots P(k,k) (e_1 \wedge \dots \wedge e_{k+1}) 
\end{align*}
as desired, completing the induction.
Then, using \eqref{eqn:gram-determinant-using-wedge-product}, we have
\begin{align*}
    \det \left(\vG_0\right)_{1:k,1:k} & = \sqnorm{(-1)^{k} \frac{1}{N^{k/2}} P(1,1) \cdots P(k-1,k-1) (e_1 \wedge \dots \wedge e_k)} \\
    & = \frac{1}{N^k} P(1,1)^2 \cdots P(k-1,k-1)^2 \\
    & = \frac{1}{N^k} h_{1,1}^{2k-2} h_{2,2}^{2k-4} \cdots h_{k-1,k-1}^2 .
\end{align*}
As we assume $P(N-1,N-1) = h_{1,1} h_{2,2} \cdots h_{N-1,N-1} = \frac{1}{N} \binom{N}{N} = \frac{1}{N}$, none of the diagonal elements of the H-matrix are $0$, so all these leading principal minors are positive.
\end{proof}

\begin{proof}[Proof of \cref{lemma:determinant}]
By Jacobi's formula, $\frac{d(\det (\vG_0 + t\vdelta))}{dt} \big|_{t=0} = \mathrm{Tr}\left( \vdelta \cdot \adj \vG_0 \right)$, where $\mathrm{adj}$ denotes the adjugate (transpose of the cofactor matrix).
As we are assuming $(\vdelta)_{N+1,N+1} = 0$, the desired statement follows if we show that
\begin{align*}
    \adj \vG_0 = 
    \begin{bmatrix}
        0 & \cdots & 0 & 0 & 0 \\
        \vdots & \ddots & \vdots & \vdots & \vdots \\
        0 & \cdots & 0 & 0 & 0 \\
        0 & \cdots & 0 & \frac{P(1,1)^2 \cdots P(N-1,N-1)^2}{N^{N-2}} & -\frac{P(1,1)^2 \cdots P(N-1,N-1)^2}{N^{N-1}} \\
        0 & \cdots & 0 & -\frac{P(1,1)^2 \cdots P(N-1,N-1)^2}{N^{N-1}} & *
    \end{bmatrix} .
\end{align*}
First, note that $\adj \vG_0$ is symmetric because $\vG_0$ is.
Now for $i=1,\dots,N+1$, we immediately see that $\left(\adj \vG_0\right)_{i,j} = 0$ for any $j \le N-1$ because eliminating the $j$-th row and $i$-th column from $\vG_0$ preserves the last two rows, except that their $i$-th column entries are eliminated.
That is, the minor matrix corresponding to the $(j,i)$-cofactor has two linearly dependent rows in the end (see \cref{lemma:worst-case-gram-matrix-explicit-form}(b, c) and \cref{lemma:worst-case-gram-matrix-properties}(b)), and therefore its determinant is $0$.
By symmetry of $\adj \vG_0$, we also have $\left(\adj \vG_0\right)_{i,j} = 0$ if $i\le N-1$.

Now it only remains to check what $\left(\adj \vG_0\right)_{N,N}$ and $\left(\adj \vG_0\right)_{N,N+1} = \left(\adj \vG_0\right)_{N+1,N}$ are.
For this, we use the definition of cofactor and Lemmas~\ref{lemma:worst-case-gram-matrix-explicit-form}(b, c) and \ref{lemma:worst-case-gram-matrix-properties}(b) to see that
\begin{align*}
    \left(\adj \vG_0\right)_{N,N} & = (-1)^{2N} \det \begin{bmatrix}
        \begin{array}{c|c}
             \left(\vG_0\right)_{1:N-1,1:N-1} & \begin{array}{c} \frac{1}{N} \\ \vdots \\ \frac{1}{N} \end{array} \\
            \hline 
            \begin{array}{ccc} \frac{1}{N} & \cdots & \frac{1}{N} \end{array} & 1
        \end{array}
    \end{bmatrix}
    = N \det \begin{bmatrix}
        \begin{array}{c|c}
             \left(\vG_0\right)_{1:N-1,1:N-1} & \begin{array}{c} \frac{1}{N} \\ \vdots \\ \frac{1}{N} \end{array} \\
            \hline 
            \begin{array}{ccc} \frac{1}{N^2} & \cdots & \frac{1}{N^2} \end{array} & \frac{1}{N}
        \end{array}
    \end{bmatrix} \\
    & = N^2 \det \begin{bmatrix}
        \begin{array}{c|c}
             \left(\vG_0\right)_{1:N-1,1:N-1} & \begin{array}{c} \frac{1}{N^2} \\ \vdots \\ \frac{1}{N^2} \end{array} \\
            \hline 
            \begin{array}{ccc} \frac{1}{N^2} & \cdots & \frac{1}{N^2} \end{array} & \frac{1}{N^2}
        \end{array}
    \end{bmatrix} = N^2 \det \left(\vG_0\right)_{1:N,1:N} = \frac{P(1,1)^2 \cdots P(N-1,N-1)^2}{N^{N-2}}
\end{align*}
and similarly,
\begin{align*}
    \left(\adj \vG_0\right)_{N+1,N} & = (-1)^{2N+1} \det \begin{bmatrix}
        \begin{array}{c}
             \left(\vG_0\right)_{1:N-1,1:N} \\
            \hline 
            \begin{array}{ccc} \frac{1}{N} & \cdots & \frac{1}{N} \end{array} 
        \end{array}
    \end{bmatrix}
    = -N \det \begin{bmatrix}
        \begin{array}{c}
             \left(\vG_0\right)_{1:N-1,1:N} \\
            \hline 
            \begin{array}{ccc} \frac{1}{N^2} & \cdots & \frac{1}{N^2} \end{array} 
        \end{array}
    \end{bmatrix} \\
    & = -N \det \left(\vG_0\right)_{1:N,1:N} = -\frac{P(1,1)^2 \cdots P(N-1,N-1)^2}{N^{N-1}} ,
\end{align*}
as desired.
\end{proof}

\begin{proof}[Proof of \cref{lemma:linear-algebra-trick}]
First, we show that $a_j < 0$ for some $j=1,\dots,p$ is necessary for the existence of $w \in W$ satisfying both (a) and (b).
Suppose that $a_j \ge 0$ for all $j=1,\dots,p$, and suppose that there is $w \in W$ such that $\inprod{w}{v_1}, \inprod{w}{v_2} > 0$ and $\inprod{w}{u_i} \ge 0$ for $i=1,\dots,p$.
By assumption, we have
\begin{align*}
    a_1 u_1 + \dots + a_p u_p + b_1 v_1 + b_2 v_2 = 0
\end{align*}
and taking the inner product of the both sides with $w$ gives
\begin{align*}
    0 = a_1 \inprod{w}{u_1} + \dots + a_p \inprod{w}{u_p} + b_1 \inprod{w}{v_1} + b_2 \inprod{w}{v_2} > 0 
\end{align*}
because $a_j \ge 0$ and $b_1, b_2 > 0$, a contradiction.
This shows that (a) and (b) cannot both hold unless $a_j < 0$ for some $j=1,\dots,p$.

Next, suppose that at least one of $a_1,\dots,a_p$ is negative, and fix an index $j \in \{1,\dots,p\}$ for which $a_j < 0$.
Consider the subspaces
\begin{align*}
    V_1 & = \spann \{ u_1, \dots, u_{j-1}, u_{j+1}, \dots, u_p , v_1 \} \\
    V_2 & = \spann \{ u_1, \dots, u_{j-1}, u_{j+1}, \dots, u_p , v_2 \} .
\end{align*}
We show that $v_1 \notin V_2$. 
Suppose to the contrary that $v_1 \in V_2$. Then there are $\alpha_1, \dots, \alpha_{j-1}, \alpha_{j+1}, \dots, \alpha_p, \beta_2 \in \reals$ such that
\begin{align*}
    \alpha_1 u_1 + \dots + \alpha_{j-1} u_{j-1} + \alpha_{j+1} u_{j+1} + \dots + \alpha_p u_p + v_1 + \beta_2 v_2 = 0 . 
\end{align*}
This implies
\begin{align*}
    \vx = (\alpha_1, \dots, \alpha_{j-1}, 0, \alpha_{j+1}, \dots, \alpha_p, 1, \beta_2) \in \ker L = \spann \{(a_1, \dots, a_{j-1}, a_j, a_{j+1}, \dots, a_p, b_1, b_2)\}
\end{align*}
and because $a_j \ne 0$, the only way $\vx$ can have $0$ in the $j$-th entry is $\vx = 0$, which is a contradiction because its $(p+1)$-th entry is $1$.
This proves that $v_1 \notin V_2$. Similarly, we have $v_2 \notin V_1$.

Now let $w = \proj_{V_2^\perp} (v_1) + \proj_{V_1^\perp} (v_2)$.
We show $w$ satisfies $\inprod{w}{v_1} > 0, \inprod{w}{v_2} > 0$, $\inprod{w}{u_i} = 0$ for $i=1,\dots,j-1,j+1,\dots,p$, and $\inprod{w}{u_j} > 0$, which in particular implies (a) and (b).
Indeed,
\begin{align*}
    \inprod{w}{v_1} = \inprod{\proj_{V_2^\perp} (v_1)}{v_1} + \inprod{\proj_{V_1^\perp} (v_2)}{v_1} = \inprod{\proj_{V_2^\perp} (v_1)}{v_1} = \inprod{\proj_{V_2^\perp} (v_1)}{\proj_{V_2} (v_1) + \proj_{V_2^\perp} (v_1)} = \sqnorm{\proj_{V_2^\perp} (v_1)} > 0 ,
\end{align*}
where we use $v_1 \in V_1 \implies V_1^\perp \perp v_1$ and $\proj_{V_2^\perp} (v_1) \ne 0$ because $v_1 \notin V_2$.
Symmetrically, we also have $\inprod{w}{v_2} > 0$.
Additionally, for $i = 1, \dots, j-1, j+1, \dots, p$, both $V_1^\perp$ and $V_2^\perp$ are orthogonal to $u_i$, and therefore $\inprod{w}{u_i} = 0$.
Then we have
\begin{align*}
    0 = a_1 \inprod{w}{u_1} + \dots + a_p \inprod{w}{u_p} + b_1 \inprod{w}{v_1} + b_2 \inprod{w}{v_2} = a_j \inprod{w}{u_j} + b_1 \inprod{w}{v_1} + b_2 \inprod{w}{v_2} 
\end{align*}
which implies $\inprod{w}{u_j} = -\frac{1}{a_j} \left( b_1 \inprod{w}{v_1} + b_2 \inprod{w}{v_2} \right) > 0$.
\end{proof}

\begin{proof}[Proof of \cref{lemma:lambda-are-unique-solutions}]

By recursively using $\vy_{i+1} = \vy_i - \sum_{j=0}^i 2h_{i+1,j+1} \vg_{j+1}$ we have for $i=0,\dots,N-1$,
\begin{align*}
    \vy_i = \vy_0 - \sum_{k=1}^{i} \sum_{j=0}^{k-1} 2h_{k,j+1} \vg_{j+1} = \ve_{N+1} - \sum_{k=1}^{i} \sum_{j=1}^k 2h_{k,j} \ve_j = \ve_{N+1} - \sum_{j=1}^i \left(\sum_{k=j}^i 2h_{k,j}\right) \ve_j
\end{align*}
and 
\begin{align*}
    \vx_{i+1} = \vy_i - \vg_{i+1} = \ve_{N+1} - \sum_{j=1}^i \left(\sum_{k=j}^i 2h_{k,j}\right) \ve_j - \ve_{i+1} .
\end{align*}
This shows that for any $i=2,\dots,N$ and $j=1,\dots,i-1$, 
\begin{align*}
    \vx_i - \vx_j \in \spann\{\ve_1, \dots, \ve_i\}
\end{align*}
and hence, 
\begin{align*}
    \vA_{i,j} & = \frac{1}{2} \left( (\vx_i - \vx_j) (\vg_i - \vg_j)^\T + (\vg_i - \vg_j) (\vx_i - \vx_j)^\T \right) \\
    & = \frac{1}{2} \left( (\vx_i - \vx_j) (\ve_i - \ve_j)^\T + (\ve_i - \ve_j) (\vx_i - \vx_j)^\T \right) \\
    & = \begin{bmatrix}
        \begin{array}{c|c}
             * & 0_{i \times (N-i+1)} \\
            \hline 
            0_{(N-i+1)\times i} & 0_{(N-i+1) \times (N-i+1)}
        \end{array}
   \end{bmatrix} 
\end{align*}
and
\begin{align*}
    \vB_i & = \frac{1}{2} \left( (\vx_i - \vy_\star) (\vg_i - \vg_\star)^\T + (\vg_i - \vg_\star) (\vx_i - \vy_\star)^\T \right) = \frac{1}{2} \left( \vx_i \ve_i^\T + \ve_i \vx_i^\T \right) = \begin{bmatrix}
        \begin{array}{c|c}
             \left(\vB_i\right)_{1:N,1:N} & \frac{1}{2} \ve_i \\
            \hline 
            \frac{1}{2} \ve_i^\T & 0
        \end{array}
   \end{bmatrix} .
\end{align*}
Assuming
\begin{align}
    \sum_{i=2}^N \sum_{j=1}^{i-1} a_{i,j} \vA_{i,j} + \sum_{i=1}^N b_i \vB_i + c_N \vC_N + d_N \vD_N + e_N \vE_N = 0   
    \label{eqn:matrix-summation-zero}
\end{align}
and considering the $(N+1,i)$-entry of the left hand side for $i=1,\dots,N-1$, we observe that only the matrix $b_i \vB_i$ contributes to it, and therefore we deduce $b_i = 0$ ($i=1,\dots,N-1$).
We also have $c_N = 0$ because $\vC_N$ is the only matrix with nonzero $(N+1,N+1)$-entry. 
Next, only $\vB_N$ and $\vD_N$ contribute to the $(N+1,N)$-entry, so we have
\begin{align}
\label{eqn:bN-dN-relation}
    \frac{1}{2} b_N - \frac{1}{N} d_N = 0 \iff b_N = \frac{2d_N}{N} .
\end{align}
Now we compare the remaining entries in the $N$-th row of both sides of \eqref{eqn:matrix-summation-zero}.
Because $\vx_i \in \ve_{N+1} + \spann\{\ve_1,\dots,\ve_i\}$, among all $\vA_{i,j}$'s and $\vB_i$'s, we see that only $\vA_{N,1}, \dots, \vA_{N,N-1}, \vB_N$ contribute to the $N$-th row.
Extracting only the $N$-th row of $\vA_{N,i}$ ($i=1,\dots,N-1$) gives:
\begin{align*}
    \left(\vA_{N,i}\right)_{N,1:N} & = \left[\frac{1}{2} \left((\vx_N - \vx_i)(\ve_N - \ve_i)^\T + (\ve_N - \ve_i)(\vx_N - \vx_i)^\T \right)\right]_{N,1:N} \\
    & = \left[ \frac{1}{2} \left( -\ve_N (\ve_N - \ve_i)^\T + \ve_N \left( -\sum_{j=1}^{i-1} \left( \sum_{k=i}^{N-1} 2h_{k,j} \right) \ve_j - \sum_{j=i}^{N-1} \left( \sum_{k=j}^{N-1} 2h_{k,j} \right) \ve_j + \ve_i - \ve_N \right)^\T \right) \right]_{N,1:N} \\
    & = \ve_i^\T - \ve_N^\T - \sum_{j=1}^{N-1} \left(\sum_{k=\max\{i,j\}}^{N-1} h_{k,j} \right) \ve_j^\T
\end{align*}
where in the second line, we omit the terms within the expression for $\vA_{N,i}$ that do not contribute to the $N$-th row.
Similarly, we have
\begin{align*}
    \left(\vB_N\right)_{N,1:N} & = \left[ \frac{1}{2} \left( \vx_N \ve_N^\T + \ve_N \vx_N^\T \right) \right]_{N,1:N} \\
    & = \left[ \frac{1}{2} \left( -\ve_N \ve_N^\T + \ve_N \left( -\sum_{j=1}^{N-1} \left( \sum_{k=j}^{N-1} 2h_{k,j} \right) \ve_j - \ve_N \right)^\T \right) \right]_{N,1:N} \\
    & = -\ve_N^\T - \sum_{j=1}^{N-1} \left( \sum_{k=j}^{N-1} h_{k,j} \right) \ve_j^\T .
\end{align*}
The influence from $\vC_N, \vD_N, \vE_N$, combined, is simply given by
\begin{align*}
    \left( c_N\vC_N + d_N\vD_N + e_N\vE_N\right)_{N,1:N} = (d_N + e_N) \ve_N^\T .
\end{align*}
Therefore, we have
\begin{align*}
    0 & = \left[ \sum_{i=2}^N \sum_{j=1}^{i-1} a_{i,j} \vA_{i,j} + \sum_{i=1}^N b_i \vB_i + c_N \vC_N + d_N \vD_N + e_N \vE_N \right]_{N,1:N} \\
    & = \sum_{i=1}^{N-1} a_{N,i} \left(\vA_{N,i}\right)_{N,1:N} + b_N \left(\vB_N\right)_{N,1:N} + (d_N + e_N) \ve_N^\T \\
    & = \left(d_N + e_N - b_N - \sum_{i=1}^{N-1} a_{N,i}\right) \ve_N^\T - \sum_{i=1}^{N-1} a_{N,i} \left(\sum_{j=1}^{N-1} \left(\sum_{k=\max\{i,j\}}^{N-1} h_{k,j}\right) \ve_j^\T - \ve_i^\T \right) - b_N \sum_{j=1}^{N-1} \left( \sum_{k=j}^{N-1} h_{k,j} \right) \ve_j^\T \\
    & = \left(d_N + e_N - b_N - \sum_{i=1}^{N-1} a_{N,i}\right) \ve_N^\T + \sum_{i=1}^{N-1} a_{N,i} \ve_i^\T - b_N \sum_{j=1}^{N-1} \left( \sum_{k=j}^{N-1} h_{k,j} \right) \ve_j^\T \\
    & \quad - \sum_{i=1}^{N-1} a_{N,i} \sum_{j=1}^{N-1} \left(\sum_{k=\max\{i,j\}}^{N-1} h_{k,j}\right) \ve_j^\T \\
    & = \left(d_N + e_N - b_N - \sum_{j=1}^{N-1} a_{N,j}\right) \ve_N^\T + \sum_{j=1}^{N-1} a_{N,j} \ve_j^\T - b_N \sum_{j=1}^{N-1} \left( \sum_{k=j}^{N-1} h_{k,j} \right) \ve_j^\T - \sum_{j=1}^{N-1} \sum_{k=j}^{N-1} \sum_{i=1}^{k} a_{N,i} h_{k,j} \ve_j^\T \\
    & = \left(d_N + e_N - b_N - \sum_{j=1}^{N-1} a_{N,j}\right) \ve_N^\T + \sum_{j=1}^{N-1} a_{N,j} \ve_j^\T - b_N \sum_{j=1}^{N-1} \left( \sum_{i=j}^{N-1} h_{i,j} \right) \ve_j^\T - \sum_{j=1}^{N-1} \sum_{i=j}^{N-1} \sum_{k=1}^{i} h_{i,j} a_{N,k} \ve_j^\T \\
    & = \left(d_N + e_N - b_N - \sum_{j=1}^{N-1} a_{N,j}\right) \ve_N^\T + \sum_{j=1}^{N-1} \left( a_{N,j} - \sum_{i=j}^{N-1} \sum_{k=1}^{i} h_{i,j} a_{N,k} - b_N \sum_{i=j}^{N-1} h_{i,j} \right) \ve_j^\T 
\end{align*}
where in the second-to-last equality, we change all occurrences of $i$ to $k$ and vice versa.
This implies that 
\begin{gather}  
    \sum_{j=1}^{N-1} a_{N,j} = d_N + e_N - b_N \label{eqn:aNj-sum} \\
    a_{N,j} - \sum_{i=j}^{N-1} \sum_{k=1}^{i} h_{i,j} a_{N,k} - b_N \sum_{i=j}^{N-1} h_{i,j} = 0 \quad (j=1,\dots,N-1) . \label{eqn:aNj-linear-system}
\end{gather}
Given $b_N, d_N, e_N$, this is an overdetermined linear system in terms of the variables $a_{N,1},\dots,a_{N,N-1}$ (where there are $N$ equations and $N-1$ variables).
The (square) matrix of the linear system consisting only of the $N-1$ equations in \eqref{eqn:aNj-linear-system} is
\begin{align}
\label{eqn:aNj-system-coefficient-matrix}
    \vM_{N} = \begin{bmatrix}
        1 - h_{1,1} - \dots - h_{N-1,1} & -h_{2,1} - \dots - h_{N-1,1} & \cdots & -h_{N-2,1} - h_{N-1,1} & -h_{N-1,1} \\
        -h_{2,2} - \dots - h_{N-1,2} & 1 - h_{2,2} - \dots - h_{N-1,2} & \cdots & -h_{N-2,2} - h_{N-1,2} & -h_{N-1,2} \\
        \vdots & \vdots & \ddots & \vdots & \vdots \\
        -h_{N-2,N-2} - h_{N-1,N-2} & -h_{N-2,N-2} - h_{N-1,N-2} & \cdots & 1 - h_{N-2,N-2} - h_{N-1,N-2} & -h_{N-1,N-2} \\
        -h_{N-1,N-1} & -h_{N-1,N-1} & \cdots & -h_{N-1,N-1} & 1 - h_{N-1,N-1}
    \end{bmatrix} .
\end{align}
But we can compute the determinant of this matrix as
\begin{align*}
    \det\vM_N = 1 - P(N-1, 1) + P(N-1, 2) - \dots + (-1)^{N-1} P(N-1, N-1) = D(N) = \frac{1}{N} \ne 0
\end{align*}
(we show $\det\vM_N = D(N)$ below as \cref{lemma:coefficient-matrix-determinant-is-DN}, and $D(N) = \frac{1}{N}$ by \cref{lemma:DN-under-H-invariance}).
Therefore, given $b_N$, the values of $a_{N,j}$ ($j=1,\dots,N-1$) satisfying \eqref{eqn:aNj-linear-system} are determined uniquely.
But note that equations \eqref{eqn:aNj-linear-system} are, up to the constant factor $b_N$, same as the linear system $s_{N,j}(\lambda) = 0$ ($j=1,\dots,N-1$) where only $\lambda_{N,1},\dots,\lambda_{N,N-1}$ are involved (we replace $\lambda_{N,j}$ with $\frac{a_{N,j}}{b_N}$).
We have shown that $\lambda_{N,j}^\star(H)$ in \cref{theorem:lambda-characterization} are solutions to that system, so we deduce that 
\begin{align*}
    a_{N,j} = b_N \lambda_{N,j}^\star (H)
\end{align*}
are the unique solutions to \eqref{eqn:aNj-linear-system}.
Now, for \eqref{eqn:aNj-sum} to hold as well, we must have
\begin{align*}
    d_N + e_N - b_N = \sum_{j=1}^{N-1} a_{N,j} = \sum_{j=1}^{N-1} b_N \lambda_{N,j}^\star(H) = b_N (N-1) .
\end{align*}
Because we have shown \eqref{eqn:bN-dN-relation}, equivalent to $d_N = \frac{Nb_N}{2}$, this implies that for the system combining \eqref{eqn:aNj-sum} and \eqref{eqn:aNj-linear-system} to be consistent we must have $e_N = \frac{Nb_N}{2} = d_N$, and when this is the case, $a_{N,j} = b_N \lambda_{N,j}^\star(H)$.

Finally, the remaining $a_{k,j}$'s with $1\le j < k \le N-1$ are determined using the first $N-1$ rows and columns of \eqref{eqn:matrix-summation-zero}.
But so far, we have derived $c_N = 0$, $d_N = e_N = \frac{Nb_N}{2}$ and $b_i = 0$ for $i=1,\dots,N-1$, and hence \eqref{eqn:matrix-summation-zero} is now equivalent to \eqref{eqn:matrix-identity-in-plain-equation}.
Therefore, $a_{k,j}$ are determined via linear equations involving the exact same forms as \eqref{eqn:skk-expression} and \eqref{eqn:skj-expression}:
\begin{align}
    \sum_{i=k+1}^N \left( \sum_{j=k}^{i-1} 2h_{j,k} - 1 \right) a_{i,k} - \sum_{j=1}^{k-1} a_{k,j} & = 0, \quad k = 1,\dots,N-1 \label{eqn:kk-entry-zero} \\
    a_{k,j} - \sum_{n=j}^{k-1} h_{n,j} \sum_{i=1}^n a_{k,i} + \underbrace{\sum_{m=k+1}^N \sum_{n=k}^{m-1} \left( h_{n,j} a_{m,k} + h_{n,k} a_{m,j} \right)}_{r_{k,j}} & = 0, \quad j=1,\dots,k-1 . \label{eqn:kj-entry-zero}
\end{align}
It is immediate from \cref{theorem:lambda-characterization} that taking $a_{k,j} = b_N \lambda_{k,j}^\star(H)$ for all $1\le j < k \le N-1$ solves \eqref{eqn:kk-entry-zero} and \eqref{eqn:kj-entry-zero}.
Now if these are \textit{unique} solutions given $a_{N,j} = b_N \lambda_{N,j}^\star(H)$ ($j=1,\dots,N-1$), then we conclude that \eqref{eqn:matrix-summation-zero} is equivalent to
\begin{align*}
    \left( \{a_{i,j}\}_{\substack{i=2,\dots,N \\ j=1,\dots,i-1}} , b_1, \dots, b_{N-1}, b_N , c_N , d_N , e_N \right) = b_N \cdot \left( \{\lambda_{i,j}^\star\}_{\substack{i=2,\dots,N \\ j=1,\dots,i-1}} , 0, \dots, 0, 1 , 0 , \frac{N}{2} , \frac{N}{2} \right)
\end{align*}
(with some abuse of set notations), completing the proof.
Therefore, it only remains to show that there are no other values of $a_{k,j}$ satisfying \eqref{eqn:kk-entry-zero} and \eqref{eqn:kj-entry-zero}.

Note that the expressions $r_{k,j}$ depend only on $a_{m,\cdot}$'s with $m>k$.
It suffices to prove that for each $k=1,\dots,N-1$, if all $a_{m,\cdot}$'s with $m>k$ are already determined (so that $r_{k,\cdot}$ are constants), there is at most one possible set of values of $a_{k,1}, \dots, a_{k,k-1}$ satisfying \eqref{eqn:kk-entry-zero} and \eqref{eqn:kj-entry-zero}.
To see this, observe that viewing only $a_{k,1}, \dots, a_{k,k-1}$ as variables and everything else as constants, \eqref{eqn:kk-entry-zero} and \eqref{eqn:kj-entry-zero} is an overdetermined linear system (with $k$ equations).
Now consider its partial system with the equation in \eqref{eqn:kj-entry-zero} corresponding to $j=1$ removed.
Its coefficient matrix is the $(k-1) \times (k-1)$ square matrix

\begin{align*}
    \begin{bmatrix}
        1 & 1 & \cdots & 1 & 1 \\
        - h_{2,2} - \dots - h_{k-1,2} & 1 - h_{2,2} - \dots - h_{k-1,2} & \cdots & - h_{k-2,2} - h_{k-1,2} & -h_{k-1,2} \\
        \vdots & \vdots & \ddots & \vdots & \vdots \\
        -h_{k-2,k-2} - h_{k-1,k-2} & -h_{k-2,k-2} - h_{k-1,k-2} & \cdots & 1 - h_{k-2,k-2} - h_{k-1,k-2} & -h_{k-1,k-2} \\
        -h_{k-1,k-1} & -h_{k-1,k-1} & \cdots & -h_{k-1,k-1} & 1 - h_{k-1,k-1}
    \end{bmatrix} 
\end{align*}
where the first row corresponds to \eqref{eqn:kk-entry-zero}, which is equivalent to
\begin{align*}
    a_{k,1} + \dots + a_{k,k-1} = \sum_{i=k+1}^{N} \left( \sum_{j=k}^{i-1} 2h_{j,k} - 1 \right) a_{i,k} ,
\end{align*}
the second row corresponds to the $j=2$ case of \eqref{eqn:kj-entry-zero}, which is equivalent to
\begin{align*}
    a_{k,2} - \sum_{n=2}^{k-1} h_{n,2} \sum_{i=1}^n a_{k,i} = a_{k,2} - \sum_{i=1}^{k-1} \left( \sum_{n=\max\{i,2\}}^{k-1} h_{n,2} \right) a_{k,i} = -r_{k,2}
\end{align*}
and so on, until the last row, which corresponds to the $j=k-1$ case of \eqref{eqn:kj-entry-zero}, which is equivalent to
\begin{align*}
    a_{k,k-1} - \sum_{n=k-1}^{k-1} h_{n,k-1} \sum_{i=1}^n a_{k,i} = a_{k,k-1} - \sum_{i=1}^{k-1} h_{k-1,k-1} a_{k,i} = -r_{k,k-1} .
\end{align*}
By performing the elementary row operation of adding $(h_{j,j} + \dots + h_{k-1,j})$ times the first row to the $j$-th row for $j=2,\dots,k-1$, we obtain
\begin{align*}
    \begin{bmatrix}
        1 & 1 & 1 & \cdots & 1 & 1 \\
        0 & 1 & h_{2,2} & \cdots & h_{2,2} + \dots + h_{k-3,2} & h_{2,2} + \cdots + h_{k-2,2} \\
        0 & 0 & 1 & \cdots & h_{3,3} + \dots + h_{k-3,3} & h_{3,3} + \dots + h_{k-2,3} \\ 
        \vdots & \vdots & \vdots & \ddots & \vdots & \vdots \\
        0 & 0 & 0 & \cdots & 1 & h_{k-2,k-2} \\
        0 & 0 & 0 & \cdots & 0 & 1
    \end{bmatrix}
\end{align*}
and the determinant of this matrix is $1$.
This implies that if $a_{k,1}, \dots, a_{k,k-1}$ satisfying \eqref{eqn:kk-entry-zero} and \eqref{eqn:kj-entry-zero} exist, then it is unique (so $a_{k,1}, \dots, a_{k,k-1}$ are uniquely determined as functions of $a_{m,\cdot}$ with $m > k$).
\end{proof}

Finally, it remains to show \cref{lemma:coefficient-matrix-determinant-is-DN}.
We first show a handy recursive formula for $D(k)$, and then prove \cref{lemma:coefficient-matrix-determinant-is-DN}.

\begin{lemma}
\label{lemma:Dk_recursion}
The following holds for $k=1,\dots,N-1$:
\begin{align*}
    D(k+1) = D(k) - \sum_{j=1}^{k} h_{k,j} D(j) = (1 - h_{k,k}) D(k) - \sum_{j=1}^{k-1} h_{k,j} D(j) .
\end{align*}
\end{lemma}

\begin{proof}
A monomial appearing in $D(k+1) = 1 + \sum_{m=1}^k (-1)^m P(k,m; H)$ is of the form $(-1)^m \prod_{r=1}^m h_{i(r),j(r)}$ for some $m \in \{1,\dots,k\}$ and $j(r) \le i(r) < j(r+1) \le i(r+1) \le k$ for $r=1,\dots,m-1$.
We have two cases: either $i(m) \le k-1$ or $i(m) = k$.
In the former case, $(i(1),j(1),\dots,i(m),j(m)) \in I(k-1,m)$, so $(-1)^m \prod_{r=1}^m h_{i(r),j(r)}$ is a part of $D(k)$ (using the same definition of the index set $I(\cdot,\cdot)$ as in the proof of \cref{lemma:P_Q_identities}).
In the latter case, if $m\ge 2$, $i(m-1) < j(m)$, so
\begin{align*}
    (i(1),j(1),\dots,i(m-1),j(m-1)) \in I(j(m)-1,m-1) .
\end{align*}
Thus, we can write
\begin{align*}
    D(k+1) & = D(k) - \sum_{j(1)=1}^k h_{k,j(1)} - \sum_{m=2}^{k} (-1)^{m-1} \sum_{j(m)=1}^{k} h_{k,j(m)} \sum_{(i(1),j(1),\dots,i(m-1),j(m-1)) \in I(j(m)-1,m-1)} \prod_{r=1}^{m-1} h_{i(r),j(r)} \\
    & = D(k) - \sum_{j=1}^k h_{k,j} \bigg[ 1 + \sum_{m=2}^{k} (-1)^{m-1} \underbrace{\sum_{(i(1),j(1),\dots,i(m-1),j(m-1)) \in I(j-1,m-1)} \prod_{r=1}^{m-1} h_{i(r),j(r)}}_{=P(j-1,m-1)} \bigg] \\
    & = D(k) - \sum_{j=1}^k h_{k,j} \sum_{\ell=0}^{j-1} (-1)^\ell P(j-1,\ell) = D(k) - \sum_{j=1}^{k} h_{k,j} D(j) 
\end{align*}
where the third equality uses $P(j-1,0) = 1$, $P(j-1,m-1) = 0$ for $m > j$, and makes the index change $\ell = m-1$.
\end{proof}

\begin{lemma}
\label{lemma:coefficient-matrix-determinant-is-DN}
Let $\vM_N$ be as in \eqref{eqn:aNj-system-coefficient-matrix}. Then $\det \vM_N = D(N)$.
\end{lemma}

\begin{proof}
Performing the elementary column operation of subtracting the last column from all remaining columns of $\vM_N$, we obtain
\begin{align}
\label{eqn:MN-column-operation-first}
    \begin{bmatrix}
        1 - h_{1,1} - \dots - h_{N-2,1} & -h_{2,1} - \dots - h_{N-2,1} & \cdots & -h_{N-2,1} & -h_{N-1,1} \\
        -h_{2,2} - \dots - h_{N-2,2} & 1 - h_{2,2} - \dots - h_{N-2,2} & \cdots & -h_{N-2,2} & -h_{N-1,2} \\
        \vdots & \vdots & \ddots & \vdots & \vdots \\
        -h_{N-2,N-2} & -h_{N-2,N-2} & \cdots & 1 - h_{N-2,N-2} & -h_{N-1,N-2} \\
        -1 & -1 & \cdots & -1 & 1 - h_{N-1,N-1}
    \end{bmatrix} .
\end{align}
Now subtracting the second-to-last column of \eqref{eqn:MN-column-operation-first} from all columns up to the $(N-3)$-th column, we obtain
\begin{align*}
    \begin{bmatrix}
        1 - h_{1,1} - \dots - h_{N-3,1} & -h_{2,1} - \dots - h_{N-3,1} & \cdots & -h_{N-3,1} & -h_{N-2,1} & -h_{N-1,1} \\
        -h_{2,2} - \dots - h_{N-3,2} & 1 - h_{2,2} - \dots - h_{N-3,2} & \cdots & -h_{N-3,2} & -h_{N-2,2} & -h_{N-1,2} \\
        \vdots & \vdots & \ddots & \vdots & \vdots & \vdots \\
        -h_{N-3,N-3} & -h_{N-3,N-3} & \cdots & 1 - h_{N-3,N-3} & -h_{N-2,N-3} & -h_{N-1,N-3} \\
        -1 & -1 & \cdots & -1 & 1 - h_{N-2,N-2} & -h_{N-1,N-2} \\
        0 & 0 & \cdots & 0 & -1 & 1 - h_{N-1,N-1}
    \end{bmatrix} .
\end{align*}
Continuing until the end, we obtain the matrix
\begin{align*}
    \vL_N = \begin{bmatrix}
        1 - h_{1,1}& -h_{2,1} & \cdots & -h_{N-2,1} & -h_{N-1,1} \\
        -1 & 1 - h_{2,2} & \cdots & -h_{N-2,2} & -h_{N-1,2} \\
        & \ddots & \ddots & \vdots & \vdots \\
        & & \ddots & 1 - h_{N-2,N-2} & -h_{N-1,N-2} \\
        & & & -1 & 1 - h_{N-1,N-1}
    \end{bmatrix} .
\end{align*}
Now we use induction to show that $\det\vM_N = \det\vL_N = D(N)$.
The base case is $N=2$, when we have $\vL_2 = \begin{bmatrix} 1 - h_{1,1} \end{bmatrix}$ and thus, trivially, $\det \vL_2 = 1 - h_{1,1} = 1 - P(1,1) = D(2)$.
Now assuming that $\det \vL_k = D(k)$ for all $k=2,\dots,N-1$ and taking the cofactor expansion of $\vL_N$ along the last row gives
\begin{align*}
    \det\vL_N & = (1 - h_{N-1,N-1}) \det \begin{bmatrix}
        1 - h_{1,1}& -h_{2,1} & \cdots & -h_{N-2,1}\\
        -1 & 1 - h_{2,2} & \cdots & -h_{N-2,2}\\
        & \ddots & \ddots & \vdots & \\
        & & -1 & 1 - h_{N-2,N-2}
    \end{bmatrix} \\
    & \quad + \det \underbrace{\begin{bmatrix}
        1 - h_{1,1}& -h_{2,1} & \cdots & -h_{N-3,1} & -h_{N-1,1} \\
        -1 & 1 - h_{2,2} & \cdots & -h_{N-3,2} & -h_{N-1,2} \\
        & \ddots & \ddots & \vdots & \vdots \\
        & & \ddots & 1 - h_{N-3,N-3} & -h_{N-1,N-3} \\
        & & & -1 & -h_{N-1,N-2}
    \end{bmatrix}}_{\Tilde{\vL}_{N,N-1}} \\
    & = (1 - h_{N-1,N-1}) \det \vL_{N-1} + \det \Tilde{\vL}_{N,N-1} .
\end{align*}
Cofactor-expanding $\Tilde{\vL}_{N,N-1}$ along the last row again, we obtain
\begin{align*}
    \det \Tilde{\vL}_{N,N-1} = -h_{N-1,N-2} \det \vL_{N-2} + \det \underbrace{\begin{bmatrix}
        1 - h_{1,1}& -h_{2,1} & \cdots & -h_{N-4,1} & -h_{N-1,1} \\
        -1 & 1 - h_{2,2} & \cdots & -h_{N-4,2} & -h_{N-1,2} \\
        & \ddots & \ddots & \vdots & \vdots \\
        & & \ddots & 1 - h_{N-4,N-4} & -h_{N-1,N-4} \\
        & & & -1 & -h_{N-1,N-3}
    \end{bmatrix}}_{\Tilde{\vL}_{N,N-2}} .
\end{align*}
Repeating the process, we have
\begin{align*}
    \det \vL_N & = (1 - h_{N-1,N-1}) \det \vL_{N-1} - h_{N-1,N-2} \det \vL_{N-2} - \dots - h_{N-1,2} \det \vL_2 - h_{N-1,1} \\
    & = (1 - h_{N-1,N-1}) D(N-1) - \sum_{k=1}^{N-2} h_{N-1,k} D(k) \\
    & = D(N)
\end{align*}
where the second line uses the induction hypothesis and the last line follows from \cref{lemma:Dk_recursion}.
\end{proof}

\section{Proof of \cref{lemma:P-recursion}}
\label{section:proof-simple-P-recursion-lemma}

\begin{proof}
By definition of $I(k,m)$ defined in the proof of \cref{lemma:P_Q_identities}, we have
\begin{align}
\label{eqn:P_kp1_mp1_definition}
    P(k+1,m) = \sum_{(i(1), j(1), \dots, i(m), j(m)) \in I(k+1,m)} \prod_{r=1}^{m} h_{i(r),j(r)} .
\end{align}
Now for each $(i(1), j(1), \dots, i(m), j(m)) \in I(k+1,m)$ we have $j(m) \ge m$, and for each $n=m,\dots,k+1$,
\begin{align*}
    & (i(1), j(1), \dots, i(m-1), j(m-1), i(m), n) \in I(k+1,m) \\
    & \iff (i(1), j(1), \dots, i(m-1), j(m-1)) \in I(n-1, m-1) \,\,\, \text{and} \,\,\, n \le i(m) \le k+1 .
\end{align*}
Therefore, by partitioning $I(k+1,m)$ based on the value of $j(m)$, we can rewrite \eqref{eqn:P_kp1_mp1_definition} as
\begin{align*}
    P(k+1,m) & = \sum_{n=m}^{k+1} \sum_{i(m)=n}^{k+1} h_{i(m),n} \left( \sum_{(i(1), j(1), \dots, i(m-1), j(m-1)) \in I(n-1, m-1)} \prod_{r=1}^{m-1} h_{i(r),j(r)} \right) \\
    & = \sum_{n=m}^{k+1} \sum_{i=n}^{k+1} h_{i,n} P(n-1,m-1)
\end{align*}
where we replace the index $i(m)$ with $i$ and use the definition of $P(n-1,m-1)$.
The desired result then follows by applying the index change $j = n-1$.
\end{proof}

\section{Proofs for \cref{section:new-boundary-algorithms}}
\label{section:new-algorithm-proofs}

\subsection{Proof of \cref{proposition:OHM-linear-equation}}

We claim that $\lambda^\star_{j+2,j} = \cdots = \lambda^\star_{N,j} = 0$ is equivalent to the linear system
\begin{align}
\label{eqn:OHM-Q-linear-system}
    \underbrace{\begin{bmatrix}
        1 & -1 & \cdots & (-1)^{N-j-1} \\
        -\binom{2}{1} & \binom{3}{1} & \cdots & (-1)^{N-j} \binom{N-j+1}{1} \\
        \vdots & \vdots & & \vdots \\
        (-1)^{N-j-2} \binom{N-j-1}{N-j-2} & (-1)^{N-j-1} \binom{N-j}{N-j-2} & \cdots & -\binom{2(N-j-1)}{N-j-2}
    \end{bmatrix}}_{\in \reals^{(N-j-1) \times (N-j)}}
    \begin{bmatrix}
        Q(N-1,1,j) \\ Q(N-1,2,j) \\ \vdots \\ Q(N-1,N-j,j)
    \end{bmatrix} = 0 
\end{align}
having $Q(N-1,1,j), Q(N-1,2,j), \dots, Q(N-1,N-j,j)$ as variables.
First, observe that
\begin{align*}
    \lambda_{N,j}^\star = 0 \iff \sum_{\ell=1}^{N-j} (-1)^{\ell-1} Q(N-1,\ell,j) = 0 
\end{align*}
is already a linear equation in $Q(N-1,\ell,j)$ for $\ell=1,\dots,N-j$, and agrees with the first row of \eqref{eqn:OHM-Q-linear-system}.
When $j=N-2$, this is the only equation that should be satisfied and our claim is proved. Also, in this case, that single equation is equivalent to $Q(N-1,1,N-2) = Q(N-1,2,N-2)$, which is consistent with the formula~\eqref{eqn:Q-relationship-formula-OHM}.

When $j<N-2$, we have
\begin{align}
\label{eqn:lambda-Nm1j-zero-to-linear-equation}
\begin{aligned}
    & \lambda_{N-1,j}^\star = NQ(N-1,1,N-1) \sum_{\ell=1}^{N-j} (-1)^{\ell} \binom{\ell+1}{1} Q(N-1,\ell,j) = 0 \\
    & \iff \sum_{\ell=1}^{N-j} (-1)^{\ell-1} \binom{\ell+1}{\ell} Q(N-1,\ell,j) = 0
\end{aligned}
\end{align}
because $Q(N-1,1,N-1) = h_{N-1,N-1} \ne 0$ (diagonal entries are nonzero under H-invariance).
This agrees with the second row of~\eqref{eqn:OHM-Q-linear-system}, and the claim is proved for the case $j=N-3$.

For $k=N-2,N-3,\dots,j+2$, assume that we have handled the equations $\lambda_{k+1,j}^\star = \cdots = \lambda_{N,j}^\star$ and obtained linear equations agreeing with \eqref{eqn:OHM-Q-linear-system} up to the $(N-k)$-th row.
Then, from the equation $\lambda_{k,j}^\star = 0$ we observe that
\begin{align}
\label{qn:lambda-kj-zero-to-linear-equation}
\begin{aligned}
    0 = \lambda_{k,j}^\star & = N \sum_{m=1}^{N-k} (-1)^m Q(N-1,m,k) \sum_{\ell=1}^{N-j} (-1)^{\ell-1} \binom{\ell+m}{m} Q(N-1,\ell,j) \\
    & = N (-1)^{N-k} Q(N-1,N-k,k) \sum_{\ell=1}^{N-j} (-1)^{\ell-1} \binom{\ell+N-k}{N-k} Q(N-1,\ell,j) \\
    & \quad + N \sum_{m=1}^{N-k-1} (-1)^m Q(N-1,m,k) \underbrace{\sum_{\ell=1}^{N-j} (-1)^{\ell-1} \binom{\ell+m}{m} Q(N-1,\ell,j)}_{=0} 
\end{aligned}
\end{align}
and $Q(N-1,N-k,k) = h_{k,k} h_{k+1,k+1} \cdots h_{N-1,N-1} \ne 0$. This implies
\begin{align*}
    \sum_{\ell=1}^{N-j} (-1)^{\ell-1} \binom{\ell+N-k}{N-k} Q(N-1,\ell,j) = 0 ,
\end{align*}
which agrees with the $(N-k+1)$-th row. 
This completes the induction and proves that our quadratic system can be reduced to \eqref{eqn:OHM-Q-linear-system}.
Conversely, if \eqref{eqn:OHM-Q-linear-system} holds, then every summation of the form $\sum_{\ell=1}^{N-j} (-1)^{\ell-1} \binom{\ell+m}{m} Q(N-1,\ell,j)$ vanishes for $m=1,\dots,N-j-2$, which immediately implies $\lambda^\star_{j+2,j} = \cdots = \lambda^\star_{N,j} = 0$.

Now, to characterize the solution, we perform Gauss-Jordan elimination on \eqref{eqn:OHM-Q-linear-system}.
By repeating the sequence of elementary row operations $R_{m+1} \to R_{m+1} + R_m$ multiple times and applying the identity $\binom{m+1}{\ell+1} - \binom{m}{\ell} = \binom{m}{\ell+1}$, until the coefficient matrix becomes upper trapezoidal, \eqref{eqn:OHM-Q-linear-system} simplifies to
\begin{align}
\label{eqn:OHM-Q-linear-system-Gaussian-elimination}
    \begin{bmatrix}
        1 & -1 & 1 & \cdots & (-1)^{N-j-3} & (-1)^{N-j-2} & (-1)^{N-j-1} \\
        0 & \binom{1}{1} & -\binom{2}{1} & \cdots & (-1)^{N-j-4} \binom{N-j-3}{1} & (-1)^{N-j-3} \binom{N-j-2}{1} & (-1)^{N-j-2} \binom{N-j-1}{1} \\
        \vdots & \vdots & \vdots &  & \vdots & \vdots & \vdots \\
        0 & 0 & 0 & \cdots & \binom{N-j-3}{N-j-3} = 1 & -\binom{N-j-2}{N-j-3} & \binom{N-j-1}{N-j-3} \\
        0 & 0 & 0 & \cdots & 0 & \binom{N-j-2}{N-j-2} = 1 & -\binom{N-j-1}{N-j-2}
    \end{bmatrix} .
\end{align}
As the matrix \eqref{eqn:OHM-Q-linear-system-Gaussian-elimination} has the full rank of $N-j-1$, its null space is one-dimensional. Now plugging $a=N-j-1$, $b=-k$, $c=N-j-k$ into \cref{lemma:chu-vandermonde}, using $\binom{N-j-1}{i} = \binom{N-j-1}{N-j-i-1}$, $\binom{-k}{N-j-k-i} = (-1)^{N-j-k-i} \binom{N-j-i-1}{N-j-i-k} = (-1)^{N-j-k-i} \binom{N-j-i-1}{k-1}$ and making the change of index $\ell = N-j-i-1$, we obtain the identity
\begin{align*}
    \sum_{\ell=k-1}^{N-j-1} (-1)^{\ell-k+1} \binom{N-j-1}{\ell} \binom{\ell}{k-1} = 0 .
\end{align*}
Comparing this with \eqref{eqn:OHM-Q-linear-system-Gaussian-elimination}, we conclude that the one-dimensional solution of \eqref{eqn:OHM-Q-linear-system} is characterized by
\begin{align*}
    Q(N-1,k,j) = \binom{N-j-1}{k-1} Q(N-1,N-j,j)
\end{align*}
for $k=1,\dots,N-j-1$.

\subsection{Proof of Propositions~\ref{proposition:Dual-OHM-linear-equation} and \ref{proposition:Dual-OHM-is-that-algorithm-with-Qs}}

We proceed similarly as in the proof of \cref{proposition:OHM-linear-equation}, but with the difference that $\lambda_{N,j}^\star = N \sum_{\ell=1}^{N-j} (-1)^{\ell-1} Q(N-1,\ell,j) \ne 0$ and $\lambda_{j+1,j}^\star = 0$ instead. We claim that \eqref{eqn:sparsity-pattern-Dual-OHM} is equivalent to
\begin{align}
\label{eqn:Dual-OHM-Q-linear-system}
    \underbrace{\begin{bmatrix}
        -\binom{2}{1} & \binom{3}{1} & \cdots & (-1)^{N-j} \binom{N-j+1}{1} \\
        \vdots & \vdots & & \vdots \\
        (-1)^{N-j-2} \binom{N-j-1}{N-j-2} & (-1)^{N-j-1} \binom{N-j}{N-j-2} & \cdots & -\binom{2(N-j-1)}{N-j-2} \\
        (-1)^{N-j-1} \binom{N-j}{N-j-1} & (-1)^{N-j} \binom{N-j+1}{N-j-1} & \cdots & \binom{2N-2j-1}{N-j-1}
    \end{bmatrix}}_{\in \reals^{(N-j-1) \times (N-j)}}
    \begin{bmatrix}
        Q(N-1,1,j) \\ Q(N-1,2,j) \\ \vdots \\ Q(N-1,N-j,j)
    \end{bmatrix} = 0 .
\end{align}
The first equation is derived in the exact the same way as in~\eqref{eqn:lambda-Nm1j-zero-to-linear-equation}, and the subsequent equations can be derived inductively as in~\eqref{qn:lambda-kj-zero-to-linear-equation}, with the only difference being that we also include the case $k=j+1$.
It is clear that \eqref{eqn:Dual-OHM-Q-linear-system} conversely implies $\lambda_{j+1,j}^\star = \cdots = \lambda_{N-1,j}^\star = 0$.

Again, to characterize the solution to \eqref{eqn:Dual-OHM-Q-linear-system}, we perform Gauss-Jordan elimination by repeating the sequential elementary row operations $R_{m+1} \to R_{m+1} + R_m$ and applying $\binom{m+1}{\ell+1} - \binom{m}{\ell} = \binom{m}{\ell+1}$.
However, unlike in the case of \cref{proposition:OHM-linear-equation}, our first row is not $\begin{bmatrix} 1 & -1 & \cdots & (-1)^{N-j-1} \end{bmatrix}$ and we cannot fully reduce \eqref{eqn:Dual-OHM-Q-linear-system} into an upper trapezoidal matrix with this process.
Instead, the result we obtain is the following matrix with nonzero subdiagonal entries (with some sign flipping):
\begin{align}
\label{eqn:Dual-OHM-Q-linear-system-Gaussian-elimination}
    \begin{bmatrix}
        \binom{2}{1} & -\binom{3}{1} & \binom{4}{1} & \cdots & (-1)^{N-j-1} \binom{N-j-1}{1} & (-1)^{N-j} \binom{N-j}{1} & (-1)^{N-j+1} \binom{N-j+1}{1} \\
        \binom{2}{2} & -\binom{3}{2} & \binom{4}{2} & \cdots & (-1)^{N-j-1} \binom{N-j-1}{2} & (-1)^{N-j} \binom{N-j}{2} & (-1)^{N-j+1} \binom{N-j+1}{2} \\
        \vdots & \vdots & \vdots &  & \vdots & \vdots & \vdots \\
        0 & 0 & 0 & \cdots & (-1)^{N-j-1} \binom{N-j-1}{N-j-2} & (-1)^{N-j} \binom{N-j}{N-j-2} & (-1)^{N-j+1} \binom{N-j+1}{N-j-2} \\
        0 & 0 & 0 & \cdots & (-1)^{N-j-1} \binom{N-j-1}{N-j-1} & (-1)^{N-j} \binom{N-j}{N-j-1} & (-1)^{N-j+1} \binom{N-j+1}{N-j-1}
    \end{bmatrix} .
\end{align}
Note that the $m$-th row entries of \eqref{eqn:Dual-OHM-Q-linear-system-Gaussian-elimination} can be written as $(-1)^{k+1} \binom{k+1}{m}$ where $k=1,\dots,N-j+1$ are column indices.
For the first row in particular, this coincides with $(-1)^{k+1} \frac{2}{k} \binom{k+1}{2}$.
Using the formula
\begin{align*}
    \binom{k+1}{m+1} - \frac{m}{m+1} \left[ \frac{m+1}{k} \binom{k+1}{m+1} \right] = \frac{k-m}{k} \binom{k+1}{m+1} = \frac{m+2}{k} \binom{k+1}{m+2} ,
\end{align*}
we inductively see that when we sequentially apply the row operations $R_{m+1} \to R_{m+1} - \frac{m}{m+1} R_m$ (in the order of $m=1,\dots,N-j-2$), the first nonzero entry of each $R_{m+1}$ gets eliminated, and the $(m,k)$-entry of the resulting matrix is $(-1)^{k+1} \frac{m+1}{k} \binom{k+1}{m+1}$:

\begin{align}
    \label{eqn:Dual-OHM-Q-linear-system-Gaussian-elimination-2}
    \begin{bmatrix}
        \frac{2}{1}\binom{2}{2} & -\frac{2}{2}\binom{3}{2} & \frac{2}{3}\binom{4}{2} & \cdots & (-1)^{N-j-1} \frac{2}{N-j-2} \binom{N-j-1}{2} & (-1)^{N-j} \frac{2}{N-j-1} \binom{N-j}{2} & (-1)^{N-j+1} \frac{2}{N-j} \binom{N-j+1}{2} \\
        0 & -\frac{3}{2}\binom{3}{3} & \frac{3}{3}\binom{4}{3} & \cdots & (-1)^{N-j-1} \frac{3}{N-j-2}\binom{N-j-1}{3} & (-1)^{N-j} \frac{3}{N-j-1}\binom{N-j}{3} & (-1)^{N-j+1} \frac{3}{N-j}\binom{N-j+1}{3} \\
        \vdots & \vdots & \vdots &  & \vdots & \vdots & \vdots \\
        0 & 0 & 0 & \cdots & (-1)^{N-j-1} \frac{N-j-1}{N-j-2} \binom{N-j-1}{N-j-1} & (-1)^{N-j} \frac{N-j-1}{N-j-1}\binom{N-j}{N-j-1} & (-1)^{N-j+1} \frac{N-j-1}{N-j} \binom{N-j+1}{N-j-1} \\
        0 & 0 & 0 & \cdots & 0 & (-1)^{N-j} \frac{N-j}{N-j-1} \binom{N-j}{N-j} & (-1)^{N-j+1} \frac{N-j}{N-j} \binom{N-j+1}{N-j}
    \end{bmatrix} .
\end{align}
This matrix has the full rank of $N-j-1$, so its null space is one-dimensional.
Finally, plugging $a=N-j$, $b=-\ell-1$, $c=N-j-\ell$ into \cref{lemma:chu-vandermonde}, making the change of index $k=N-j-i$ and using $\binom{-\ell-1}{k-\ell} = (-1)^{k-\ell} \binom{k}{k-\ell} = (-1)^{k-\ell} \binom{k}{\ell}$ we obtain the identity
\begin{align*}
    \sum_{k=\ell}^{N-j} (-1)^{k+1} \binom{N-j}{k} \binom{k}{\ell} = 0 .
\end{align*}
Comparing this with \eqref{eqn:Dual-OHM-Q-linear-system-Gaussian-elimination-2}, we conclude that the relations 
\begin{align*}
    Q(N-1,k,j) = \frac{N-j+1}{k+1}\binom{N-j-1}{k-1} Q(N-1,N-j,j) 
\end{align*} 
characterize the one-dimensional solution to~\eqref{eqn:Dual-OHM-Q-linear-system}. 

To verify the formula $Q(N-1,k,j) = \frac{1}{k+1} \binom{N-j-1}{k-1}$ of \cref{proposition:Dual-OHM-is-that-algorithm-with-Qs}, as it is immediate that they satisfy $Q(N-1,k,j) = \frac{N-j+1}{k+1}\binom{N-j-1}{k-1} Q(N-1,N-j,j) $, we only need to check the H-invariance~\eqref{eqn:H-invariance-Q-sum}. This is an immediate application of \cref{lemma:generalized-hockeystick} with $p=N-1$, $q=m-1$ and $r=0$:
\begin{align*}
    \sum_{j=1}^{N-m} Q(N-1,m,j) = \frac{1}{m+1} \sum_{j=1}^{N-m} \binom{N-j-1}{m-1} = \frac{1}{m+1} \left[ \binom{N}{m} - \binom{N-1}{m-1} \right] = \frac{1}{m+1} \binom{N-1}{m} = \frac{1}{N} \binom{N}{m+1} .
\end{align*}
Finally, the fact that Dual-OHM provides \textit{an} exact optimal $H$ with the sparsity pattern of \cref{proposition:Dual-OHM-linear-equation} \citep[Theorem~3.3]{YoonKimSuhRyu2024_optimal}, together with the discussion above showing that the sparsity pattern uniquely determines the Q-profile (and hence $H$ via the procedure of Section~\ref{subsubsection:re-deriving-OHM}), proves that Dual-OHM must be that corresponding algorithm.

\subsection{Proof of~\cref{proposition:self-dual-family-Q-functions}}

Observe that our proposed Q-profile
\begin{align*}
    Q(N-1,k,j) = \begin{cases}
        \frac{j}{N}\binom{N-j-1}{k-1} & \text{for } j=1,\dots,N'-1 \\
        \frac{N'(N-N'+1)}{N(k+1)}\binom{N-N'-1}{k-1} & \text{for } j=N' \\
        \frac{1}{k+1}\binom{N-j-1}{k-1} & \text{for } j=N'+1,\dots,N-1 .
    \end{cases}
\end{align*}
satisfies \eqref{eqn:new-algorithm-search-ckj} with
\begin{align*}
    c_{k,j} = \begin{cases}
        \binom{N-j-1}{k-1} & \text{if } j=1,\dots,N'-1 \\
        \frac{N-j+1}{k+1} \binom{N-j-1}{k-1} & \text{if } j=N',\dots,N-1 ,
    \end{cases}
\end{align*}
which implies that \eqref{eqn:sparsity-pattern-OHM} holds for $j=1,\dots,N'-1$ (by \cref{proposition:OHM-linear-equation}) and \eqref{eqn:sparsity-pattern-Dual-OHM} holds for $j=N',\dots,N-1$ (by \cref{proposition:Dual-OHM-linear-equation}).
To see that these values of $Q(N-1,\cdot,\cdot)$ indeed represent a valid optimal algorithm, it remains to check: \textbf{(i)} H-invariance condition \eqref{eqn:H-invariance-Q-sum} and \textbf{(ii)} positivity of the remaining H-certificates.
The uniqueness of this Q-profile follows from the fact that they are determined by the same procedure as described in Section~\ref{subsubsection:re-deriving-OHM}.

\begin{enumerate}
    \item[\textbf{(i)}] \textbf{H-invariance.} For $m\ge N-N'$, all values of $Q(N-1,m,j)$ ($j=1,\dots,N-m$) agree with those from OHM, and for them, \eqref{eqn:H-invariance-Q-sum} is already proved.
    For $m\le N-N'-1$, we have:
    \begin{align*}
        \sum_{j=1}^{N-m} Q(N-1,m,j) = \sum_{j=1}^{N'-1} \frac{j}{N} \binom{N-j-1}{m-1} + \frac{N'(N-N'+1)}{N(m+1)}\binom{N-N'-1}{m-1} + \sum_{j=N'+1}^{N-m} \frac{1}{m+1} \binom{N-j-1}{m-1} .
    \end{align*}
    Using \cref{lemma:simple-combinatorial-summations}, we can simplify the summations as
    \begin{align*}
        \sum_{j=1}^{N'-1} \frac{j}{N} \binom{N-j-1}{m-1} & = \frac{1}{N} \left( \binom{N}{m+1} - (N'-1)\binom{N-N'}{m} - \binom{N-N'+1}{m+1} \right) \\
        \sum_{j=N'+1}^{N-m} \frac{1}{m+1} \binom{N-j-1}{m-1} & = \frac{1}{m+1} \binom{N-N'-1}{m} 
    \end{align*}
    where the first line uses part (b) with $p=N-1$, $q=m-1$ and $s=N'-1$, and the second line uses part (a) with $p=N-N'-2$, $q=m-1$ and the change of index $j=N-i-1$.
    Plugging these back in and simplifying, we obtain $\sum_{j=1}^{N-m} Q(N-1,m,j) = \frac{1}{N}\binom{N}{m+1}$ as desired.

    \item[\textbf{(ii)}] \textbf{Positivity of remaining $\lambda^\star$.} Note that the values of $Q(N-1,k,j)$ agree with those from OHM for $j=1,\dots,N'-1$, which implies that the values of $\lambda^\star_{j+1,j}$ also agree with those from OHM for $j=1,\dots,N'-2$, and in particular, they are positive.
    Similarly, for $j=N'+1,\dots,N-1$, the values of $Q(N-1,k,j)$ agree with those from Dual-OHM, and thus $\lambda^\star_{N,j}$ are the same as those from Dual-OHM, hence positive.
    It only remains to check that $\lambda^\star_{N',N'-1}$ and $\lambda^\star_{N,N'}$ are positive.
    This can be computed directly as
    \begin{align*}
        \lambda^\star_{N,N'} & = N \sum_{m=1}^{N-N'} (-1)^{m-1} Q(N-1,m,N') = N \sum_{m=1}^{N-N'} (-1)^{m-1} \frac{N'(N-N'+1)}{N(m+1)} \binom{N-N'-1}{m-1} \\
        & = \frac{N'}{N-N'} \sum_{m=1}^{N-N'} (-1)^{m-1} m\binom{N-N'+1}{m+1} = \frac{N'}{N-N'} \left[ x \cdot \frac{d}{dx} (1-x)^{N-N'+1} - (1-x)^{N-N'+1} + 1 \right]_{x=1} \\
        & = \frac{N'}{N-N'} 
    \end{align*}
    and
    \begin{align*}
        \lambda^\star_{N',N'-1} & = N \sum_{m=1}^{N-N'} Q(N-1,m,N') \sum_{\ell=1}^{N-N'+1} (-1)^{\ell+m-1} \binom{\ell+m}{m}  Q(N-1,\ell,N'-1) \\
        & = N \sum_{m=1}^{N-N'} (-1)^{m-1} Q(N-1,m,N') \cdot \frac{N'-1}{N} \underbrace{\sum_{\ell=1}^{N-N'+1} (-1)^{\ell} \binom{\ell+m}{m} \binom{N-N'}{\ell-1}}_{=\begin{cases}0 & \text{if } m < N-N' \\ (-1)^{N-N'+1} & \text{if } m = N-N' \end{cases}} \\
        & = N \cdot Q(N-1,N-N',N') \cdot \frac{N'-1}{N} = N \cdot \frac{N'}{N} \cdot \frac{N'-1}{N} = \frac{N'(N'-1)}{N}
    \end{align*}
    where the simplification of the second line follows by plugging $a=N-N'$, $b=-m-1$, $c=N-N'+1$ into \cref{lemma:chu-vandermonde}, making the change of index $\ell = N-N'-i+1$ and using $\binom{-m-1}{\ell} = (-1)^\ell \binom{\ell+m}{\ell}$; the summation becomes $\binom{N-N'-m-1}{N-N'+1}$, which is $0$ for $m=1,\dots,N-N'-1$ and $\binom{-1}{N-N'+1} = (-1)^{N-N'+1}$ when $m=N-N'$.
    
\end{enumerate}

\subsection{Proof of~\cref{proposition:self-dual-family-H-matrix}}

As there is a one-to-one correspondence between the values of $Q(N-1,\cdot,\cdot)$ and $H$, it suffices to check that \eqref{eqn:H-matrix-self-dual-family} has the $Q$-values specified in \eqref{eqn:Q-formula-self-dual-family} by direct computation.
For $j = N'+1, \dots, N-1$, the entries $h_{k,j}$ of $H$ in \eqref{eqn:H-matrix-self-dual-family} agree with those of $H_{\text{Dual-OHM}}(N-1)$.
Because $Q(N-1,k,j)$'s with $j\ge N'+1$ only involve the entries within those last $N-N'-1$ columns, their values agree with those from Dual-OHM, which are precisely $\frac{1}{k+1}\binom{N-j-1}{k-1}$ as in \eqref{eqn:Q-formula-self-dual-family}.

Next, for $j=1,\dots,N'-1$ and $\ell=j+1,\dots,N$, we have
\begin{align}
\label{eqn:h-vertical-sum-self-dual-family}
    \sum_{i=j}^{\ell-1} h_{i,j} = \begin{cases}
        \frac{j}{\ell} & \text{if } \ell \le N' \\
        \frac{j}{N}    & \text{if } \ell \ge N'+1 .
    \end{cases}
\end{align}
In particular, $Q(N-1,1,j) = \sum_{i=j}^{N-1} h_{i,j} = \frac{j}{N}$ for $j=1,\dots,N'-1$, which agrees with \eqref{eqn:Q-formula-self-dual-family}.
Also we have
\begin{align*}
    Q(N-1,1,N') = h_{N',N'} + \sum_{k=N'+1}^{N-1} h_{k,N'} = \frac{N'(N-N')}{N} + \left(\frac{1}{N} - \frac{1}{N-N'}\right) \sum_{k=N'+1}^{N-1} (N-k) = \frac{N'(N-N'+1)}{2N} ,
\end{align*}
again consistent with \eqref{eqn:Q-formula-self-dual-family}.

Now we use induction on $m$ to show that the values of $Q(N-1,m,j)$ agree with those from \eqref{eqn:Q-formula-self-dual-family} for $j=1,\dots,N-m$.
The base case $m=1$ is proved by the argument above.
Now assume that this holds for some $1\le m \le N-2$.
We use the second recursion from \cref{lemma:P_Q_identities} to compute $Q(N-1,m+1,j)$ for $j \le N'-1$:
\begin{align*}
    Q(N-1,m+1,j) & = \sum_{\ell=j+1}^{N-m} Q(N-1,m,\ell) \sum_{i=j}^{\ell-1} h_{i,j} \\
    & = \sum_{\ell=j+1}^{N'} \frac{j}{\ell} Q(N-1,m,\ell)  + \sum_{\ell=N'+1}^{N-m} \frac{j}{N} Q(N-1,m,\ell) 
\end{align*}
where the second line uses \eqref{eqn:h-vertical-sum-self-dual-family}. Then using the induction hypothesis, we obtain
\begin{align*}
    Q(N-1,m+1,j) & = \sum_{\ell=j+1}^{N'-1} \frac{j}{\ell} \frac{\ell}{N} \binom{N-\ell-1}{m-1} + \frac{j}{N'} \frac{N'(N-N'+1)}{N(m+1)} \binom{N-N'-1}{m-1} + \frac{j}{N} \sum_{\ell=N'+1}^{N-m} \frac{1}{m+1} \binom{N-\ell-1}{m-1} \\
    & = \frac{j}{N} \sum_{\ell=j+1}^{N'-1} \binom{N-\ell-1}{m-1} + \frac{j(N-N'+1)}{N(m+1)} \binom{N-N'-1}{m-1} + \frac{j}{N} \frac{1}{m+1} \sum_{\ell=N'+1}^{N-m} \binom{N-\ell-1}{m-1} \\
    & = \frac{j}{N} \left[ \binom{N-j-1}{m} - \binom{N-N'}{m} \right] + \frac{j(N-N'+1)}{N(m+1)} \binom{N-N'-1}{m-1} + \frac{j}{N} \frac{1}{m+1} \binom{N-N'-1}{m} \\
    & = \frac{j}{N} \binom{N-j-1}{m} - \frac{j}{N} \underbrace{\left[ \binom{N-N'}{m} - \frac{N-N'+1}{m+1} \binom{N-N'-1}{m-1} - \frac{1}{m+1} \binom{N-N'-1}{m} \right]}_{=0} \\
    & = \frac{j}{N} \binom{N-j-1}{m} ,
\end{align*}
where the third line uses \cref{lemma:simple-combinatorial-summations}(a) and the second term in the fourth line becomes $0$ when we expand all binomial coefficients into explicit forms and simplify.

As we already know that $Q(N-1,m+1,j)$ agrees with \eqref{eqn:Q-formula-self-dual-family} for $j>N'$, the final step needed to complete the induction step is to show that \eqref{eqn:Q-formula-self-dual-family} holds true for $Q(N-1,m+1,N')$.
For this, we again use \cref{lemma:P_Q_identities}, but switch the order of summations:
\begin{align*}
    Q(N-1,m+1,N') & = \sum_{\ell=N'+1}^{N-m} Q(N-1,m,\ell) \sum_{i=N'}^{\ell-1} h_{i,N'} \\
    & = \sum_{i=N'}^{N-m-1} h_{i,N'} \sum_{\ell=i+1}^{N-m} Q(N-1,m,\ell) \\
    & = h_{N',N'} \sum_{\ell=N'+1}^{N-m} Q(N-1,m,\ell) - \sum_{i=N'+1}^{N-m-1} (N-i) \left( \frac{1}{N-N'} - \frac{1}{N} \right) \sum_{\ell=i+1}^{N-m} Q(N-1,m,\ell) \\
    & = \frac{N'(N-N')}{N} \sum_{\ell=N'+1}^{N-m} \frac{1}{m+1} \binom{N-\ell-1}{m-1} - \sum_{i=N'+1}^{N-m-1} (N-i) \left( \frac{1}{N-N'} - \frac{1}{N} \right) \sum_{\ell=i+1}^{N-m} \frac{1}{m+1} \binom{N-\ell-1}{m-1} \\
    & = \frac{N'(N-N')}{N(m+1)} \binom{N-N'-1}{m} - \frac{N'}{N(N-N')(m+1)} \sum_{i=N'+1}^{N-m-1} (N-i) \binom{N-i-1}{m}
\end{align*}
where in the last line, we use \cref{lemma:simple-combinatorial-summations}(a) to simplify the summations over $\ell$.
Then we note that $\sum_{i=N'+1}^{N-m-1} (N-i) \binom{N-i-1}{m} = \sum_{i'=m}^{N-N'-2} (i'+1) \binom{i'}{m}$ (by index change $i'=N-i-1$), which can be simplified to $(m+1)\binom{N-N'}{m+2}$ using \cref{lemma:simple-combinatorial-summations}(c).
Plugging this back into the above expression we obtain
\begin{align*}
    Q(N-1,m+1,N') = \frac{N'(N-N')}{N(m+1)} \binom{N-N'-1}{m} - \frac{N'}{N(N-N')} \binom{N-N'}{m+2} = \frac{N'(N-N'+1)}{N(m+2)} \binom{N-N'-1}{m}
\end{align*}
where the last identity can be verified by expanding the binomial coefficients.
This completes the induction, and concludes the proof that \eqref{eqn:H-matrix-self-dual-family} is the correct $H$ with the desired sparsity pattern.

\subsection{Proof of \cref{proposition:second-mixture-family-Q-functions}}

Our proposed Q-function profile
\begin{align*}
     Q(N-1,k,j) = \begin{cases}
        \frac{1}{k+1} \binom{N-j-1}{k-1} & \text{for } j=1,\dots,N'-1 \\
        \frac{j-N'+1}{N-N'+1} \binom{N-j-1}{k-1} & \text{for } j=N',\dots,N-1 
    \end{cases}
\end{align*}
clearly satisfies \eqref{eqn:new-algorithm-search-ckj} with
\begin{align*}
    c_{k,j} = \begin{cases}
        \frac{N-j+1}{k+1} \binom{N-j-1}{k-1} & \text{for } j=1,\dots,N'-1 \\
         \binom{N-j-1}{k-1} & \text{for } j=N',\dots,N-1 .
    \end{cases}
\end{align*}
This implies that \eqref{eqn:sparsity-pattern-Dual-OHM} holds for $j=1,\dots,N'-1$ (by \cref{proposition:Dual-OHM-linear-equation}) and \eqref{eqn:sparsity-pattern-OHM} holds for $j=N',\dots,N-1$ (by \cref{proposition:OHM-linear-equation}).
Therefore, as before (with the procedure of Section~\ref{subsubsection:re-deriving-OHM} determining Q-profile), it only remains to check: \textbf{(i)} H-invariance condition \eqref{eqn:H-invariance-Q-sum} and \textbf{(ii)} positivity of the remaining H-certificates.

\begin{enumerate}
    \item[\textbf{(i)}] \textbf{H-invariance.} For each $m=1,\dots,N-1$,
    \begin{align*}
        \sum_{j=1}^{N-m} Q(N-1,m,j) & = \sum_{j=1}^{N'-1} \frac{1}{m+1} \binom{N-j-1}{m-1} + \sum_{j=N'}^{N-m} \frac{j-N'+1}{N-N'+1} \binom{N-j-1}{m-1} \\
        & = \frac{1}{m+1} \sum_{i=N-N'}^{N-2} \binom{i}{m-1} + \frac{1}{N-N'+1} \sum_{i=1}^{N-N'-m+1} i \binom{N-N'-i}{m-1} \\
        & = \frac{1}{m+1} \left[ \binom{N-1}{m} - \binom{N-N'}{m} \right] + \frac{1}{N-N'+1} \binom{N-N'+1}{m+1}
    \end{align*}
    where for the last equality, the first summation is simplified using \cref{lemma:simple-combinatorial-summations}(a) and the second summation uses \cref{lemma:simple-combinatorial-summations}(b).
    Finally, because $\frac{1}{m+1} \binom{N-1}{m} = \frac{1}{N} \binom{N}{m+1}$ and $\frac{1}{m+1} \binom{N-N'}{m} = \frac{1}{N-N'+1} \binom{N-N'+1}{m+1}$ we conclude that $\sum_{j=1}^{N-m} Q(N-1,m,j) = \frac{1}{N} \binom{N}{m+1}$ as desired.

    \item[\textbf{(ii)}] \textbf{Positivity of remaining $\lambda^\star$.} Because the values of $Q(N-1,k,j)$ for $j=1,\dots,N'-1$ agree with those of Dual-OHM, the values of $\lambda^\star_{N,j} = N \sum_{\ell=1}^{N-j} (-1)^{\ell-1} Q(N-1,\ell,j)$ also agree with those of Dual-OHM, and hence are positive.
    Next, $\lambda^\star_{N,N-1} = N\cdot Q(N-1,1,N-1) > 0$. 
    Finally, for $j=N',\dots,N-2$, we have
    \begin{align*}
        \lambda^\star_{j+1,j} & = N \sum_{m=1}^{N-j-1} (-1)^{m-1} Q(N-1,m,j+1) \sum_{\ell=1}^{N-j} (-1)^{\ell} \binom{\ell+m}{m}  Q(N-1,\ell,j) \\
        & = N \sum_{m=1}^{N-j-1} (-1)^{m-1} Q(N-1,m,j+1) \frac{j-N'+1}{N-N'+1} \underbrace{\sum_{\ell=1}^{N-j} (-1)^{\ell} \binom{\ell+m}{m} \binom{N-j-1}{\ell-1}}_{=\begin{cases}0 & \text{if } m < N-j-1 \\ (-1)^{N-j} & \text{if } m = N-j-1 \end{cases}} \\
        & = N \cdot Q(N-1,N-j-1,j+1) \cdot \frac{j-N'+1}{N-N'+1} = \frac{N(j-N'+2)(j-N'+1)}{(N-N'+1)^2} \\
        & > 0 
    \end{align*}
    where the simplification of the second line follows by plugging $a=N-j-1$, $b=-m-1$, $c=N-j$ into \cref{lemma:chu-vandermonde}, making the change of index $\ell = N-j-i$ and using $\binom{-m-1}{\ell} = (-1)^\ell \binom{\ell+m}{\ell}$; the summation becomes $\binom{N-j-m-2}{N-j}$, which is $0$ for $m=1,\dots,N-j-2$ and $\binom{-1}{N-j} = (-1)^{N-j}$ when $m=N-j-1$.
\end{enumerate}

\subsection{Proof of \cref{proposition:second-mixture-family-H-matrix}}

For $j = N',\dots,N-1$, $Q(N-1,k,j)$ depends only on the lower right block of $H$, which equals $H_{\text{OHM}}(N-N')$. 
The rule determining the terms within $Q(N-1,k,j)$ for \eqref{eqn:H-matrix-second-mixture-family} is the same as determining the terms within $Q(N-N',k,j-N'+1)$ for $H_{\text{OHM}}(N-N')$ (note the index shift in the last argument).
Therefore, we have
\begin{align*}
    Q(N-1,k,j; H) & = Q(N-N',k,j-N'+1; H_{\text{OHM}}(N-N')) \\
    & = \frac{j-N'+1}{N-N'+1} \binom{(N-N'+1)-(j-N'+1)-1}{k-1} \\
    & = \frac{j-N'+1}{N-N'+1} \binom{N-j-1}{k-1} 
\end{align*}
which agrees with \eqref{eqn:Q-formula-second-mixture-family} for $j=N',\dots,N-1$.

For $j=1,\dots,N'-1$, we first see that
\begin{align*}
    Q(N-1,1,j) & = h_{j,j} + \sum_{k=j+1}^{N'-1} h_{k,j} + \sum_{k=N'}^{N-1} h_{k,j} \\
    & = \frac{N-j}{N-j+1} - \sum_{k=j+1}^{N'-1} \frac{N-k}{(N-j)(N-j+1)} - \sum_{k=N'}^{N-1} \frac{N-N'+1}{2(N-j)(N-j+1)} \\
    & = \frac{N-j}{N-j+1} - \frac{1}{(N-j)(N-j+1)} \frac{(N'-j-1)(2N-N'-j)}{2} - \frac{(N-N'+1)(N-N')}{2(N-j)(N-j+1)} \\
    & = \frac{1}{2(N-j)(N-j+1)} \underbrace{\left[ 2(N-j)^2 - (N'-j-1)(2N-N'-j) - (N-N'+1)(N-N') \right]}_{=(N-j)(N-j+1)} \\
    & = \frac{1}{2} .
\end{align*}
Next, we use induction on $m=1,\dots,N-1$ to show that the values of $Q(N-1,m,j)$ agree with those from \eqref{eqn:Q-formula-second-mixture-family} for any $j=1,\dots,N-m$. 
(In fact, we only need to prove it for $j\le N'-1$.)
The case $m=1$ is proved by the above argument.
Suppose the statement holds for some $m=1,\dots,N-2$ and fix any $j\le \min\{N'-1, N-m-1\}$.
We use the recursion of $Q$-functions in \cref{lemma:P_Q_identities} with the order of summations switched:
\begin{align}
\label{eqn:second-mixture-family-Q-identity-order-switched}
    Q(N-1,m+1,j) & = \sum_{\ell=j+1}^{N-m} Q(N-1,m,\ell) \sum_{i=j}^{\ell-1} h_{i,j} = \sum_{i=j}^{N-m-1} h_{i,j} \sum_{\ell=i+1}^{N-m} Q(N-1,m,\ell) .
\end{align}
Given the induction hypothesis, when $i\ge N'-1$, the inner summation can be simplified as
\begin{align}
    \sum_{\ell=i+1}^{N-m} Q(N-1,m,\ell) & = \sum_{\ell=i+1}^{N-m} \frac{\ell-N'+1}{N-N'+1} \binom{N-\ell-1}{m-1} \nonumber \\
    & = \frac{1}{N-N'+1} \sum_{j=i-N'+2}^{N-N'-m+1} j\binom{N-N'-j}{m-1} \nonumber \\
    & = \frac{1}{N-N'+1} \left( (i-N'+1)\binom{N-i-1}{m} + \binom{N-i}{m+1} \right) 
    \label{eqn:second-mixture-family-Q-sum-large-indices}
\end{align}
where the second line follows from the change of index $j=\ell-N'+1$ and the third line uses \cref{lemma:simple-combinatorial-summations}(b) with $p=N-N'$, $q=m-1$ and $s=i-N'+1$.
When $i < N'-1$, we have
\begin{align}
    \sum_{\ell=i+1}^{N-m} Q(N-1,m,\ell) & = \sum_{\ell=i+1}^{N'-1} \frac{1}{m+1} \binom{N-\ell-1}{m-1} + \sum_{\ell=N'}^{N-m} \frac{\ell-N'+1}{N-N'+1} \binom{N-\ell-1}{m-1} \nonumber \\
    & = \frac{1}{m+1} \sum_{j=N-N'}^{N-i-2} \binom{j}{m-1} + \frac{1}{N-N'+1} \sum_{j=1}^{N-N'-m+1} j\binom{N-N'-j}{m-1} \nonumber \\
    & = \frac{1}{m+1} \left[ \binom{N-i-1}{m} - \binom{N-N'}{m} \right] + \frac{1}{N-N'+1} \binom{N-N'+1}{m+1} \label{eqn:second-mixture-family-Q-sum-small-indices} \\
    & = \frac{1}{m+1} \binom{N-i-1}{m} \nonumber
\end{align}
where the last identity holds because $\frac{1}{m+1} \binom{N-N'}{m} = \frac{1}{N-N'+1} \binom{N-N'+1}{m+1}$.
Note that \eqref{eqn:second-mixture-family-Q-sum-small-indices} yields the same result as \eqref{eqn:second-mixture-family-Q-sum-large-indices} for $i=N'-1$ so we can use \eqref{eqn:second-mixture-family-Q-sum-small-indices} even for $i=N'-1$.
Now plugging these identities into \eqref{eqn:second-mixture-family-Q-identity-order-switched} and using $h_{j,j} = \frac{N-j}{N-j+1}, h_{i,j} = -\frac{N-i}{(N-j)(N-j+1)}$ for $i=j+1,\dots,N'-1$ and $h_{i,j} = -\frac{N-N'+1}{2(N-j)(N-j+1)}$ for $i=N',\dots,N-1$, we obtain
\begin{align}
    Q(N-1,m+1,j) & = \frac{N-j}{N-j+1} \frac{1}{m+1} \binom{N-j-1}{m} - \sum_{i=j+1}^{N'-1} \frac{N-i}{(N-j)(N-j+1)} \frac{1}{m+1} \binom{N-i-1}{m} \nonumber \\
    & \quad - \sum_{i=N'}^{N-m-1} \frac{N-N'+1}{2(N-j)(N-j+1)} \frac{1}{N-N'+1} \left( (i-N'+1)\binom{N-i-1}{m} + \binom{N-i}{m+1} \right) \nonumber \\
    & \begin{aligned}
        & = \frac{N-j}{(N-j+1)(m+1)} \binom{N-j-1}{m} - \frac{1}{(N-j)(N-j+1)(m+1)} \underbrace{\sum_{i=j+1}^{N'-1} (N-i) \binom{N-i-1}{m}}_{(*)} \\
        & \quad - \frac{1}{2(N-j)(N-j+1)} \bigg( \underbrace{\sum_{i=N'}^{N-m-1} (i-N'+1)\binom{N-i-1}{m}}_{=\binom{N-N'+1}{m+2}} +  \underbrace{\sum_{i=N'}^{N-m-1} \binom{N-i}{m+1}}_{=\binom{N-N'+1}{m+2}} \bigg) 
    \end{aligned}
    \label{eqn:second-mixture-family-induction-key-step}
\end{align}
where the simplification of the last two terms respectively uses \cref{lemma:simple-combinatorial-summations}(b) and (a).
The summation $(*)$ can be simplified as follows, using the identity $(N-i)\binom{N-i-1}{m} = (m+1) \binom{N-i}{m+1}$:
\begin{align*}
   (*) = \sum_{i=j+1}^{N'-1} (m+1) \binom{N-i}{m+1} = (m+1) \sum_{\ell=N-N'+1}^{N-j-1} \binom{\ell}{m+1} = (m+1) \left[\binom{N-j}{m+2} - \binom{N-N'+1}{m+2} \right] 
\end{align*}
where we make the change of index $\ell=N-i$ and use \cref{lemma:simple-combinatorial-summations}(a).
Plugging this into \eqref{eqn:second-mixture-family-induction-key-step} and canceling out the terms, we obtain
\begin{align*}
    Q(N-1,m+1,j) & = \frac{N-j}{(N-j+1)(m+1)} \binom{N-j-1}{m} - \frac{1}{(N-j)(N-j+1)} \binom{N-j}{m+2} \\
    & = \frac{N-j}{(N-j+1)(m+1)} \binom{N-j-1}{m} - \frac{N-j-m-1}{(N-j+1)(m+1)(m+2)} \binom{N-j-1}{m} \\
    & = \frac{(m+2)(N-j) - (N-j-m-1)}{(N-j+1)(m+1)(m+2)} \binom{N-j-1}{m}  \\
    & = \frac{1}{m+2} \binom{N-j-1}{m} 
\end{align*}
which completes the induction, and thereby proves the correctness of \eqref{eqn:H-matrix-second-mixture-family}.

\end{document}